\documentclass[sn-mathphys-num]{sn-jnl}

\usepackage{graphicx}%
\usepackage{multirow}%
\usepackage{amsmath,amssymb,amsfonts}%
\usepackage{amsthm}%
\usepackage{mathrsfs}%
\usepackage[title]{appendix}%
\usepackage{xcolor}%
\usepackage{textcomp}%
\usepackage{manyfoot}%
\usepackage{booktabs}%
\usepackage{algorithm}%
\usepackage{algorithmicx}%
\usepackage{algpseudocode}%
\usepackage{listings}%

\usepackage{amssymb}
\usepackage{mathtools}
\usepackage{url}
\usepackage{autonum}
\usepackage{xspace}
\usepackage{xcolor}
\usepackage{cancel}
\usepackage{graphicx}
\usepackage{tikz} 
\usepackage{pgfplots} 
\usepackage{verbatim}
\usepackage{caption}
\usepackage{subcaption}

\usepackage{orcidlink}

\numberwithin{equation}{section}

\definecolor{mycolor1}{cmyk}{0.93,.5,0,.05}
\definecolor{mycolor2}{cmyk}{0,0.62,0.88,0.15}%
\definecolor{myred}{cmyk}{.25,1,1,0}
\definecolor{mydarkgray}{cmyk}{0,0,0,.50} 
\definecolor{mygray}{cmyk}{0,0,0,.30} 
\definecolor{mygreen}{cmyk}{  1,0,.60,0} 
\definecolor{myblue}{cmyk}{.80,.500,0,0} 
\definecolor{myviolet}{cmyk}{0.23,.99,0,0.31} 

\definecolor{color0}{rgb}{0.12156862745098,0.466666666666667,0.705882352941177}
\definecolor{color1}{rgb}{1,0.498039215686275,0.0549019607843137}
\definecolor{color2}{rgb}{0.172549019607843,0.627450980392157,0.172549019607843}
\definecolor{color3}{rgb}{0.83921568627451,0.152941176470588,0.156862745098039}
\definecolor{color4}{rgb}{0.580392156862745,0.403921568627451,0.741176470588235}
\definecolor{color5}{rgb}{0.549019607843137,0.337254901960784,0.294117647058824}
\definecolor{color6}{rgb}{0.890196078431372,0.466666666666667,0.76078431372549}

\newcommand{\Fenicsx}{\texttt{FEniCSx}}
\newcommand{\PETSc}{\texttt{PETSc}}

\newcommand{\mycode}{\url{https://doi.org/10.35097/1871}}
\definecolor{darkgreen}{rgb}{0,0.5,0}

\newcommand{\sol}{u}

\newcommand{\solinit}{\sol_0}
\newcommand{\soltinit}{v_0}

\newcommand{\solbeta}{\sol^\beta}
\newcommand{\solzero}{\sol^{\beta=0}}
\newcommand{\soln}[1]{\widehat{\sol}^{#1}}

\newcommand{\solh}{\sol_h}
\newcommand{\solhn}[1]{\sol_h^{#1}}

\newcommand{\solhnbeta}[1]{\sol_{h,\beta}^{#1}}
\newcommand{\solhnzero}[1]{\sol_{h,\beta=0}^{#1}}

\newcommand{\solhbar}{\bar{\sol}_h}
\newcommand{\solhnbar}[1]{\bar{\sol}_h^{#1}}

\newcommand{\solhbeta}{\solh^\beta}
\newcommand{\solhzero}{\solh^{\beta=0}}

\newcommand{\errhzero}{\errh^{\beta=0}}
\newcommand{\resolvent}[1]{\mathcal{R}^{#1}}

\newcommand{\solhnt}[1]{v_h^{#1}}
\newcommand{\spacepar}{c_{\mathrm{sp}}}
\newcommand{\timepar}{c_{\mathrm{time}}}

\newcommand{\err}{e}

\newcommand{\errh}{\err_h}
\newcommand{\errhn}[1]{\errh^{#1}}

\newcommand{\errhnzero}[1]{\err_{h,\beta=0}^{#1}}

\newcommand{\errhbeta}{\errh^\beta}

\newcommand{\Deltah}{\Delta_h}
\def\Om{\Omega}
\def\Ltwo{L^2(\Omega)}
\def\Lthree{L^3(\Omega)}
\def\Lsix{L^6(\Omega)}
\def\Linf{L^\infty(\Omega)}
\def\Woneinf{W^{1, \infty}(\Omega)}
\def\Hone{H^1(\Omega)}

\newcommand{\Vh}{V_h}

\newcommand{\Triag}{\mathcal{T}_h}

\newcommand{\tn}[1]{t_{#1}}

\def\calU{\mathcal{U}}
\def\ds{\, \textup{d}s}

\newcommand{\dint}[1]{\,\textup{d} #1}

\newcommand{\finalth}{t_h^*}
\newcommand{\finalthbeta}{t_{h,\beta}^*}

\newcommand{\quasilowerbound}{\gamma}

\newcommand{\phih}{\varphi_h}

\newcommand{\pt}{\partial_t}
\newcommand{\ptau}{\partial_{\tau}}

\newcommand{\intt}{\int_0^t}

\newcommand{\norm}[1]{\| #1 \|}
\newcommand{\ip}[2]{ (#1 , #2 )_{L^2(\Omega)} }

\newcommand{\projLtwo}{\pi_h}
\newcommand{\projRitz}{R_h}

\newcommand{\Ih}{I_h}
\newcommand{\Id}{\textup{I}}

\newcommand{\normlinfLtwo}[3]{ \norm{#1}_{\ell^\infty(#2, \,  #3  , \, \Ltwo) } }

\newcommand{\Honethree}{{H_\diamondsuit^3(\Omega)}}
\newcommand{\Honefour}{{H_\diamondsuit^4(\Omega)}}

\newcommand{\R}{\mathbb{R}} 

\makeatletter
\renewcommand{\fnum@figure}{Figure \thefigure}
\makeatother

\newcommand{\midskip}{~\\[2mm]}

\newtheorem{lemma}{Lemma}[section]

\newtheorem{mytheorem}{Theorem}[section]
\newtheorem{myproposition}{Proposition}[section]

\usepackage{marginnote}
\newcommand{\mathreferences}[1]{%
	\marginnote{%
		\footnotesize%
		\normalfont
		\color{gray}%
		#1%
		~~~~~
	}%
}
\newcommand{\revisedd}[2]{{\color{blue} #2}\mathreferences{{\color{blue} R#1}}}

\begin{document}

\title[Robust fully discrete error bounds for the Kuznetsov equation]{Robust fully discrete error bounds for the Kuznetsov equation in the inviscid limit}


\author[1]{\fnm{Benjamin} \sur{D{\"o}rich*}\orcidlink{0000-0001-5840-2270}}\email{benjamin.doerich@kit.edu}

\author[2]{\fnm{Vanja} \sur{Nikoli\'c} \orcidlink{0009-0003-7610-2900} }\email{vanja.nikolic@ru.nl}

\affil[1]{\orgdiv{Institute for Applied and Numerical Mathematics}, \orgname{Karlsruhe Institute of Technology}, \orgaddress{\street{Englerstr. 2}, \city{Karlsruhe}, \postcode{76131}, \country{Germany}}}

\affil[2]{\orgdiv{Department of Mathematics}, \orgname{Radboud University}, \orgaddress{\street{Heyendaalseweg 135}, \city{Nijmegen}, \postcode{6525 AJ}, \country{The Netherlands}}}


\abstract{The Kuznetsov equation is a classical wave model of nonlinear acoustics that incorporates quadratic gradient nonlinearities. When its strong damping vanishes, it undergoes a singular behavior change, switching from a parabolic-like to a hyperbolic quasilinear evolution. In this work, we devise an analytical framework that allows establishing optimal error bounds for its finite element approximation as well as a semi-implicit fully discrete approximation that are robust with respect to the vanishing damping parameter. The core of the new arguments lies in deriving suitable energy estimates directly for the error equation where one can more easily exploit the polynomial structure of the nonlinearities and compensate inverse estimates with smallness conditions on the error. Numerical experiments are included to illustrate the theoretical results.
}

\keywords{asymptotic-preserving error estimates, full discretization, Kuznetsov equation, nonlinear acoustics}



\maketitle

\section{Introduction}
	We consider quasilinear wave equations of the following form:
	\begin{equation} \label{eq:Kuznetsov}
		( 1 +  \kappa \pt \sol ) \pt^2 \sol - c^2 \Delta \sol -\beta \Delta \pt \sol+ \ell \nabla \sol \cdot \nabla \pt \sol = f.
	\end{equation}
	This model arises in nonlinear acoustics under the name Kuznetsov equation~\cite{kuznetsov1971equations}.  It describes propagation of sound waves through fluids and can be understood as an approximation to the Navier--Stokes--Fourier system of governing equations of sound motion that is more accurate than Westervelt's equation~\cite{westervelt1963parametric}.
	 In the context of nonlinear acoustics,  $u=u(x,t)$ in \eqref{eq:Kuznetsov} is the acoustic velocity potential, $c>0$ denotes the speed of sound in the medium. 
	 {Furthermore, $\ell=2$ and $\kappa = \frac{1}{c^2}\frac{B}{A}$, where $\frac{B}{A}$ is the so-called parameter of nonlinearity of the medium, but we take here $\kappa$, $\ell \in \R$ so to slightly generalize the analysis setting.} The quadratic gradient nonlinearity (that is, $ (\frac12\ell |\nabla \sol|^2)_t$)  captures local (non-cumulative) nonlinear effects in sound propagation, which may be prominent, for example, close to the sound source; see the discussion in~\cite[Ch.\ 3.6]{HamBL24} for more details on modeling. \\
	\indent Equation \eqref{eq:Kuznetsov} is strongly damped when the parameter $\beta$, known in acoustics as the sound diffusivity, is positive and it exhibits parabolic-like behavior leading to an exponential decay of the energy of the solutions as time grows; see, e.g.,~\cite{kaltenbacher2012analysis, Wilke} for its global well-posedness analysis in this parameter regime. In the case $\beta=0$, however, smooth solutions are only expected to exist locally in time after which a gradient blow-up is expected; see the  analysis in~\cite{dekkers2017cauchy, kaltenbacher2022parabolic}, and numerical experiments conducted in~\cite{walsh2007finite}.  In practice, sound diffusivity is small and it may become negligible in certain (inviscid) propagation media. 
	{
	When the strong damping vanishes, the equation undergoes a singular behavior change, switching from a parabolic-like to a hyperbolic quasilinear evolution.
	}
	Investigation of the singular inviscid limit of a Dirichlet boundary-value problem for \eqref{eq:Kuznetsov} has been conducted in \cite{kaltenbacher2022parabolic}.  However, the questions of stability and asymptotic behavior of \emph{approximate} solutions of \eqref{eq:Kuznetsov} as $\beta \rightarrow 0^+$ are open in the field. \\[1mm]
\textbf{{Main aims}}. The first aim of the present work is to establish $\beta$-robust error bounds for the finite element and full discretizations of \eqref{eq:Kuznetsov}, {where by robust we mean that all derived bounds are uniform with respect to $\beta$}. As it turns out, the robust discretization is possible when using (at least) quadratic finite elements. Secondly, we determine the behavior of the (semi-)discrete solutions as $\beta \rightarrow 0^+$, and the conditions under which it asymptotically preserves the order of convergence of the exact solution established in~\cite{kaltenbacher2022parabolic}.
In addition, we determine how one has to couple the spatial discretization parameter, time step size, and the damping parameter to allow also linear finite elements in space.
{
Let us emphasize that for the time-integration we only consider a semi-implicit Euler method, sparing several technicalities  coming, for example, from using a higher-order method, such as an adaption of the second-order IMEX scheme~\cite{HocL21}. Further, we do not consider explicit schemes as their CFL condition is observed in experiments, whereas our {CFL-type condition} appears to be (at least in this strong form) an artifact of the analysis.
}
\\[1mm]
\textbf{{Novelty and related results}}. 	To the best of our knowledge, this is the first work dealing with the robust numerical analysis of the semi-discrete Kuznetsov equation, and the first work analyzing a fully discrete scheme for it. For the strongly damped Kuznetsov equation (with $\beta>0$ fixed), where one can exploit the parabolic-like evolution, \emph{a priori} analysis of a mixed-approximation has been conducted in~\cite{meliani2022mixed}
 and an \emph{a priori} analysis of a discontinuous Galerkin coupling for a nonlinear elasto-acoustic problem based on this model has been performed in~\cite{muhr2023discontinuous}.\\
\indent In contrast, quasilinear equations of Westervelt type given by
		\begin{equation} \label{eq:Westervelt}
	( 1 +  \kappa_1  \sol + \kappa_2 \pt \sol) \pt^2 \sol - c^2 \Delta \sol -\beta \Delta \pt \sol+ \kappa_1 ( \pt \sol)^2 = f, \quad \kappa_1, \kappa_2 \in \R
\end{equation}
 are  by now much better understood from the point of view of the numerical analysis as they do not involve quadratic gradient nonlinearities.  Again here, the cases $\beta=0$ and $\beta>0$ are qualitatively different. Concerning the spatial discretization,
 the results in \cite{hochbruck2022error} yield optimal order of convergence of space discrete solutions in the energy norm for $\beta=\kappa_2=0$.  A $\beta$-uniform analysis of a mixed approximation of \eqref{eq:Westervelt} with $\kappa_1=0$ has been conducted in \cite{meliani2022mixed}. Fully discrete	schemes for \eqref{eq:Westervelt} with $\beta=\kappa_2=0$ have been analyzed in~\cite{Doe23,Mai22}.  We also point out the works~\cite{CraS86,Kan90,Kob75,Tak76}, where existence of solutions to undamped quasilinear and nonlinear evolution equations of this type is established, and one can find approximation rates of the implicit and semi-implicit Euler methods.  Within an (extended) Kato framework, 
	optimal order for these methods has been determined in \cite{HocP17}
	and rigorous error bounds for the time discretization by higher-order Runge-Kutta methods are derived in \cite{HocSP18,KovL18}. \\
	\indent For the strongly damped Westervelt equation with $\beta>0$ fixed and $\kappa_2=0$,  optimal order of convergence of
	continuous Galerkin methods has been established in~\cite{nikolic2019priori}. Recently also Westervelt's equation with time-fractional dissipation instead of $-\beta \Delta \pt \sol$ has received attention. A time-stepping method for  such a model  has been analyzed in~\cite{baker2022numerical}, and a $\beta$-robust finite element analysis for both time-fractional and strongly damped Westervelt's equation has been performed in~\cite{nikolic2023asymptotic}, together with establishing the vanishing $\beta$ convergence rates of the approximate solution.   \\
	  \indent We mention also that other quasilinear wave models have been rigorously investigated in the literature. In~\cite{GauLMRS19}, trigonometric integrators have been analyzed for nonlinear wave equations in the form of
	\[
	\pt^2 u = \partial_x^2 u - u + \kappa a(u) \partial_x^2 u + \kappa g(u, \partial_x u)
	\]
	in one space dimension under periodic boundary conditions. 
	Analysis of different time stepping schemes for nonlinear hyperbolic problems can also be found in~\cite{Bal86,Bal88,Ewi89, shao2022discontinuous},
	and two-step methods are considered in~\cite{BalD89}.
	For a class of linearly implicit single-step schemes as well as a linearly and a fully implicit two-step scheme, optimal error bounds are derived in \cite{makridakis1993finite}. \\
	\indent Compared to the available works, the main challenge here comes from treating the nonlinear term $\ell \nabla \sol \cdot \nabla \pt \sol $ after discretization, in combination with having to guarantee that the discrete version preserves the non-degeneracy condition:
	\[
	 1 +  \kappa \pt \sol  \geq \tilde{\gamma} >0
	\]
	and that the bounds are uniform with respect to $\beta$.  To guarantee $\beta$ uniformity, we have to work also with the time-differentiated version of the (semi-)discrete problems, which introduces (a discrete version of) the term $\ell \nabla \sol \cdot \nabla \pt^2 \sol $. A fixed-point argument along the lines of existing results on the damped Kuznetsov equation in~\cite{meliani2022mixed, muhr2023discontinuous} would then not allow us to match the order of convergence in the fixed-point iterates. We will instead first show the existence of a unique approximate solution on a discretization-dependent time interval and then derive uniform bounds to extend it beyond it, in the spirit of~\cite{hochbruck2022error}. To tackle the quadratic gradient nonlinearity, the main idea here is to devise energy estimates directly for the error equation where one can more easily exploit the polynomial structure of the nonlinearities  so as to mimic the following identity from the continuous setting: 
	\[
			\ip{\nabla \sol \cdot \nabla \pt^2 \sol   }{\pt^2 \sol}
=
- \frac12 \ip{\Delta \sol \, \pt^2 \sol   }{\pt^2 \sol}
	\]
and then compensate inverse estimates with smallness conditions on the error. We refer to Proposition~\ref{prop:fe_FirstEstimate} for details.
\subsection*{Organization of the exposition} The rest of the manuscript is organized as follows.  We first state our main results in Section~\ref{sec:Main_results}, both for the spatially semi-discrete and the fully discrete problem,
and 
present numerical experiments which corroborate the theory.
In Section~\ref{sec:FE_analysis}, we conduct the finite element analysis and establish the $\beta$-robust optimal error bounds in the energy norm when using quadratic or higher-order elements as well as the $\beta$-limiting behavior of the finite element solution. 
Section~\ref{sec:Semi_implicit} is dedicated to the stability and error analysis of a fully discrete problem based on a semi-implicit Euler method for the time discretization.
An extension to linear finite elements is given in Section~\ref{sec:non_robust} with non-robust estimates with respect to $\beta$.

\subsection*{Notation} Below we use  $x\lesssim y$ to denote $x \leq C y$, where $C>0$ is a generic constant that does not depend on the discretization parameters nor on the damping coefficient $\beta$, but may depend on the exact solution and the final time $T$. We use $\ip{\cdot}{\cdot}$ to denote the scalar product in $\Ltwo$. We omit the temporal domain when writing norms; 
for example, $\|\cdot\|_{L^p(L^q(\Omega))}$ denotes the norm on $L^p(0,T; L^q(\Omega))$. 
 We use $\|\cdot\|_{L^p_t(L^q(\Omega))}$ to denote the norm on $L^p(0,t; L^q(\Omega))$ for some $t \in (0,T)$. 
		
		
\section{Statements of the main results}  \label{sec:Main_results}
In this section, we present the main results of this work. To this end, we first discuss the assumptions on the exact solution. As we are interested in the vanishing $\beta$ dynamics, we may assume that $\beta \in [0, \bar{\beta}]$ for some fixed $\bar{\beta}>0$. 
\subsection{Assumptions on the exact solution}
Throughout, we assume that the initial data and source term are sufficiently smooth and small and the final time $T>0$ short so that the initial boundary-value problem
		\begin{equation} \label{ibvp:Kuznetsov}
			\left\{\begin{aligned}
				&( 1 +  \kappa \pt \sol ) \pt^2 \sol - c^2 \Delta \sol -\beta \Delta \pt \sol+ \ell \nabla \sol \cdot \nabla \pt \sol = f \ &&\text{in} \  \Omega \times (0,T),\\
				&u \vert_{\partial \Omega}=0, \ && \\
				&(u, u_t)\vert_{t=0}= (\solinit, \soltinit),&&
			\end{aligned}	\right.
		\end{equation}
		has a unique solution in
		\begin{equation} \label{def_calU}
			\begin{aligned}
			\calU = &\, \begin{multlined}[t] L^\infty(0,T; W^{2, \infty}(\Om)) \cap  H^3(0,T; H^{k+1}(\Om) \cap W^{1,\infty}(\Om)\cap H_0^1(\Om)) \\
			 \cap
			  W^{3,\infty}(0,T ; H^2(\Omega)) 
			 \cap 
			  H^4(0,T ; \Ltwo) , 
			  \end{multlined}
\end{aligned}
		\end{equation}
%
for $k \geq  2$ (or $k \geq 1$ in Section~\ref{sec:non_robust}),
		with $\beta$-uniform bounds:
		\begin{equation} \label{assumptions_exact_sol}
			\|u\|_{\calU} \leq C, \qquad 1+ \kappa \pt \sol \geq \tilde{\gamma} >0 \quad \text{for all } (x,t) \in \bar{\Omega} \times [0,T].
		\end{equation}
Note that the $\tilde{\gamma}$-bound in \eqref{assumptions_exact_sol} guarantees that the leading term in the Kuznetsov equation does not degenerate. The $\beta$-uniform well-posedness analysis of \eqref{ibvp:Kuznetsov} with $f=0$ can be found in \cite{kaltenbacher2022parabolic}. Compared to the results of~\cite{kaltenbacher2022parabolic}, we require more smoothness from the solution. More precisely, assuming the domain $\Omega$ is sufficiently smooth and final time $T$ sufficiently short, the results of~\cite[Thm.\ 6.2]{kaltenbacher2022parabolic} provide uniform well-posedness in the following space:
\begin{equation}
	\begin{aligned}
	\begin{multlined}[t] 
	H^3(0,T;H_0^1(\Omega))\cap W^{2,\infty}(0,T;H_0^1(\Omega)\cap H^2(\Omega)) 
	\cap W^{1,\infty}(0,T;\Honethree) \\
	\cap L^\infty(0,T;\Honefour),
	\end{multlined} 
	\end{aligned}
\end{equation}
where we have denoted
\begin{equation} \label{sobolev_withtraces}
	\begin{aligned}
		\Honethree=&\, \left\{u\in H^3(\Omega)\,:\,  u\vert_{\partial \Omega} = 0, \ \Delta u \vert_{\partial \Omega}= 0\right\},\quad \Honefour = H^4(\Omega) \cap \Honethree.
	\end{aligned}
\end{equation}
{The condition $u\vert_{\partial \Omega} = \Delta u \vert_{\partial \Omega}= 0$ arises due to the fact that the bi-Laplacian operator is used in the testing procedure of the well-posedness analysis to obtain bounds of a suitably defined energy of a linearized problem, which is later combined with employing Banach's fixed-point theorem; we refer to~\cite[Proposition 6.1 and Theorem 6.1]{kaltenbacher2022parabolic} for details.} We expect that the techniques in~\cite{kaltenbacher2022parabolic} can be extended in a relatively straightforward manner to rigorously prove higher-order uniform well-posedness in $\calU$ for sufficiently smooth and small data, as assumed in the present numerical analysis. We also note that the higher regularity for the strongly damped Kuznetsov equation (i.e., with $\beta>0$) follows by the results of~\cite{kawashima1992global}.  \\
				\indent	Our main contributions concern robust error bounds for a finite element discretization of \eqref{ibvp:Kuznetsov} and a fully discrete scheme, as well as establishing asymptotic-preserving behavior of respective solutions as $\beta$ vanishes; we illustrate them in Figure~\ref{fig:diagram}. 
					
		\begin{figure}[h]
			\definecolor{mblue}{rgb}{0,0,0}
\definecolor{morange}{rgb}{0,0,0}%
%
\begin{tikzpicture}[ font = \normalsize, align = center]
\pgfmathsetmacro{\shift}{.6ex}

\node[align=center]  (1) at (-4.5,1.5) {$\solhnbeta{n}$};
\node  (2) at (0,1.5) {$\solhbeta$};
\node  (3) at (4.5,1.5) {$\solbeta$};

\node  (12) at (-4.5,-1.5) {$\solhnzero{n}$}; 
\node  (22) at (0,-1.5) {$\solhzero$}; 
\node  (32) at (4.5,-1.5) {$\solzero$}; 

Theorem~\ref{thm:ErrorSemiImplicit}


\draw  [-stealth] (2) -- node[below] {$h \to 0$ , Theorem~\ref{thm:kuznetsov_space_main}} (3);


\draw  [-stealth] (22) -- node[above] {$h \to 0$ , Theorem~\ref{thm:kuznetsov_space_main}} (32);

\draw  [-stealth] (1) -- node[left] {$\beta \to 0$ \\ Theorem~\ref{thm:FullyDiscrete_beta_limit}} (12);

\draw  [-stealth] (2) -- node[left] {$\beta \to 0$ \\ Theorem~\ref{thm:FE_beta_limit}} (22);

\draw  [-stealth] (3) -- node[right] {$\beta \to 0$ \\ \cite{kaltenbacher2022parabolic}} (32);

\draw [>=stealth, black, ->] (1)  to [out=35,in=145] 
node [below, sloped]
(TextNode1) {~\\[2mm]$\tau,h \to 0$, Theorem~\ref{thm:ErrorSemiImplicit}}   (3);

\draw [>=stealth, black, ->] (12)  to [out=-35,in=-145] 
node [above, sloped]
(TextNode1) {$\tau,h \to 0$, Theorem~\ref{thm:ErrorSemiImplicit}\\}   (32);



%

%
\end{tikzpicture}

%
%
%
%
%
%
%
%
%
%
			\caption{Diagram representing the main contributions of this work} \label{fig:diagram}
		\end{figure}

		
\subsection{Main results for the finite element discretization}	
In the present work, we employ continuous finite elements and consider a quasi-uniform triangulation  $\Triag$ and the space
	\begin{align}
	\Vh \coloneqq \{ \phih \in C(\Om) \mid  \phih |_K \in \mathcal{P}_k(K)   \text{ for all } K \in \Triag  \}
	\end{align}
	of piecewise polynomials of degree $k$. To conduct the error analysis bellow in a $\beta$-uniform manner, we assume that $k \geq 2$.
	The case $k=1$ with non-uniform bounds is treated separately in Section~\ref{sec:non_robust}.
	 We introduce the Ritz
	  projection
defined for $\varphi \in \Hone$
 via
\begin{align}
	\ip{\nabla \varphi}{ \nabla \phih} &= \ip{\nabla {\projRitz} \varphi}{ \nabla \phih}
\end{align}
for all $\phih \in \Vh$. Further, we make use of the nodal interpolation operator $\Ih \colon C(\Omega) \to \Vh$, and define the discrete Laplacian operator $\Deltah \colon \Vh \to \Vh$ 
for $\psi_h,\phih \in \Vh$
via the relation
\begin{equation}
\ip{\Deltah \psi_h}{\phih} = - \ip{\nabla \psi_h}{ \nabla \phih}.
\end{equation}
	
	\indent 
	With these preparations, we consider the spatially discrete Kuznetsov equation
	\begin{subequations} \label{eq:Kuznetsov_space_discr_full_eq}
	\begin{equation} \label{eq:Kuznetsov_space_discr}
		\begin{aligned}
	\begin{multlined}[t]	\ip{( 1 + \kappa \pt \solh ) \pt^2 \solh }{ \phih}
		- 
		\ip{c^2 \Deltah \solh  }{ \phih}	- 
		\ip{\beta \Deltah \pt \solh  }{ \phih}
	\\	+
		\ell \ip{ \nabla \solh \cdot \nabla \pt \solh }{\phih} 
		=
		\ip{f_h }{\phih} ,
		\end{multlined}
		\end{aligned}
	\end{equation}
	for all $\phih \in V_h$, supplemented by approximate initial data
	\begin{equation}\label{approx_data}
		(\solh, \pt \solh) \vert_{t=0}=(u_{0h}, u_{1h}).
	\end{equation}
	\end{subequations}
	Our first main results establishes \emph{a priori} error bounds for $\solh$ in the energy norm that are uniform with respect to the damping parameter $\beta$.
	\begin{mytheorem}[Robust finite element estimates] \label{thm:kuznetsov_space_main} 
	Let $k \geq 2$ and $\beta \in [0, \bar{\beta}]$ for some $\bar{\beta}>0$. Furthermore, assume that $f$, $f_h \in H^1(0,T; \Ltwo)$ are such that
	\begin{equation} \label{assumption_accuracy_fh}
	\begin{aligned}
	\|f-f_h\|_{H^1(\Ltwo)} \lesssim  h^k
	\end{aligned}
	\end{equation}
{where the hidden constant does not depend on $h$ or $\beta$}, 
	and that the approximate initial data are chosen as
	\begin{equation} \label{eq:init_choice_space}
	\solh(0)= \projRitz \solinit, \ \pt \solh (0)= \projRitz \soltinit,
	\end{equation}
where $\sol \in \calU$ is the solution of 
\eqref{eq:Kuznetsov} 
satisfying \eqref{assumptions_exact_sol}, and $\pt^2 \solh(0)$ is given by
\begin{equation} \label{third_ic}
\begin{aligned}
\begin{multlined}[t]	\ip{( 1 + \kappa \pt \solh(0) ) \pt^2 \solh(0) }{ \phih}
- 
\ip{c^2 \Deltah \solh(0)  }{ \phih}	\\- 
\ip{\beta \Deltah \pt \solh(0)  }{ \phih}+
\ell \ip{ \nabla \solh(0) \cdot \nabla \pt \solh(0) }{\phih} \\
=
\ip{f_h(0) }{\phih}
\end{multlined}
\end{aligned}
\end{equation}
for all $\phih \in V_h$. Then, there exists $h_0>0$ and  a constant $C>0$,  independent of $h$ and $\beta$, such that for all $h\leq h_0$, the following error bound holds:
	\begin{equation} \label{FE_final_est}
	\begin{aligned}
\begin{multlined}[t]	\norm{\pt^2 \sol(t) - \pt^2 \solh(t) }^2_{\Ltwo}
	+  
 	\norm{\nabla \pt \sol(t) - \nabla \pt \solh(t)}^2_{\Ltwo}  
 \\	\hspace*{5cm}+
 	\int_0^t \norm{\nabla \sol(s) - \nabla \solh(s)}^2_{L^6(\Omega)} \ds
	\leq C 
h^{2 k } \end{multlined}
\end{aligned}
	\end{equation}
	for all $t \in [0,T]$.
\end{mytheorem}

A non-robust variant of this result for $k=1$ is presented later in 
Theorem~\ref{thm:kuznetsov_space_main_beta_dep}. The reason for the exclusion of linear elements in our theory comes the inverse estimate which forces us to control the product $h^{-d/2 + \varepsilon} $ with the error in $H^1$ by a sufficiently small constant, such that for $d=3$ the error has to scale with $h^2$. In the case $d=2$, one could hope for some sharper bounds for linear elements including logarithmic term, however this is beyond the scope of the present work. 
The second main result confirms that, in the setting of Theorem~\ref{thm:kuznetsov_space_main}, the finite element solution preserves the asymptotic behavior as $\beta \rightarrow 0$ of the exact solution established in~\cite{kaltenbacher2022parabolic}.
\begin{mytheorem}[Asymptotic-preserving behavior in the inviscid limit]\label{thm:FE_beta_limit}
Under the assumptions of Theorem~\ref{thm:kuznetsov_space_main}, for $h \in (0, h_0]$, the family $\{\solhbeta\}_{\beta \in (0, \bar{\beta}]}$ of finite element solutions of \eqref{eq:Kuznetsov_space_discr_full_eq} converges in the energy norm to the finite element solution $\solhzero$ of the inviscid semi-discrete problem (i.e., with $\beta=0$) at a linear rate as $\beta \rightarrow 0$. In other words, 
\begin{equation}
	\| \pt \solhbeta - \pt \solhzero \|_{L^\infty(\Ltwo)} + \|\nabla (\solhbeta-\solhzero)\|_{L^\infty(\Ltwo)} \leq C \beta,
\end{equation}
where the constant $C>0$ is independent of  $\beta$ and $h$.
\end{mytheorem}

\subsection{Main results for a fully discrete semi-implicit Euler method}

Our next results concern a full discretization of \eqref{ibvp:Kuznetsov} based on a semi-implicit Euler method. 
To present it, we introduce the discrete derivative $\ptau$ as follows:
\begin{subequations} \label{eq:def_ptau}
\begin{align+}  
	\ptau a^n &= \frac{1}{\tau}  (a^n -a^{n-1} )  ,
	\, \,
	n \geq 1,
&	
	\ptau^{k+1} a^n &= \ptau  \ptau^k a^n,
	\, \,
	k \geq 0 \,,
\intertext{	and the notational conventions}
	\ptau a^0 &= a^0 ,
&
		\ptau^{n+j} a^{n} &=   	\ptau^{n} a^{n},
		\, \,
	j \geq 0 \,.
\end{align+}
\end{subequations}
Then for $1 \leq n \leq N$
with
 $(N+1) \tau \leq T$, we consider
\begin{equation} \label{eq:Euler}
\begin{aligned}
\ip{
	(1+\kappa \ptau \solhn{n}) \ptau^2 \solhn{n+1} 
}{\varphi_h} 
&\, 
-
\ip{ c^2 \Delta_h \solhn{n+1}
}{\varphi_h} 
  -
  \ip{ \beta \Delta_h  \ptau\solhn{n+1} 
  }{\varphi_h} 
\\
&\, +
\ip{ \ell \nabla \solhn{n} \cdot \nabla \ptau \solhn{n+1}}{\varphi_h} 
 =
 \ip{ f_h^{n+1}
}{\varphi_h}  ,
\end{aligned}
\end{equation}
for all $\varphi_h \in \Vh$.
We set the initial conditions as 
\begin{subequations} \label{eq:def_Euler_inits}
\begin{align+} 
		\solhn{0} &\coloneqq \projRitz \solinit, 
		\qquad
		\solhn{1} \coloneqq 
		\projRitz \bigl( \solinit
		+ \tau  
		\soltinit
		+	\frac{\tau^2}{2} w_0 \bigr) ,
		\label{eq:def_Euler_inits_u_1}
\intertext{using the approximation $w_0\approx \pt^2 \sol(\tn{0}) $ defined as}
		w_0 &=	
		\projRitz \Bigl( 	( 1 +  \kappa \soltinit )^{-1} \bigl(
		c^2 \Delta \solinit + \beta \Delta \soltinit - \ell \nabla \solinit \cdot \nabla \soltinit  + f
		\bigr)
		\Bigr) ,
\end{align+}
\end{subequations}
such that $\solhn{1}$
resembles a projected Taylor approximation to $\sol(\tn{1})$.
The quadratic gradient nonlinearity forces us to assume the following 
{CFL-type condition}:
\begin{equation} \label{eq:def_CFL_Euler}
	\tau \leq C h^{1+d/6 + 2\varepsilon}
\end{equation}
for some arbitrarily chosen $\varepsilon>0$. Without loss of generality, we assume $\varepsilon \in (0, \frac12 - \frac{d}{12})$,
such that $1+d/6 + 2\varepsilon < 2$. 

\begin{mytheorem}[Robust fully discrete error bounds] \label{thm:ErrorSemiImplicit}
Let $k \geq 2$ and $\beta \in [0, \bar{\beta}]$ for some $\bar{\beta}>0$,
and let the {CFL-type condition} \eqref{eq:def_CFL_Euler} hold. Furthermore, assume that $f, f_h \in C^1(0,T; \Ltwo)$ are such that
\begin{equation} \label{assumption_accuracy_fh_full}
	\begin{aligned}
		\|f-f_h\|_{C^1(\Ltwo)} \lesssim  h^k ,
	\end{aligned}
\end{equation}
{where the hidden constant does not depend on $h$, $\tau$, or $\beta$}, 
and the initial values are chosen as in \eqref{eq:def_Euler_inits}.
If $\solhn{n}$ is the solution of \eqref{eq:Euler},
and the solution $\sol \in \calU$ of \eqref{eq:Kuznetsov} satisfies 
\eqref{assumptions_exact_sol},
then for
for $h \leq h_0$ and $\tau \leq \tau_0$,
it holds
\begin{equation} \label{eq:thm_ErrorSemiImplicit_estimate}
\begin{aligned}
\norm{\pt^2 \sol(\tn{n}) - \ptau^2 \solhn{n}}^2_{\Ltwo}
&+  
\norm{\nabla \pt \sol(\tn{n}) - \nabla \ptau \solhn{n} }^2_{\Ltwo}  
\\
&+
\tau \sum_{j=1}^n
\norm{\nabla \sol(\tn{j}) - \nabla \solhn{j} }^2_{L^6(\Omega)}  
\leq C  \bigl( \tau + h^k \bigr)^2,
\end{aligned}
\end{equation}
for all $n=2, \ldots, N+1$,
where the constant $C$ is independent of $h$, $\tau$, and $\beta$.
\end{mytheorem}

A non-robust variant of this result for $k=1$ is presented later in 
Theorem~\ref{thm:ErrorSemiImplicit_beta}. As in the semi-discrete case, our fourth main
result confirms that, in the setting of Theorem~\ref{thm:ErrorSemiImplicit}, also the fully discrete solution preserves the asymptotic behavior as $\beta \rightarrow 0$ of the exact solution established in~\cite{kaltenbacher2022parabolic}.

\begin{mytheorem}[Asymptotic-preserving behavior in the inviscid limit]\label{thm:FullyDiscrete_beta_limit}
	Under the assumptions of Theorem~\ref{thm:ErrorSemiImplicit}, for $h \in (0, h_0]$ and $n \in \{2, \ldots, N+1\}$ fixed, the family $\{\solhnbeta{n} \}_{\beta \in (0, \bar{\beta}]}$ of finite element solutions of \eqref{eq:Euler}
	 converges in the discrete energy norm to the finite element solution $\solhnzero{n}$ of the inviscid fully discrete problem (i.e., with $\beta=0$) at a linear rate as $\beta \rightarrow 0$. In other words, 
	\begin{equation} \label{est:fully_discrete_beta}
		\| \ptau \solhnbeta{n} - \ptau \solhnzero{n} \|_{\Ltwo}
	 +
	 \|\nabla (\solhnbeta{n} - \solhnzero{n})\|_{\Ltwo} \leq C \beta,
	\end{equation}
	for all $n=1, \ldots, N+1$,
	where the constant $C>0$ is independent of $\beta$, $h$, and $\tau$.
\end{mytheorem}

\subsection{Numerical results} \label{sec:numerics}

In this section, we illustrate our theoretical findings with three numerical experiments. 
We first study the case of a smooth (\emph{a priori} known) exact solution $u$ to show the optimality of the derived convergence rates. 
In the second experiment, we verify the convergence with respect to the vanishing damping parameter $\beta$ for given data without a known solution.
Thirdly, we study a more realistic scenario of a traveling Gaussian pulse. In all experiments, we observe the optimality of our main results.

\vspace{1ex}

\noindent \textbf{Discretization.}
For the discretization in space with Lagrangian finite elements, we use the open-source Python tool \Fenicsx{},
(\url{https://fenicsproject.org/}); see
\cite{BarattaEtal2023} and \cite{AlnaesEtal2014}.
%
For a stable implementation, we introduce the 
auxiliary quantity for the discrete derivative 
\begin{align}
	\solhnt{n+1} = \ptau \solhn{n+1}, \qquad n\geq 0 .
\end{align}
We reformulate this as an update step
\begin{align} \label{eq:update_solhn}
	\solhn{n+1}	
	=
	&\solhn{n}	
	+ \tau \solhnt{n+1} 
\end{align}
once we have computed $\solhnt{n+1}$.
Note that by \eqref{eq:def_Euler_inits} it holds
\begin{equation}
	\solhnt{1} = 	\projRitz \bigl(	\soltinit
	+	\frac{\tau}{2} w_0 \bigr) .
\end{equation}
With \eqref{eq:update_solhn}, we eliminate $\solhn{n+1}$ in \eqref{eq:Euler}
and obtain the following relation for $n \geq 1$: 
\begin{align}
	&\, \begin{multlined}[t] \ip{(1+\kappa \solhnt{n}) \ptau \solhnt{n+1} }{\varphi_h}
		+
		\tau c^2 \ip{ \nabla \solhnt{n+1}	}{ \nabla\varphi_h}
		+	
		\beta \ip{ \nabla \solhnt{n+1}  }{ \nabla \varphi_h}
	\\
	+
		\ell	\ip{  \nabla \solhn{n} \cdot 
			\nabla \solhnt{n+1}  }{\varphi_h} \end{multlined}
	\\
	= 
	&\,	- 
	c^2	 \ip{  \nabla		\solhn{n}    }{\nabla \varphi_h}
	+
	\ip{f_h^{n+1}  }{\varphi_h},
\end{align}
which in turn yields the system
\begin{equation} 	\label{eq:Euler_assemble}
	\begin{aligned} 
		& \begin{multlined}[t] \ip{(1+\kappa \solhnt{n}) \solhnt{n+1}  }{\varphi_h}
		+
		c^2 \tau^2	\ip{ \nabla 
			\solhnt{n+1}	}{ \nabla\varphi_h}
		+	
		\tau \beta \ip{ \nabla \solhnt{n+1}  }{ \nabla \varphi_h}
		\\
		+
		\tau  \ell	\ip{  \nabla \solhn{n} \cdot 
			\nabla \solhnt{n+1}  }{\varphi_h}  \end{multlined}
		\\
		= \,
		&\ip{(1+\kappa \solhnt{n}) \solhnt{n}  }{\varphi_h}
		-
		\tau c^2	 \ip{  \nabla		\solhn{n}    }{\nabla \varphi_h}
		+
	 \tau 	\ip{f_h^{n+1}  }{\varphi_h}.
	\end{aligned}
\end{equation}
Since the mass and stiffness matrix change in each time step,
the routines in \Fenicsx{} assemble the mass and stiffness matrix 
and use the \PETSc{} linear algebra backend to
solve the linear system \eqref{eq:Euler_assemble}.
The codes to reproduce the results are available at 
\[
\text{\mycode}.
\]

\subsubsection{Smooth solution}

In the first example, we consider the domain
$\Omega = [0,1] \times [0,1]$
and choose initial data as
\begin{equation} \label{eq:numexp_inits_smooth} 
	\solinit(x) = 
	\spacepar \sin(\pi x_1) 	\sin(\pi x_2), 
	\quad
	\soltinit(x) = 
	\spacepar 
	\timepar
	\sin(\pi x_1) 	\sin(\pi x_2), \quad \spacepar, \timepar>0 
\end{equation}
with the parameters
\begin{equation}  \label{eq:numexp_parameters_smooth}
	\kappa = 0.7, 
	\qquad
	c^2 = 1,
	\qquad
	\ell = 2 ,
\end{equation}
and vary 
$\beta \geq 0$.
The forcing term $f$ is chosen such that the exact solution is given by
\begin{equation} \label{eq:sol_shock}
	\sol (x,t) = 	\spacepar e^{\timepar t} \sin(\pi x_1) 	\sin(\pi x_2).
\end{equation}
%
\indent In Figure~\ref{fig:space_conv_exa_1}, we present the computed error 
\begin{equation}
	\mathbf{E}(t) = \norm{\nabla \pt \sol(t) - \nabla \pt \solh (t) }_{\Ltwo}
\end{equation}
for the semi-discrete method 	\eqref{eq:Kuznetsov_space_discr_full_eq}
at time $t=0.8$, using a small time-step size $\tau =  1.5 \cdot 10^{-3} $. We perform experiments for the space and time discretization with elements of order $k=1,2,3$.\\
\indent We observe convergence of order $k$ until a plateau caused by the temporal discretization is reached. For smaller time-step sizes the plots look qualitatively similar with a lower plateau. 
Further, we observe that the plots for $\beta  = 0 , 10^{-3}, 10^{-2}$
have no visible difference, which is in alignment with Theorem~\ref{thm:kuznetsov_space_main}. 
Note that the case $k=1$ is not covered by Theorem~\ref{thm:kuznetsov_space_main}, but appears to work well in practice even for $\beta=0$.

\begin{figure}[h]
	\centering
	\begin{subfigure}{0.33\textwidth}
\begin{tikzpicture}

\begin{axis}[
legend cell align={left},
legend style={
  fill opacity=0.8,
  draw opacity=1,
  text opacity=1,
  at={(0.97,0.03)},
  anchor=south east,
  draw=white!80!black,
  column sep=7mm,
  /tikz/every odd column/.append style={column sep=0cm},
},
width=0.9\columnwidth,
height=1.1\columnwidth,
log basis x={10},
log basis y={10},
tick align=outside,
tick pos=left,
title={$\beta = 0 $},
x grid style={white!69.0196078431373!black},
xmin=0.004, xmax=0.45,
xmode=log,
xtick style={color=black},
xlabel style={font=\color{white!15!black}},
xlabel={mesh width $h$},
ylabel style={font=\color{white!15!black}},
ylabel={$\mathbf{E}(0.8)$},
y grid style={white!69.0196078431373!black},
ymin=0.7e-05, ymax=0.1,
ymode=log,
ytick style={color=black},
ytick ={1e-5,1e-4,1e-3,1e-2,1e-1},
legend columns=-1,
legend to name=Legendforall,
]
\addplot [semithick, mycolor1,mark=triangle, mark size=3, mark options={solid}]
table {%
0.353553390593274 0.0626437202727693
0.176776695296637 0.0322574974645409
0.0883883476483186 0.0162328891101325
0.0441941738241594 0.00812969596227599
0.0220970869120798 0.00406677283180288
0.01104854345604 0.00203420461328599
0.00552427172802006 0.0010219988327219
};
\addlegendentry{$k=1$}
\addplot [semithick, mydarkgray, dashed, forget plot]
table {%
0.353553390593274 0.0563793482454924
0.176776695296637 0.0281896741227462
0.0883883476483186 0.0140948370613731
0.0441941738241594 0.00704741853068658
0.0220970869120798 0.0035237092653433
0.01104854345604 0.00176185463267166
0.00552427172802006 0.000880927316335844
};
\addplot [semithick, mycolor1 , mark=*, mark size=2, mark options={solid}]
table {%
0.353553390593274 0.00969617201109911
0.176776695296637 0.00249264746068635
0.0883883476483186 0.000630699093590112
0.0441941738241594 0.000165628499426855
0.0220970869120798 6.9903736024764e-05
0.01104854345604 6.02430441144039e-05
0.00552427172802006 6.5497024884193e-05
};
\addlegendentry{$k=2$}
\addplot [semithick, mydarkgray, dashed, forget plot]
table {%
0.353553390593274 0.0087265548099892
0.176776695296637 0.0021816387024973
0.0883883476483186 0.000545409675624327
0.0441941738241594 0.000136352418906082
0.0220970869120798 3.40881047265208e-05
0.01104854345604 8.52202618163032e-06
0.00552427172802006 2.13050654540764e-06
};
\addplot [semithick, mycolor1, mark=diamond*, mark size=3, mark options={solid}]
table {%
0.353553390593274 0.000990671255772769
0.176776695296637 0.000136607327983939
0.0883883476483186 5.96625278946528e-05
0.0441941738241594 5.78645533317343e-05
0.0220970869120798 5.92290442648901e-05
0.01104854345604 4.65882923216794e-05
0.00552427172802006 6.13999806009421e-05
};
\addlegendentry{$k=3$}
\addplot [semithick, mydarkgray , dashed, , forget plot]
table {%
0.353553390593274 0.000891604130195492
0.176776695296637 0.000111450516274437
0.0883883476483186 1.39313145343046e-05
0.0441941738241594 1.74141431678809e-06
};
\end{axis}

\end{tikzpicture}	
	\end{subfigure}%
	\begin{subfigure}{0.33\textwidth}
\begin{tikzpicture}

\begin{axis}[
legend cell align={left},
legend style={
  fill opacity=0.8,
  draw opacity=1,
  text opacity=1,
  at={(0.97,0.03)},
  anchor=south east,
  draw=white!80!black
},
width=0.9\columnwidth,
height=1.1\columnwidth,
log basis x={10},
log basis y={10},
tick align=outside,
tick pos=left,
title={$\beta = 0.001 $},
x grid style={white!69.0196078431373!black},
xmin=0.004, xmax=0.45,
xmode=log,
xtick style={color=black},
xlabel style={font=\color{white!15!black}},
xlabel={mesh width $h$},
ylabel style={font=\color{white!15!black}},
y grid style={white!69.0196078431373!black},
ymin=0.7e-05, ymax=0.1,
ymode=log,
ytick style={color=black},
ytick ={1e-5,1e-4,1e-3,1e-2,1e-1}
]
\addplot [semithick, mycolor1, mark=triangle, mark size=3, mark options={solid}]
table {%
0.353553390593274 0.0626411257399284
0.176776695296637 0.0322539172756184
0.0883883476483186 0.0162324403986889
0.0441941738241594 0.00812961701201738
0.0220970869120798 0.00406676020071179
0.01104854345604 0.00203417483045658
0.00552427172802006 0.00103478255762307
};
\addplot [semithick, mydarkgray, dashed]
table {%
0.353553390593274 0.0563770131659355
0.176776695296637 0.0281885065829678
0.0883883476483186 0.0140942532914839
0.0441941738241594 0.00704712664574197
0.0220970869120798 0.003523563322871
0.01104854345604 0.00176178166143551
0.00552427172802006 0.000880890830717768
};
\addplot [semithick, mycolor1, mark=*, mark size=2, mark options={solid}]
table {%
0.353553390593274 0.00968983836972371
0.176776695296637 0.00249230101981378
0.0883883476483186 0.000630625202913553
0.0441941738241594 0.000166117308876967
0.0220970869120798 6.94253992402572e-05
0.01104854345604 5.91395973764038e-05
0.00552427172802006 7.3019657148331e-05
};
\addplot [semithick, mydarkgray, dashed]
table {%
0.353553390593274 0.00872085453275134
0.176776695296637 0.00218021363318784
0.0883883476483186 0.000545053408296961
0.0441941738241594 0.000136263352074241
0.0220970869120798 3.40658380185604e-05
0.01104854345604 8.51645950464022e-06
0.00552427172802006 2.12911487616012e-06
};
\addplot [semithick, mycolor1, mark=diamond*, mark size=3, mark options={solid}]
table {%
0.353553390593274 0.000990181131741786
0.176776695296637 0.000136403338092363
0.0883883476483186 5.93555929924288e-05
0.0441941738241594 5.75747750573337e-05
0.0220970869120798 5.84826557665e-05
0.01104854345604 5.8598045507207e-05
0.00552427172802006 6.05128115076362e-05
};
\addplot [semithick, mydarkgray, dashed]
table {%
0.353553390593274 0.000891163018567608
0.176776695296637 0.000111395377320951
0.0883883476483186 1.39244221651189e-05
0.0441941738241594 1.74055277063988e-06
};
\end{axis}

\end{tikzpicture}	
	\end{subfigure}%
	\begin{subfigure}{0.33\textwidth}
\begin{tikzpicture}

\definecolor{color0}{rgb}{0.12156862745098,0.466666666666667,0.705882352941177}
\definecolor{color1}{rgb}{1,0.498039215686275,0.0549019607843137}
\definecolor{color2}{rgb}{0.172549019607843,0.627450980392157,0.172549019607843}
\definecolor{color3}{rgb}{0.83921568627451,0.152941176470588,0.156862745098039}
\definecolor{color4}{rgb}{0.580392156862745,0.403921568627451,0.741176470588235}
\definecolor{color5}{rgb}{0.549019607843137,0.337254901960784,0.294117647058824}

\definecolor{crimson2143940}{RGB}{214,39,40}
\definecolor{darkgray176}{RGB}{176,176,176}
\definecolor{darkorange25512714}{RGB}{255,127,14}
\definecolor{darkturquoise23190207}{RGB}{23,190,207}
\definecolor{forestgreen4416044}{RGB}{44,160,44}
\definecolor{goldenrod18818934}{RGB}{188,189,34}
\definecolor{gray127}{RGB}{127,127,127}
\definecolor{lightgray204}{RGB}{204,204,204}
\definecolor{mediumpurple148103189}{RGB}{148,103,189}
\definecolor{orchid227119194}{RGB}{227,119,194}
\definecolor{sienna1408675}{RGB}{140,86,75}
\definecolor{steelblue31119180}{RGB}{31,119,180}

\begin{axis}[
legend cell align={left},
legend style={
  fill opacity=0.8,
  draw opacity=1,
  text opacity=1,
  at={(0.97,0.03)},
  anchor=south east,
  draw=white!80!black
},
width=0.9\columnwidth,
height=1.1\columnwidth,
log basis x={10},
log basis y={10},
tick align=outside,
tick pos=left,
title={$\beta = 0.01 $},
x grid style={white!69.0196078431373!black},
xmin=0.004, xmax=0.45,
xmode=log,
xtick style={color=black},
xlabel style={font=\color{white!15!black}},
xlabel={mesh width $h$},
ylabel style={font=\color{white!15!black}},
y grid style={white!69.0196078431373!black},
ymin=0.7e-05, ymax=0.1,
ymode=log,
ytick style={color=black},
ytick ={1e-5,1e-4,1e-3,1e-2,1e-1}
]
\addplot [semithick, mycolor1, mark=triangle, mark size=3, mark options={solid}]
table {%
0.353553390593274 0.0626333946588933
0.176776695296637 0.0322373383294901
0.0883883476483186 0.0162304077567336
0.0441941738241594 0.00812921043929322
0.0220970869120798 0.00406667176645006
0.01104854345604 0.00203410473272526
0.00552427172802006 0.00101888172888849
};
\addplot [semithick, mydarkgray, dashed]
table {%
0.353553390593274 0.056370055193004
0.176776695296637 0.028185027596502
0.0883883476483186 0.014092513798251
0.0441941738241594 0.00704625689912552
0.0220970869120798 0.00352312844956277
0.01104854345604 0.0017615642247814
0.00552427172802006 0.000880782112390712
};
\addplot [semithick, mycolor1, mark=*, mark size=2, mark options={solid}]
table {%
0.353553390593274 0.00967337061861386
0.176776695296637 0.00249209114878444
0.0883883476483186 0.000630432306600328
0.0441941738241594 0.000166554329099608
0.0220970869120798 6.74099242395543e-05
0.01104854345604 2.05460447803243e-05
0.00552427172802006 5.27457372125603e-05
};
\addplot [semithick, mydarkgray, dashed]
table {%
0.353553390593274 0.00870603355675247
0.176776695296637 0.00217650838918812
0.0883883476483186 0.000544127097297031
0.0441941738241594 0.000136031774324258
0.0220970869120798 3.40079435810648e-05
0.01104854345604 8.50198589526633e-06
0.00552427172802006 2.12549647381664e-06
};
\addplot [semithick, mycolor1, mark=diamond*, mark size=3, mark options={solid}]
table {%
0.353553390593274 0.000989366128503692
0.176776695296637 0.000135306875486633
0.0883883476483186 5.7037291691433e-05
0.0441941738241594 5.51083681282613e-05
0.0220970869120798 5.64250424158196e-05
0.01104854345604 5.58713573933551e-05
0.00552427172802006 7.2646564084446e-05
};
\addplot [semithick, mydarkgray, dashed]
table {%
0.353553390593274 0.000890429515653323
0.176776695296637 0.000111303689456666
0.0883883476483186 1.39129611820832e-05
0.0441941738241594 1.73912014776041e-06
};
\end{axis}

\end{tikzpicture}	
	\end{subfigure}
	\ref{Legendforall}
	\caption{ 	Convergence of 
		\eqref{eq:Kuznetsov_space_discr_full_eq} with 
		$\norm{\nabla \pt \sol(t) - \nabla \pt \solh (t) }_{\Ltwo}$ 
		at $t=0.8$
		with the parameters
		$\spacepar =  0.1$, $\timepar = 0.5$ in \eqref{eq:sol_shock} for elements of order $k=1,2,3$
		and 
		$\tau \approx 1.5 \cdot 10^{-3}$
		and damping parameters $\beta  = 0 , 10^{-3}, 10^{-2}$ (from left to right).
		The dashed lines indicate order $\mathcal{O}(h^k)$ for $k=1,2,3$.
	}
	\label{fig:space_conv_exa_1}
\end{figure}

\begin{figure}[h]
	\centering
	\begin{subfigure}{0.33\textwidth}
\begin{tikzpicture}

\definecolor{color0}{rgb}{0.12156862745098,0.466666666666667,0.705882352941177}
\definecolor{color1}{rgb}{1,0.498039215686275,0.0549019607843137}

\definecolor{crimson2143940}{RGB}{214,39,40}
\definecolor{darkgray176}{RGB}{176,176,176}
\definecolor{darkorange25512714}{RGB}{255,127,14}
\definecolor{forestgreen4416044}{RGB}{44,160,44}
\definecolor{lightgray204}{RGB}{204,204,204}
\definecolor{mediumpurple148103189}{RGB}{148,103,189}
\definecolor{sienna1408675}{RGB}{140,86,75}
\definecolor{steelblue31119180}{RGB}{31,119,180}

\begin{axis}[
legend cell align={left},
legend style={
  fill opacity=0.8,
  draw opacity=1,
  text opacity=1,
  at={(0.97,0.03)},
  anchor=south east,
  draw=white!80!black
},
width=0.9\columnwidth,
height=1.1\columnwidth,
log basis x={10},
log basis y={10},
tick align=outside,
tick pos=left,
title={$\beta = 0$},
x grid style={white!69.0196078431373!black},
xmin= 5e-4 , xmax=6e-1,
xmode=log,
xlabel={step size $\tau$},
xtick style={color=black},
ylabel style={font=\color{white!15!black}},
ylabel={$\mathbf{E}(0.8)$},
y grid style={white!69.0196078431373!black},
ymin= 2e-5 , ymax= 3e-2,
ymode=log,
ytick style={color=black},
ytick ={1e-6,1e-5,1e-4,1e-3,1e-2,1e-1,1e0}
]
\addplot [semithick, mydarkgray, dashed]
table {%
0.4 0.0178727119696861
0.2 0.00893635598484305
0.1 0.00446817799242152
0.05 0.00223408899621076
0.025 0.00111704449810538
0.0125 0.00055852224905269
0.00625 0.000279261124526345
0.003125 0.000139630562263173
0.0015625 6.98152811315863e-05
0.00078125 3.49076405657931e-05
};
\addplot [semithick, mycolor2 ,mark=triangle, mark size=3, mark options={solid}]
table {%
0.4 0.00708061923719471
0.2 0.00374271333979106
0.1 0.00244199665379418
0.05 0.00145714337144417
0.025 0.000809333056798198
0.0125 0.000430808053575514
0.00625 0.000224619387681119
0.003125 0.000117343818376708
0.0015625 6.02430441144039e-05
0.00078125 3.17342186961756e-05
};
\end{axis}

\end{tikzpicture}	
	\end{subfigure}%
	\begin{subfigure}{0.33\textwidth}
\begin{tikzpicture}

\definecolor{color0}{rgb}{0.12156862745098,0.466666666666667,0.705882352941177}
\definecolor{color1}{rgb}{1,0.498039215686275,0.0549019607843137}

\definecolor{crimson2143940}{RGB}{214,39,40}
\definecolor{darkgray176}{RGB}{176,176,176}
\definecolor{darkorange25512714}{RGB}{255,127,14}
\definecolor{forestgreen4416044}{RGB}{44,160,44}
\definecolor{lightgray204}{RGB}{204,204,204}
\definecolor{mediumpurple148103189}{RGB}{148,103,189}
\definecolor{sienna1408675}{RGB}{140,86,75}
\definecolor{steelblue31119180}{RGB}{31,119,180}

\begin{axis}[
legend cell align={left},
legend style={
  fill opacity=0.8,
  draw opacity=1,
  text opacity=1,
  at={(0.97,0.03)},
  anchor=south east,
  draw=white!80!black
},
width=0.9\columnwidth,
height=1.1\columnwidth,
log basis x={10},
log basis y={10},
tick align=outside,
tick pos=left,
title={$\beta = 0.001$},
x grid style={white!69.0196078431373!black},
xmin= 5e-4 , xmax=6e-1,
xmode=log,
xlabel={step size $\tau$},
xtick style={color=black},
ylabel style={font=\color{white!15!black}},
y grid style={white!69.0196078431373!black},
ymin= 2e-5 , ymax= 3e-2,
ymode=log,
ytick style={color=black},
ytick ={1e-6,1e-5,1e-4,1e-3,1e-2,1e-1,1e0}
]
\addplot [semithick, mydarkgray, dashed]
table {%
0.4 0.0174572073420309
0.2 0.00872860367101543
0.1 0.00436430183550772
0.05 0.00218215091775386
0.025 0.00109107545887693
0.0125 0.000545537729438464
0.00625 0.000272768864719232
0.003125 0.000136384432359616
0.0015625 6.81922161798081e-05
0.00078125 3.4096108089904e-05
};
\addplot [semithick, mycolor2, mark=triangle, mark size=3, mark options={solid}]
table {%
0.4 0.00706801443810927
0.2 0.00373773981523014
0.1 0.0024388434992582
0.05 0.00145444474141835
0.025 0.000807244071886532
0.0125 0.000430366547045497
0.00625 0.000226846619263895
0.003125 0.000121344088252652
0.0015625 5.91395973764038e-05
0.00078125 3.09964618999127e-05
};
\end{axis}

\end{tikzpicture}	
	\end{subfigure}%
	\begin{subfigure}{0.33\textwidth}
\begin{tikzpicture}

\definecolor{color0}{rgb}{0.12156862745098,0.466666666666667,0.705882352941177}
\definecolor{color1}{rgb}{1,0.498039215686275,0.0549019607843137}

\definecolor{crimson2143940}{RGB}{214,39,40}
\definecolor{darkgray176}{RGB}{176,176,176}
\definecolor{darkorange25512714}{RGB}{255,127,14}
\definecolor{forestgreen4416044}{RGB}{44,160,44}
\definecolor{lightgray204}{RGB}{204,204,204}
\definecolor{mediumpurple148103189}{RGB}{148,103,189}
\definecolor{sienna1408675}{RGB}{140,86,75}
\definecolor{steelblue31119180}{RGB}{31,119,180}

\begin{axis}[
legend cell align={left},
legend style={
  fill opacity=0.8,
  draw opacity=1,
  text opacity=1,
  at={(0.97,0.03)},
  anchor=south east,
  draw=white!80!black
},
width=0.9\columnwidth,
height=1.1\columnwidth,
log basis x={10},
log basis y={10},
tick align=outside,
tick pos=left,
title={$\beta = 0.01$},
x grid style={white!69.0196078431373!black},
xmin= 5e-4 , xmax=6e-1,
xmode=log,
xlabel={step size $\tau$},
xtick style={color=black},
ylabel style={font=\color{white!15!black}},
y grid style={white!69.0196078431373!black},
ymin= 2e-5 , ymax= 3e-2,
ymode=log,
ytick style={color=black},
ytick ={1e-6,1e-5,1e-4,1e-3,1e-2,1e-1,1e0}
]
\addplot [semithick, mydarkgray, dashed]
table {%
0.4 0.0161997144143559
0.2 0.00809985720717794
0.1 0.00404992860358897
0.05 0.00202496430179449
0.025 0.00101248215089724
0.0125 0.000506241075448621
0.00625 0.000253120537724311
0.003125 0.000126560268862155
0.0015625 6.32801344310777e-05
0.00078125 3.16400672155388e-05
};
\addplot [semithick, mycolor2, mark=triangle, mark size=3, mark options={solid}]
table {%
0.4 0.00695875874454807
0.2 0.00369388236738839
0.1 0.00240936825196387
0.05 0.0014321577540838
0.025 0.000790555548123994
0.0125 0.000418361991580299
0.00625 0.000215935386834478
0.003125 0.000112692572174108
0.0015625 2.05460447803243e-05
0.00078125 2.87636974686717e-05
};
\end{axis}

\end{tikzpicture}	
	\end{subfigure}
	\caption{Convergence of 
		\eqref{eq:Euler} with
		$\norm{\nabla \pt \sol(\tn{n}) - \nabla \ptau \solhn{n} }_{\Ltwo}$ 
		for $n = N+1$
		with the parameters
		$\spacepar =  0.1$, $\timepar = 0.5$   in \eqref{eq:sol_shock} 
		with $k=2$ and 
		$h \approx 1.1 \cdot 10^{-2}$
		and damping parameters $\beta  = 0 , 10^{-3}, 10^{-2}$ (from left to right).
		The dashed lines indicate order $\mathcal{O}(\tau)$.	
	}
	\label{fig:time_conv_exa_1}
\end{figure}

In Figure~\ref{fig:time_conv_exa_1}, we present the computed error
\begin{equation}
	\mathbf{E}(\tn{n}) = \norm{\nabla \pt \sol(\tn{n}) - \nabla \ptau \solhn{n} }_{\Ltwo}
\end{equation}
for the fully discrete method \eqref{eq:Euler}
at $n=N+1$ with elements of order $k=2$ and $h \approx 1.1 \cdot 10^{-2}$.
As predicted by Theorem~\ref{thm:ErrorSemiImplicit}, we observe convergence or order $\mathcal{O}(\tau)$ independent of the damping parameter $\beta$.
\newpage

\subsubsection{Convergence in the inviscid limit}

In the next experiment, we verify the sharpness of the results in Theorems~\ref{thm:FE_beta_limit} and \ref{thm:FullyDiscrete_beta_limit}. Here, we use
the same domain $\Om$ and initial data as in \eqref{eq:numexp_inits_smooth},
but parameters and source term are chosen as
\begin{equation} 
	\kappa = 0.3, 
	\qquad
	c^2 = 1,
	\qquad
	\ell = 2,
	\qquad
	f = 0 \,,
\end{equation}
with	
$\spacepar =  0.01$ and $\timepar = 1$.\\
\begin{figure}[h]
\begin{tikzpicture}

\definecolor{crimson2143940}{RGB}{214,39,40}
\definecolor{darkgray176}{RGB}{176,176,176}
\definecolor{darkorange25512714}{RGB}{255,127,14}
\definecolor{forestgreen4416044}{RGB}{44,160,44}
\definecolor{lightgray204}{RGB}{204,204,204}
\definecolor{mediumpurple148103189}{RGB}{148,103,189}
\definecolor{orchid227119194}{RGB}{227,119,194}
\definecolor{sienna1408675}{RGB}{140,86,75}
\definecolor{steelblue31119180}{RGB}{31,119,180}

\begin{axis}[
legend cell align={left},
legend style={
  fill opacity=0.8,
  draw opacity=1,
  text opacity=1,
  at={(0.97,0.03)},
  anchor=south east,
},
width=0.9\columnwidth,
height= 0.4\columnwidth,
log basis x={10},
log basis y={10},
tick align=outside,
tick pos=left,
title={},
x grid style={darkgray176},
xmin=0.000707945784384138, xmax=1.41253754462275,
xmode=log,
xtick style={color=black},
y grid style={darkgray176},
ymin=1.94424295055181e-05, ymax=0.401219455502362,
ymode=log,
ytick style={color=black},
xlabel= {$\beta$},
ylabel={$\mathbf{\overline{E}}(0.8)$},
]
\addplot [semithick, mycolor1, mark=diamond*, mark size=3, mark options={solid}]
table {%
	0.001 9.02720008268504e-05
	0.00215443469003188 0.000193821844872005
	0.00464158883361278 0.000414528292482513
	0.01 0.00087922338531398
	0.0215443469003188 0.00183277419451103
	0.0464158883361278 0.00368890438706007
	0.1 0.00694947468023459
	0.215443469003188 0.0117262497407083
	0.464158883361278 0.0170226158362812
	1 0.0216549842540685
};
\addlegendentry{$\tau = 0.1$, $h = 0.3535$}
\addplot [semithick, mygreen, mark=asterisk, mark size=3, mark options={solid}]
table {%
	0.001 0.000159365560115096
	0.00215443469003188 0.000342124647350278
	0.00464158883361278 0.000731484490049027
	0.01 0.00155046764925466
	0.0215443469003188 0.00322712075429208
	0.0464158883361278 0.00647078911640462
	0.1 0.0120743441533374
	0.215443469003188 0.0199360146206233
	0.464158883361278 0.0280808528375304
	1 0.034374842298982
};
\addlegendentry{$\tau = 0.025$, $h = 0.08839$}
\addplot [semithick, mycolor2,mark=triangle, mark size=3, mark options={solid}]
table {%
	0.001 0.000181584089086305
	0.00215443469003188 0.000390151092179689
	0.00464158883361278 0.000832312798613753
	0.01 0.00176532854694606
	0.0215443469003188 0.00367054329134455
	0.0464158883361278 0.00737630154211059
	0.1 0.0137788953235828
	0.215443469003188 0.0226923282522274
	0.464158883361278 0.0316304856349226
	1 0.0382912246952077
};
\addlegendentry{$\tau = 0.00625$, $h = 0.02210$}
\addplot [semithick, mydarkgray, dashed, forget plot]
table {%
	0.001 0.000199742497994936
	0.00215443469003188 0.000430332166753913
	0.00464158883361278 0.000927122548291215
	0.01 0.00199742497994935
	0.0215443469003188 0.00430332166753913
	0.0464158883361278 0.00927122548291215
	0.1 0.0199742497994936
	0.215443469003188 0.0430332166753913
	0.464158883361278 0.0927122548291215
	1 0.199742497994936
};
\end{axis}

\end{tikzpicture}	
	\caption{Convergence of
		$(\solhnzero{n},\ptau \solhnzero{n})$ and $(\solhnbeta{n},\ptau \solhnbeta{n})$
		in the $ \Hone \times \Ltwo$-norm at the end time $t=0.8$
		for different values of $h$ and $\tau$.
		The dashed line indicates order $\mathcal{O}(\beta)$.}
	\label{fig:beta_conv}
\end{figure}

Since the estimates only compare the numerical solution, we do not need an exact or a reference solution.
We use $k=2$ and
compute the difference
between 
$(\solhnzero{n},\ptau \solhnzero{n})$ and $(\solhnbeta{n},\ptau \solhnbeta{n})$
in the $ \Hone \times \Ltwo$-norm, i.e.,
\begin{equation}
	\mathbf{\overline{E}}(\tn{n}) = 
	\norm{\nabla \solhnzero{n}  - \nabla \solhnbeta{n} }_{\Ltwo}
	+
	\norm{\ptau \solhnzero{n}  - \ptau \solhnbeta{n} }_{\Ltwo} ,
\end{equation}
for different values of $\beta$ at the end time $n=N+1$. \\
\indent We observe 
in Figure~\ref{fig:beta_conv}
that for varying values of $h$ and $\tau$, the convergence in $\beta$ is uniform of order $\mathcal{O}(\beta)$.
Different values of $h$ and $\tau$ lead to qualitatively similar pictures,
with a clustering at the dashed line for finer resolutions which confirms the assertions in
Theorems~\ref{thm:FE_beta_limit} and \ref{thm:FullyDiscrete_beta_limit}.

\subsubsection{Gaussian pulse}

In the last experiment, we simulate the propagation of a Gaussian pulse
on the larger domain $\Omega = [-4,4] \times [-4,4]$.
 We use the initial states
\begin{equation} \label{eq:numexp_inits_Gaussain} 
	\solinit(x) = 
	- e^{- |x|^2},
	\qquad
	\soltinit(x) = 
	0 ,
\end{equation}
where although $\solinit$ is not zero on the boundary of $\Omega$,  by the size of the domain it is still 
 within machine precision.
Further, we take the following parameters and source term:
\begin{equation}  \label{eq:numexp_parameters_Gaussian}
	\kappa = -0.29, 
	\qquad
	c^2 = 1,
	\qquad
	\ell = 2,
	\qquad
	f = 0 \,,
\end{equation}
and vary 
$\beta \geq 0$.
Here we have to choose $\kappa> - 0.3$ in order to prevent 
$1 + \kappa \ptau \solhn{n} < 0$ after a short time.
Since we do not have an exact solution, we first compute a reference solution with
finer spatial and temporal resolution, i.e.,
$h_{\mathrm{ref}} \approx 4 \cdot 10^{-2}$
and
$\tau_{\mathrm{ref}} = 4 \cdot  10^{-4}$.
Due to the larger domain, we have to increase the number of elements by a factor $16$, and hence compute the errors only for a coarser resolution.
\begin{figure}[h]
	\centering
	\begin{subfigure}{0.45\textwidth}
\begin{tikzpicture}

\definecolor{color0}{rgb}{0.12156862745098,0.466666666666667,0.705882352941177}
\definecolor{color1}{rgb}{1,0.498039215686275,0.0549019607843137}
\definecolor{color2}{rgb}{0.172549019607843,0.627450980392157,0.172549019607843}
\definecolor{color3}{rgb}{0.83921568627451,0.152941176470588,0.156862745098039}
\definecolor{color4}{rgb}{0.580392156862745,0.403921568627451,0.741176470588235}
\definecolor{color5}{rgb}{0.549019607843137,0.337254901960784,0.294117647058824}

\begin{axis}[
legend cell align={left},
legend style={
	fill opacity=0.8,
	draw opacity=1,
	text opacity=1,
	at={(0.97,0.03)},
	anchor=south east,
	draw=white!80!black,
	column sep=7mm,
	/tikz/every odd column/.append style={column sep=0cm},
},
width=0.9\columnwidth,
height=0.9\columnwidth,
log basis x={10},
log basis y={10},
tick align=outside,
tick pos=left,
x grid style={white!69.0196078431373!black},
xmin=0.07, xmax=3.5,
xmode=log,
xtick style={color=black},
xlabel style={font=\color{white!15!black}},
xlabel={mesh width $h$},
ylabel style={font=\color{white!15!black}},
ylabel={$\mathbf{E}(0.8)$},
y grid style={white!69.0196078431373!black},
ymin=0.7e-02, ymax=1.1e1,
ymode=log,
ytick style={color=black},
legend columns=-1,
legend to name=Legendforall2,
]
\addplot [semithick, mycolor1,mark=diamond*, mark size=3, mark options={solid}]
table {%
2.82842712474619 2.37473509587258
1.4142135623731 0.978462136311972
0.707106781186548 0.469960394992366
0.353553390593274 0.0959993649340826
0.176776695296638 0.0221978854551011
0.0883883476483191 0.0102461694530741
};
\addlegendentry{$\beta = 0.0$}
\addplot [semithick, mycolor2 , mark=*, mark size=2, mark options={solid}]
table {%
2.82842712474619 2.37055280622167
1.4142135623731 0.9759710680071
0.707106781186548 0.448478644901117
0.353553390593274 0.0924286432618954
0.176776695296638 0.0222496536061632
0.0883883476483191 0.0105795545529577
};
\addlegendentry{$\beta = 0.001$}
\addplot [semithick, mygreen, mark=triangle, mark size=3, mark options={solid}]
table {%
2.82842712474619 2.33853689963811
1.4142135623731 0.964141917735042
0.707106781186548 0.364581855754306
0.353553390593274 0.0849877192599953
0.176776695296638 0.0224100752534809
0.0883883476483191 0.00949541641658864
};
\addlegendentry{$\beta = 0.01$}
\addplot [semithick, mydarkgray, dashed, forget plot]
table {%
2.82842712474619 10.6956370516453
1.4142135623731 2.67390926291133
0.707106781186548 0.668477315727832
0.353553390593274 0.167119328931958
0.176776695296638 0.0417798322329897
0.0883883476483191 0.0104449580582475
};
\addlegendentry{$O(h^2)$}
\end{axis}

\end{tikzpicture}	
	\end{subfigure}%
	\begin{subfigure}{0.45\textwidth}
\begin{tikzpicture}

\begin{axis}[
legend cell align={left},
legend style={
	fill opacity=0.8,
	draw opacity=1,
	text opacity=1,
	at={(0.97,0.03)},
	anchor=south east,
	draw=white!80!black,
	column sep=7mm,
	/tikz/every odd column/.append style={column sep=0cm},
},
width=0.9\columnwidth,
height=0.9\columnwidth,
log basis x={10},
log basis y={10},
tick align=outside,
tick pos=left,
x grid style={white!69.0196078431373!black},
xmin=0.4e-03, xmax=0.5,
xmode=log,
xtick style={color=black},
xlabel style={font=\color{white!15!black}},
xlabel={step size $\tau$},
ylabel style={font=\color{white!15!black}},
y grid style={white!69.0196078431373!black},
ymin=0.7e-02, ymax=1.1e1,
ymode=log,
ytick style={color=black},
]
\addplot [semithick, mycolor1,mark=diamond*, mark size=3, mark options={solid}]
table {%
	0.4 1.36967529797562
	0.2 1.03847845800872
	0.1 1.21589407616438
	0.05 0.901621357518555
	0.025 0.523587235433089
	0.0125 0.27531099118114
	0.00625 0.13666019040107
	0.003125 0.0640943156946739
	0.0015625 0.0272167933265117
	0.00078125 0.0102461694530741
};
\addplot [semithick, mycolor2 , mark=*, mark size=2, mark options={solid}]
table {%
	0.4 1.3748961613334
	0.2 1.03376466448738
	0.1 1.20206024797546
	0.05 0.890609763822206
	0.025 0.51774552824441
	0.0125 0.272768736380787
	0.00625 0.135715137792282
	0.003125 0.0639734442312093
	0.0015625 0.0274118117388685
	0.00078125 0.0105795545529577
};
\addplot [semithick, mygreen, mark=triangle, mark size=3, mark options={solid}]
table {%
	0.4 1.42768250927443
	0.2 1.00103001260378
	0.1 1.08798391781214
	0.05 0.796082629006441
	0.025 0.463277081020655
	0.0125 0.244904669010106
	0.00625 0.12185614570575
	0.003125 0.0576320199478143
	0.0015625 0.0250561056785644
	0.00078125 0.00949541641658864
};
\addplot [semithick, mydarkgray, dashed, forget plot]
table {%
	0.4 9.35855199020161
	0.2 4.6792759951008
	0.1 2.3396379975504
	0.05 1.1698189987752
	0.025 0.584909499387601
	0.0125 0.2924547496938
	0.00625 0.1462273748469
	0.003125 0.0731136874234501
	0.0015625 0.036556843711725
	0.00078125 0.0182784218558625
};
\end{axis}
\end{tikzpicture}	
	\end{subfigure}
	\ref{Legendforall2}
	\caption{
		\textbf{Left}: Convergence of 
		\eqref{eq:Kuznetsov_space_discr_full_eq} with 
		$\norm{\nabla \pt \sol(t) - \nabla \pt \solh (t) }_{\Ltwo}$ 
		at $t=0.8$ for elements of order $k=2$
		and 
		$\tau \approx 7.8 \cdot 10^{-4}$
		and damping parameters $\beta  = 0 , 10^{-3}, 10^{-2}$.
		The dashed line indicates order $\mathcal{O}(h^2)$.
		\textbf{Right}:
		Convergence of 
		\eqref{eq:Euler} with
		$\norm{\nabla \pt \sol(\tn{n}) - \nabla \ptau \solhn{n} }_{\Ltwo}$ 
		for $n = N+1$
		with $k=2$ and 
		$h \approx 9 \cdot 10^{-2}$
		and damping parameters $\beta  = 0 , 10^{-3}, 10^{-2}$.
		The dashed line indicates order $\mathcal{O}(\tau)$.	
	}
	\label{fig:time_conv_exa_2}
\end{figure}
~\\
In Figure~\ref{fig:time_conv_exa_2}, we observe that also in this example we have convergence of optimal order uniformly in the damping parameter $\beta$.

\section{Uniform finite element analysis} \label{sec:FE_analysis}
	
%
%
%
%
%
	

	\indent	In this section, we conduct a $\beta$-uniform analysis of the semi-discrete problem \eqref{eq:Kuznetsov_space_discr} with approximate data \eqref{approx_data}.  We begin by discussing the general strategy. Due to the type of quasilinearity present in the problem and the need to conduct estimates uniformly in $\beta$, one would have to resort to higher-order Sobolev spaces to mimic the approach of the $\beta$-uniform well-posedness analysis of the Kuznetsov equation in~\cite{kaltenbacher2022parabolic}.   As we cannot exploit such global spatial smoothness arguments for the approximate solution, we rely instead on inverse finite element estimates in careful combination with working with a time-differentiated problem given by
		\begin{equation} \label{eq:Kuznetsov_space_discr_no_kappa_pt}
			\begin{aligned}
		\begin{multlined}[t]	\ip{( 1 +  \kappa \pt \solh )\pt^3 \solh }{ \phih}+\kappa\ip{ (\pt^2 \solh)^2 }{ \phih}
			- 
			\ip{c^2 \Deltah \pt \solh }{ \phih}	\\- 
			\ip{\beta \Deltah \pt^2 \solh }{ \phih}
				+
			\ell\ip{ \nabla \pt \solh \cdot \nabla \pt \solh }{ \phih}
			+
			\ell\ip{ \nabla \solh \cdot \nabla \pt^2 \solh }{ \phih} \\
			=
			\ip{\pt f_h }{\phih}.
		\end{multlined}
		\end{aligned}
	\end{equation}
The ``problematic" nonlinear term in \eqref{eq:Kuznetsov_space_discr_no_kappa_pt} is the one involving $\ell \nabla \solh \cdot \nabla \pt^2 \solh$, and it has to be treated as a right-hand side perturbation. In the literature, error bounds for nonlinear (wave-type) problems are often established via some variant of a fixed-point argument for the numerical solution $\solh$ which combines existence and error analysis; see, e.g.,~\cite{makridakis1993finite, ortner2007discontinuous, nikolic2019priori, shao2022discontinuous} and the references provided therein. However, such strategies do not transfer easily to our setting as applying an inverse bound to estimate  $\nabla \pt^2 \solh$ would prevent the fixed-point iterates to match in the order of $h$-convergence. 
\\
	\indent
		Our finite element analysis instead builds upon that of~\cite{hochbruck2022error} to first show that an accurate approximate solution exists $\solh$ on a discretization-dependent time interval $[0, \finalth]$.   We then derive uniform estimates for 
		\[\errh = \projRitz \sol -\solh,\]  
		which in turn allow extending the existence interval and optimal error bounds to the whole time interval $[0,T]$. Crucially, with this approach we can exploit the polynomial structure of the nonlinearity and the fact that
	 \[ \pt^2 \errh  \nabla \pt^2 \errh = \frac12 \nabla (\pt^2 \errh)^2 \]
	 to compensate inverse estimates with smallness conditions
	 on the error $\errh$; see Proposition~\ref{prop:fe_FirstEstimate} for details. 
\subsection{Auxiliary results} Before we turn to the proofs of the main results in this section, we recall the relevant known estimates from the literature that we employ frequently within our analysis. We rely on the approximation properties of the Ritz projection for $0\leq \ell \leq k$:
%
	\begin{align} \label{eq:projRitz_approx}
	\norm{\varphi - \projRitz \varphi}_{L^p(\Om)}	
	+
	h \norm{\varphi - \projRitz \varphi}_{W^{1,p}(\Om)} &\leq C h^{\ell+1}  \norm{\varphi}_{W^{\ell+1,p}(\Omega)},
	\quad \varphi \in W^{\ell+1,p}(\Omega),
\end{align}
for all $2\leq p \leq \infty$; see, for example, \cite[Thm.\ 8.5.3]{BreS08}.
In addition, we have the following bounds for the interpolant:
\begin{equation}  \label{eq:Ih_approx}
	\norm{\varphi - \Ih \varphi}_{L^p(\Om)} 
	+
	h
	\norm{\varphi - \Ih \varphi}_{W^{1,p}(\Om)} \leq C h^{\ell+1}  \norm{\varphi}_{W^{\ell+1,p}(\Om)},
	\quad \varphi \in W^{\ell+1,p}(\Om),
\end{equation}
for $2 \leq p \leq \infty$ and $1 \leq \ell \leq k$.

For $\phih \in \Vh$
also
the discrete Sobolev embedding
%
\begin{equation} \label{eq:discrete_Sobolev_embedding}
	\norm{\phih}_{ \Linf} + \norm{\phih}_{W^{1,6}(\Om)} \leq C \norm{\Deltah \phih}_{\Ltwo}
\end{equation}
with a constant $C$ independent of $h$ is heavily used, see for example~\cite{Doe23, FujSS01,SuzF86}.
%
Furthermore, we rely on the following inverse estimates:
\begin{subequations} \label{eq:inv_estimate}
	\begin{align+}
		\norm{ \nabla \varphi_h}_{\Ltwo} &\leq C h^{-1} \norm{ \varphi_h}_{\Ltwo},
		\\
		\norm{\Delta_h \varphi_h}_{\Ltwo} &\leq C h^{-1} \norm{ \nabla \varphi_h}_{\Ltwo},
		\\
		\norm{\varphi_h}_{\Linf} &\leq C h^{-d/p} \norm{ \varphi_h}_{L^p(\Omega)},
	\end{align+}
\end{subequations}
for $\varphi_h \in \Vh$ and $p\in[1,\infty]$, 
with constants independent of $h$.

\subsection{Finite element analysis} 
	We begin the analysis by defining the (possibly $h$-dependent) time $\finalth$ as follows:
	\begin{equation} \label{def:finaltimeh}
		\begin{aligned}
			\finalth \coloneqq \sup  \Big \{  t \in (0,T] \mid
	\	&\text{a unique solution }	\solh \in H^{3}(0,t; V_h) \text{ of \eqref{eq:Kuznetsov_space_discr_full_eq} exists, and} \\
			%
			&\, h^{-1-d/6} \norm{\pt^2 \errh(s)}_{\Ltwo} \leq C_0,  \\
			&\, h^{-1-d/6} \norm{\nabla \pt \errh(s)}_{\Ltwo}
			\leq C_0 \text{ for all } s \in [0,t]  \Big  \}
		\end{aligned}
	\end{equation}
for some $C_0>0$. Our first task is to establish that this set is non-empty. To this end, we estimate $\norm{\pt^2 \errh}_{\Ltwo}$ and $\norm{\nabla \pt \errh}_{\Ltwo}$ at initial time.
	\begin{lemma} \label{lemma:est_errh_zero}
			Under the assumptions of Theorem~\ref{thm:kuznetsov_space_main},
		with approximate initial values chosen to be the Ritz projections of the exact ones as in 
		\eqref{eq:init_choice_space} and $\pt^2 \solh(0)$ determined by \eqref{third_ic}, the following estimate holds:
			\begin{equation}
						\norm{\pt^2 \errh(0)}_{\Ltwo} +\norm{\nabla \pt \errh(0)}_{\Ltwo} \leq C h^k
				\end{equation}
				with a constant $C>0$ independent of $h$ and $\beta$.
		\end{lemma}
\begin{proof}
With our choice of the approximate initial data,  $\errh(0)=\pt \errh(0)=0$ and thus trivially
\[
\norm{\nabla \pt \errh(0)}_{\Ltwo} \leq C h^k.
\]
It remains to estimate $\pt^2 \errh(0)$. We note that the Ritz projection of $u$ satisfies the following problem at $t=0$ :
\begin{equation} \label{weakform_projRitz}
	\begin{aligned}
		\begin{multlined}[t]	
			\ip{( 1 +  \kappa \pt {\solh}(0))\pt^2 \projRitz \sol (0)}{ \phih}
			- 
			\ip{c^2 \Deltah \projRitz \sol(0)   }{\phih}- 
			\ip{\beta \Deltah \pt \projRitz \sol (0)  }{\phih}
			\\	+
			\ell \ip{ \nabla {\solh(0)} \cdot \nabla \pt \projRitz \sol (0) }{\phih}
			=
			\ip{f_h(0) }{\phih} 
			+
			\ip{\delta_h(0)}{\phih} \end{multlined}
	\end{aligned}
\end{equation}
for all $\phih \in V_h$, with the defect at zero satisfying
\begin{equation}
	\begin{aligned}
		\ip{\delta_h(0)}{\phih} =&\, \begin{multlined}[t]
			\ip{( 1 + \kappa \pt {\solh}(0))\pt^2 \projRitz \sol(0) -( 1 + \kappa \pt  \sol(0))\pt^2  \sol(0) }{\phih}
			\\
			+ \ell \ip{\nabla  {\solh(0)} \cdot \nabla \pt \projRitz \sol(0)  -\ell \nabla  \sol(0) \cdot \nabla \pt  \sol(0)}{\phih}
			\\
			+\ip{f(0)-f_h(0)}{\phih}.
		\end{multlined}
	\end{aligned}
\end{equation}
Since $\pt^2 \solh(0)$ is determined by \eqref{third_ic}, by subtracting \eqref{third_ic} from \eqref{weakform_projRitz} and using the fact that $\errh(0)=\pt \errh(0)=0$, we see that $\pt^2 \errh(0)$ solves
\begin{align} \label{weakform_errh}
	\ip{(1+\kappa \pt \solh(0)) \pt^2 \errh(0) }{ \phih}
	=  \ip{\delta_h(0)}{\phih} 
\end{align}
for all $\phih \in V_h$. By the inverse estimates \eqref{eq:inv_estimate} and 
the approximation properties of the Ritz projection stated in \eqref{eq:projRitz_approx}, we have
\begin{align} \label{eq:init_estimate_u_1}
	\norm{\pt \solh (0)}_{\Linf} 
	&\leq \norm{\soltinit}_{\Linf}
	 + 
	 \norm{ \soltinit - \Ih \soltinit}_{\Linf}
	 +
	 \norm{\Ih \soltinit - \projRitz \soltinit}_{\Linf}
	\\
	&\leq \norm{\soltinit}_{\Linf} +C h^{1-d/4} \|\soltinit\|_{W^{1,4}(\Omega)}.
\end{align}
Thus, for sufficiently small $h\leq h_0$ (relative to $\soltinit$), we can guarantee that 
\[
|\kappa| \|\pt \solh(0)\|_{\Linf}  <1.
\]
This further implies that there exists $\gamma>0$, independent of $h$ and $\beta$, such that
\begin{align}\label{gamma_nondegeneracy}
	1 +\kappa \pt \solh(0) \geq \gamma >0.
\end{align} 
By the approximation properties of the Ritz projection, and the accuracy of $f_h$ assumed in \eqref{assumption_accuracy_fh}, we have
\begin{equation} \label{est_defect_zero}
		\norm{\delta_h(0)}_{\Ltwo} \lesssim h^k.
\end{equation}
Therefore, by using $\phih= \pt^2 \errh(0)$ in \eqref{weakform_errh} and relying on \eqref{gamma_nondegeneracy} and \eqref{est_defect_zero}, we immediately obtain
\begin{equation} \label{est_errh_zero}
	\norm{\pt^2 \errh(0)}_{\Ltwo} \lesssim h^k,
\end{equation}
which concludes the proof.
\end{proof}	

We next aim to prove that $\finalth>0$ by applying a local version of the Picard--Lindel\"of theorem to the time-differentiated semi-discrete problem.
	\begin{lemma} \label{Lemma:Initial_bounds}
			Under the assumptions of Theorem~\ref{thm:kuznetsov_space_main}, we have
	 $\finalth > 0$.
	\end{lemma}
	
	\begin{proof}
	For the purposes of stating the time-differentiated problem in a compact manner, we introduce the discrete multiplication operator  $\lambda_h = \lambda_h(\pt \solh) \colon \Vh \to \Vh$ defined by  
\begin{equation}
	\ip{\lambda_h \varphi_h}{\psi_h} = 	\ip{(1+\kappa \pt \solh) \varphi_h}{\psi_h}
\end{equation}
for $\varphi_h,\psi_h \in \Vh$, which is invertible at $t=0$ by \eqref{gamma_nondegeneracy}.
The time-differentiated semi-discrete problem can then be written as
\begin{equation} \label{rewritten_problem}
			\begin{aligned}
		\begin{multlined}[t]
	   				\pt^3 \solh 
				   =
				   \Bigl( \lambda_h^{-1} 
				   \bigl( c^2  \Delta_h \solh
				   + 
				   \beta \Delta_h \pt \solh
				   - 
				   \ell  \nabla \solh \cdot \pt \nabla \pt \solh
				   - 
				    f_h \bigr) \Bigr)_t
		\end{multlined}
	\end{aligned}
\end{equation}
and further rewritten as a first-order problem for $(\solh, \pt \solh, \pt^2 \solh)^T$. Unique solvability then follows by a similar reasoning to that of~\cite[Lemma 4.2]{hochbruck2022error} using a local version of the Picard--Lindel\"of theorem on the open set
\begin{equation}\label{def_U}
	\begin{aligned}
		U_h = \{ (\solh, \pt \solh, \pt^2 \solh) \in \left(C([0,t]; V_h)\right)^3:& \  |\kappa| \|\pt \solh(s)\|_{\Linf}  <1, \\
			&\, h^{-1-d/6} \norm{\pt^2 \errh(s)}_{\Ltwo} < C_0,  \\
		&\, h^{-1-d/6} \norm{\nabla \pt \errh(s)}_{\Ltwo}
		< C_0,\ s \in [0,t]  \}.
	\end{aligned}
\end{equation}
 We check first that $(\solh(0), \pt \solh(0), \pt^2 \solh(0)) \in U_h$. As concluded in the proof of Lemma~\ref{lemma:est_errh_zero}, for $h \leq h_0$ small enough, we have
\[
 |\kappa| \|\pt \solh(0)\|_{\Linf}  <1.
\]
By Lemma~\ref{lemma:est_errh_zero}, we also have
\begin{equation} \label{est_errh_zero}
 \norm{\pt^2 \errh(0)}_{\Ltwo} \lesssim h^k.
\end{equation}
Therefore, since $\pt \errh(0)=0$, for $h\leq h_0$ and $k \geq 2$, we conclude that
\begin{align}
	h^{-1-d/6} \max \{ \norm{\pt^2 \errh(0)}_{\Ltwo}   , \norm{\nabla \pt \errh(0)}_{\Ltwo}    \} =	h^{-1-d/6}  \norm{\pt^2 \errh(0)}_{\Ltwo}     < C_0
\end{align} 
and thus $(\solh(0), \pt \solh(0), \pt^2 \solh(0)) \in U_h$.
 \\
\indent Equation  \eqref{rewritten_problem} rewritten as a first-order system in time is driven by a locally Lipschitz continuous right-hand side. Indeed, Lipschitz continuity of the right-hand side follows analogously to the arguments of~\cite[Lemma 4.2]{hochbruck2022error} by the fact that $V_h$ is a finite-dimensional space and that we can use inverse estimates  
\eqref{eq:inv_estimate}
for functions in $V_h$.   \\
\indent Thus by the local version of the Picard--Lindel\"of theorem, a unique solution $\solh \in H^3(0,\finalth; V_h) \hookrightarrow C^2([0,T]; V_h)$ of 
\eqref{rewritten_problem} supplemented with initial data exists on $[0,\tilde{t}]$ for some $\tilde{t}>0$. Time integrating \eqref{rewritten_problem} and using \eqref{third_ic} shows that $\solh$ solves \eqref{eq:Kuznetsov_space_discr}, \eqref{eq:init_choice_space}.
 We therefore conclude that $\finalth>0$.
\end{proof}

We have shown that a unique approximate solution exists on $[0,\finalth]$. The next result establishes additional uniform bounds on this time interval.
	
	\begin{lemma} \label{lem:a-priori_bounds_on_solh} 
	Let the assumptions of Theorem~\ref{thm:kuznetsov_space_main} hold. On the interval $[0,\finalth]$, the following bound holds for sufficiently small $h$:
		\begin{equation} \label{est_solh_1}
			\begin{aligned}
			\norm{ \solh }_{L^\infty_t(\Woneinf)} 
			+
			\norm{ \nabla \pt \solh }_{L^\infty_t(\Linf)} 
			+
			\norm{ \pt^2 \solh }_{L^\infty_t(\Linf)}
			\lesssim 1.
			\end{aligned}
		\end{equation}
		In addition,
		\begin{equation} \label{gamma_bound}
			1 + \kappa \pt \solh \geq 
			\quasilowerbound > 0,  \quad (x,t) \in \Omega \times [0, \finalth],
		\end{equation}
		where  $\gamma>0$ does not depend on $h$, $\beta$, or $\finalth$.
	\end{lemma}

	\begin{proof}
		Using the stability properties of the Ritz projection
		in \eqref{eq:projRitz_approx}
		and the definition of $\finalth$, we obtain the following bound:
		\begin{equation} \label{bound_nabla_solh}
			\begin{aligned}
				\norm{ \solh }_{\Woneinf} 
				&\lesssim 
				\norm{ \projRitz \sol }_{\Linf}+	\norm{\nabla \projRitz \sol }_{\Linf} +  \norm{ \errh }_{\Woneinf}\\
				&\lesssim 
				\norm{ \sol }_{\Woneinf} + h^{-d/2} \norm{\nabla \errh }_{\Ltwo} \leq C, 
			\end{aligned}
		\end{equation}
		since $d/2 \geq 1+ d/6$, as well as 
		\begin{equation}
			\norm{ \pt^2 \solh }_{\Linf} \leq
			\norm{ \pt^2 \sol }_{\Woneinf}
			+
			C h^{-d/2} \norm{ \pt^2 \errh}_{\Ltwo}
			\leq C
		\end{equation}
		on $[0, \finalth]$. Furthermore,
		\begin{align}
			\norm{ \nabla \pt \solh }_{\Linf} 
			&\lesssim 
			\norm{ \pt \projRitz \sol }_{\Woneinf}
			 + \norm{ \nabla \pt \errh }_{\Linf}
			\\
			&\lesssim 
			\norm{ \pt \sol }_{\Woneinf} 
			+ h^{-d/2} \norm{ \nabla \pt \errh }_{\Ltwo} \leq C
		\end{align}
for all $t \in [0, \finalth]$.	The bound  \eqref{gamma_bound} follows by the solvability of the (differentiated) semi-discrete problem in $U_h$; cf.\ \eqref{def_U}.
	\end{proof}

	Our main task in the remaining of this section is to prove that
	\begin{equation} \label{goal_est_finalth}
		\norm{\pt^2 \errh(\finalth)}_{\Ltwo}   + \norm{\nabla \pt \errh(\finalth)}_{\Ltwo} \lesssim h^k,
	\end{equation}
	where the error $\errh = \projRitz \sol -\solh$ satisfies
	\begin{equation} \label{eq_errh}
		\begin{aligned} 
			&\begin{multlined}[t]	\ip{(1+\kappa \pt \solh) \pt^2 \errh }{ \phih}+\kappa\ip{ \pt \errh \pt^2 \projRitz \sol }{ \phih}
				- 
				\ip{c^2 \Deltah \errh }{\phih} 	\\- 
				\ip{\beta \Deltah \pt \errh }{\phih} \end{multlined}\\
			&= 
			-
			\ell \ip{\nabla \solh \cdot \nabla \pt \errh  + \nabla \errh \cdot \nabla \pt \projRitz \sol  }{\phih}
			+
			\ip{\delta_h}{\phih}
		\end{aligned}
	\end{equation}
			and the defect is given by
	\begin{equation} \label{def_defect}
		\begin{aligned}
			\ip{\delta_h}{\phih} &= \begin{multlined}[t]
				\ip{( 1 + \kappa \pt {\solh})\pt^2 \projRitz \sol -( 1 + \kappa \pt  \sol)\pt^2  \sol }{\phih}
				\\
				+ \ell \ip{\nabla  {\solh} \cdot \nabla \pt \projRitz \sol  -\ell \nabla  \sol \cdot \nabla \pt  \sol}{\phih}
				+\ip{f-f_h}{\phih}
			\end{multlined}
		\end{aligned}
	\end{equation}
	for all $\phih \in V_h$. The bound \eqref{goal_est_finalth} will allow us to extend the existence interval beyond $[0, \finalth]$. \\

To prove \eqref{goal_est_finalth}, we use a two-step testing procedure. In the first step, we test the time-differentiated error equation with $\pt^2 \errh$. 

	\begin{myproposition} \label{prop:fe_FirstEstimate}

			Let the assumptions of Theorem~\ref{thm:kuznetsov_space_main} hold.  For $t\in[0,\finalth]$, it holds
		\begin{equation} \label{est_eh_Step1}
			\begin{aligned}
				&	\norm{\pt^2 \errh(t) }^2_{\Ltwo}
				+   \norm{\nabla \pt \errh(t)}^2_{\Ltwo} + \beta \intt \norm{\nabla \pt^2 \errh}^2_{\Ltwo} \\
				&\lesssim \begin{multlined}[t]			
					\norm{\pt^2 \errh(0) }^2_{\Ltwo}+
					\int\limits_0^t \big(
					\norm{\nabla \pt \errh(s)  }^2_{\Ltwo}
					+
					\norm{\pt^2 \errh(s)}^2_{\Ltwo}
					+
					\norm{\pt \delta_h (s)}^2_{\Ltwo} \big)\ds\\+	\alpha 
					\int_0^t
					\norm{\Deltah \errh(s)}_{\Ltwo}^2 \dint{s}. \end{multlined}
			\end{aligned}	
		\end{equation}
		for all $t \in [0, \finalth]$ and $\alpha>0$, with the hidden constant independent of $h$, $\finalth$, and $\beta$.
	\end{myproposition}
	\begin{proof}
	As announced,	we test the time-differentiated error equation with
		$\phih=\pt^2 \errh(t)$:
		\begin{align} 
			&\begin{multlined}[t]	\ip{(1+\kappa \pt \solh)\pt^3 \errh }{ \pt^2 \errh}+	\kappa\ip{\pt^2 \solh \pt^2 \errh }{ \pt^2 \errh}
				+\kappa\ip{ \pt^2 \errh \pt^2 \projRitz \sol }{ \pt^2 \errh}\\+\kappa\ip{ \pt \errh \pt^3 \projRitz \sol }{ \pt^2 \errh} - 
				\ip{c^2 \Deltah  \pt \errh }{\pt^2 \errh} - 
				\ip{\beta \Deltah  \pt^2 \errh }{\pt^2 \errh}
			\end{multlined} \\
			=&\, \begin{multlined}[t]
				-
				\ell \ip{\nabla \pt \solh \cdot \nabla \pt \errh  + \nabla \pt \errh \cdot \nabla \pt \projRitz \sol  }{\pt^2 \errh}
				\\ -
				\ell \ip{\nabla \solh \cdot \nabla \pt^2 \errh  
					+ \nabla \errh \cdot \nabla \pt^2 \projRitz \sol  }{\pt^2 \errh}
								+
				\ip{\pt \delta_h}{\pt^2 \errh}
			\end{multlined} 
		\end{align}
		for all $t \in [0, \finalth]$. Using integration by parts in time and
		 Young's inequality yields the following estimate:
		\begin{equation}\label{ineq:error_test_pt2_errh}
			\begin{aligned}
				&	\pt \norm{(1+\kappa \pt \solh)^{1/2} \pt^2 \errh }^2_{\Ltwo}
				+ c^2  \pt  \norm{\nabla \pt \errh}^2_{\Ltwo}+ \beta    \norm{\nabla \pt^2 \errh}^2_{\Ltwo} \\
				&\lesssim \begin{multlined}[t]
					\|\pt^2 \solh \pt^2 \errh \|^2_{\Ltwo}+	\| \pt^2 \errh \pt^2 \projRitz \sol \|^2_{\Ltwo}+\| \pt \errh \pt^3 \projRitz \sol \|^2_{\Ltwo}
						+
					\norm{\pt \errh }^2_{\Ltwo} 
					\\			
					+	\ell^2 \| \nabla \pt \solh \cdot \nabla \pt \errh \|^2_{\Ltwo} +\ell^2\| \nabla \pt \errh \cdot \nabla \pt \projRitz \sol \|_{\Ltwo}^2
					\\
					+
					\ell  \ip{\nabla \solh \cdot \nabla \pt^2 \errh  }{\pt^2 \errh}
					+	\ell^2 \|  \nabla \errh \cdot \nabla \pt^2 \projRitz \sol \|^2_{\Ltwo}
					+
					\norm{\pt \delta_h }^2_{\Ltwo} 
				\end{multlined}
			\end{aligned}
		\end{equation}
		for $0\leq t \leq\finalth$. We can rely on the bounds on $\solh$ obtained in Lemma~\ref{lem:a-priori_bounds_on_solh} to further estimate the right-hand side terms. First, using also the embedding $\Hone \hookrightarrow L^p(\Om)$, $p \in [1,6]$, we estimate the first three terms on the right-hand side of \eqref{ineq:error_test_pt2_errh} as follows:
		\begin{equation}
			\begin{aligned}
				&\|\pt^2 \solh \pt^2 \errh \|^2_{\Ltwo}+	\| \pt^2 \errh \pt^2 \projRitz \sol \|^2_{\Ltwo}+\| \pt \errh \pt^3 \projRitz \sol \|^2_{\Ltwo} \\
				&\lesssim \|\pt^2 \solh \|^2_{\Linf}\|\pt^2 \errh \|^2_{\Ltwo}+	\| \pt^2 \errh\|_{\Ltwo}^2 \|\pt^2 \projRitz \sol \|^2_{\Linf}+\| \pt \errh \|^2_{\Lsix}\| \pt^3 \projRitz \sol \|^2_{\Lthree}\\
				&\lesssim \|\pt^2 \errh \|^2_{\Ltwo}
					+\| \pt \errh \|^2_{\Hone},
			\end{aligned}
		\end{equation}
		where we have employed $\|\pt^2 \solh(t)\|_{\Linf} \lesssim 1$ on $[0, \finalth]$. Next, we can bound the $\ell^2$ terms in the following manner using \eqref{est_solh_1}:
		\begin{equation} \label{est_grad_1}
			\begin{aligned}
			\ell^2	\|\nabla \pt \solh \cdot \nabla \pt \errh   \|^2_{\Ltwo}
				&\lesssim 
				\norm{\nabla \pt \solh }^2_{\Linf}
				\norm{\nabla \pt \errh  }^2_{\Ltwo}
				\lesssim 
				\norm{\nabla \pt \errh  }^2_{\Ltwo}.
			\end{aligned}
		\end{equation}
		Similarly, using \eqref{eq:projRitz_approx},
		\begin{equation}\label{est_grad_2}
			\begin{aligned}
			\ell^2	\|\nabla \pt \errh \cdot \nabla \pt \projRitz \sol \|^2_{\Ltwo}
				\lesssim 
				\norm{\nabla \pt \errh  }_{\Ltwo}^2
				\norm{\nabla \pt  \projRitz \sol }_{\Linf}^2
				\lesssim 
			\norm{\pt u}^2_{\Woneinf}	\norm{\nabla \pt \errh  }^2_{\Ltwo},
			\end{aligned}
		\end{equation}
		and
		\begin{equation}\label{est_grad_3}
			\begin{aligned}
			\ell^2	\|\nabla \errh \cdot \nabla \pt^2 \projRitz \sol  \|^2_{\Ltwo}
				\lesssim 
				\norm{\nabla \errh  }^2_{\Ltwo}
				\norm{\nabla \pt^2  \projRitz \sol }^2_{\Linf}
				\lesssim 
				\norm{\pt^2 u}^2_{\Woneinf}	\norm{\nabla  \errh  }^2_{\Ltwo}.
			\end{aligned}
		\end{equation}
		The most salient point of the proof lies in estimating the term $
		\ell \ip{\nabla \solh \cdot \nabla \pt^2 \errh  }{\pt^2 \errh} $ in \eqref{ineq:error_test_pt2_errh}.  To bound this term, we
		split the scalar product into three components by involving the exact solution:
		\begin{equation} \label{est_nabla_ptsq_errh}
			\begin{aligned}
				&\ip{\nabla \solh \cdot \nabla \pt^2 \errh   }{\pt^2 \errh}\\
				&=\begin{multlined}[t]
				\ip{\nabla \sol \cdot \nabla \pt^2 \errh   }{\pt^2 \errh}
				-
				\ip{\nabla ( \sol- \projRitz \sol) \cdot \nabla \pt^2 \errh   }{\pt^2 \errh}
 			\\	-
				\ip{\nabla \errh \cdot \nabla \pt^2 \errh   }{\pt^2 \errh}. \end{multlined}
			\end{aligned}
		\end{equation}	
		We can then	use the fact that \[ \pt^2 \errh  \nabla \pt^2 \errh = \frac12 \nabla (\pt^2 \errh)^2 \]
		and integration by parts to rewrite the first term on the right-hand side of \eqref{est_nabla_ptsq_errh} as
		\begin{align}
			\ip{\nabla \sol \cdot \nabla \pt^2 \errh   }{\pt^2 \errh}
			=
			- \frac12 \ip{\Delta \sol \, \pt^2 \errh   }{\pt^2 \errh}.
		\end{align}
		Employing H\"older's and Young's inequalities in \eqref{est_nabla_ptsq_errh} then yields	
		\begin{equation} \label{eq:nonlinear_term_v1}
			\begin{aligned}
				&\ip{\nabla \solh \cdot \nabla \pt^2 \errh   }{\pt^2 \errh} \\
				&= \begin{multlined}[t]
				- \frac12 \ip{\Delta \sol \, \pt^2 \errh   }{\pt^2 \errh}
				-
				\ip{\nabla ( \sol- \projRitz \sol) \cdot \nabla \pt^2 \errh   }{\pt^2 \errh}
			\\	-
				\ip{\nabla \errh \cdot \nabla \pt^2 \errh   }{\pt^2 \errh} \end{multlined}
				\\
				&\lesssim \begin{multlined}[t]
					\norm{\Delta \sol}_{\Linf}   \norm{\pt^2 \errh}_{\Ltwo}^2
					+
					\bigl( \norm{ \sol- \projRitz \sol}_{\Woneinf}
					\bigr)
					\norm{\nabla \pt^2 \errh}_{\Ltwo}
					\norm{\pt^2 \errh}_{\Ltwo}\\-
					\ip{\nabla \errh \cdot \nabla \pt^2 \errh   }{\pt^2 \errh} \end{multlined}
				\\
				&\lesssim
				\bigl( 
				\norm{\Delta \sol}_{\Linf} 
				+
				h^{-1} \norm{ \sol- \projRitz \sol}_{\Woneinf}
				\bigr)
				\norm{\pt^2 \errh}_{\Ltwo}^2-
				\ip{\nabla \errh \cdot \nabla \pt^2 \errh   }{\pt^2 \errh},
			\end{aligned}
		\end{equation}
		where we have also used the inverse estimate \eqref{eq:inv_estimate} on 
		$\norm{\nabla \pt^2 \errh}_{\Ltwo}$
		 in the last line. 
	After integrating in time, we can estimate the last term on the right-hand side of \eqref{eq:nonlinear_term_v1} as follows:
		\begin{align}
			&\int_0^t
			\norm{\errh(s)}_{\Woneinf}
			\norm{\nabla \pt^2 \errh(s)}_{\Ltwo}
			\norm{\pt^2 \errh(s)}_{\Ltwo}\ds\\
			&\lesssim
			h^{-1-d/6}
			\int_0^t
			\norm{\errh(s)}_{W^{1,6}(\Om)}
			\norm{\pt^2 \errh(s)}_{\Ltwo}
			\norm{\pt^2 \errh(s)}_{\Ltwo}
			\ds
			\\
			&\lesssim
			\bigl(
			\max_{s\in [0,\finalth]}
			h^{-1-d/6}
			\norm{\pt^2 \errh(s)}_{\Ltwo} \bigr)
			\int_0^t
			\norm{\Deltah \errh(s)}_{\Ltwo}
			\norm{\pt^2 \errh(s)}_{\Ltwo}
			\ds
			\\
			&\leq
			\alpha 
			\int_0^t
			\norm{\Deltah \errh(s)}_{\Ltwo}^2
			\ds
			+
			C_\alpha
			\int_0^t
			\norm{\pt^2 \errh(s)}_{\Ltwo}^2
			\ds
		\end{align}
		for any $\alpha>0$, where we used the definition of $\finalth$. 
	From \eqref{eq:nonlinear_term_v1},	relying also on the estimate
		\begin{align}
			h^{-1} \norm{\sol(t)- \projRitz \sol(t)}_{\Woneinf} &\lesssim h^{-1}h \norm{\sol(t)}_{W^{2,\infty}(\Om)} 
			\lesssim \norm{\sol(t)}_{W^{2,\infty}(\Om)} , \quad t \in [0,T],
		\end{align}
		(which holds by \eqref{eq:projRitz_approx}),
		we then have
			\begin{equation} 
			\begin{aligned}
				&\ell \intt \ip{\nabla \solh \cdot \nabla \pt^2 \errh   }{\pt^2 \errh} \ds \\
				&\lesssim \begin{multlined}[t]
		\intt	\norm{ \sol(s)}_{W^{2,\infty}(\Om)} 	\norm{\pt^2 \errh(s)}_{\Ltwo}^2 \ds+
\alpha				\intt
				\norm{\Deltah \errh(s)}_{\Ltwo}^2
				\ds
					\\	+
				C_\alpha
				\int_0^t
				\norm{\pt^2 \errh(s)}_{\Ltwo}^2
				\ds 
				\end{multlined}
			\end{aligned}
		\end{equation}
	for any $\alpha>0$. Integrating over $(0,t)$ in \eqref{ineq:error_test_pt2_errh} and using this estimate together with \eqref{est_grad_1}--\eqref{est_grad_3} yields \eqref{est_eh_Step1}.
	\end{proof}

	Note that we cannot yet control the $\Delta_h \errh$ term on the right-hand side of \eqref{est_eh_Step1}. Therefore, in the second step, we additionally test the error equation  \eqref{eq_errh} with $-\Delta_h \errh$.

	\begin{myproposition} \label{prop:kuznetsov_space_main_v2}
	Let the assumptions of Theorem~\ref{thm:kuznetsov_space_main} hold. For $t\in[0,\finalth]$, it holds
		\begin{equation}\label{final_uniformbound_errh}
			\begin{aligned}
				&\norm{\pt^2 \errh(t) }^2_{\Ltwo}
				+   \norm{\nabla \pt \errh(t)}^2_{\Ltwo}  
				+
				\beta 
				\norm{ \Deltah \errh(t)}_{\Ltwo}^2
				+
				\frac{c^2}{4}
				\int_0^t \norm{ \Deltah \errh(s)}_{\Ltwo}^2 \ds
				\\
				&\leq \begin{multlined}[t]
					C
					\norm{\pt^2 \errh(0) }^2_{\Ltwo}
					+
					C
					\int\limits_0^t \big(
					\norm{\nabla \pt \errh(s)  }^2_{\Ltwo}
					+
					\norm{\pt^2 \errh(s)}^2_{\Ltwo}
					+
					\norm{\pt \delta_h (s)}^2_{\Ltwo}
					\\+
					\norm{\delta_h(s)}_{\Ltwo}^2
					+
					\norm{\nabla \pt \errh(s)}_{\Ltwo}^2
					+
					\norm{\nabla \errh(s)}_{\Ltwo}^2 \big)\ds
				\end{multlined}
			\end{aligned}
		\end{equation}
		for all $t \in [0, \finalth]$, with a constant $C>0$ independent of $h$, $\finalth$, and $\beta$.
	\end{myproposition}
	
	\begin{proof}
		Testing the error equation \eqref{eq_errh} with $\phih = - \Delta_h \errh$, integrating over $(0,t)$, and using $\errh(0)=0$ yields
		\begin{equation} \label{eq_Deltah_errh}
			\begin{aligned}
			&c^2 \int_0^t \norm{ \Deltah \errh(s)}_{\Ltwo}^2 \ds
			+
			\beta 
			 \norm{ \Deltah \errh(t)}_{\Ltwo}^2\\
			&=\begin{multlined}[t]
			\int_0^t	\ip{(1+\kappa \pt \solh)\pt^2 \errh 
				+\kappa \pt \errh \pt^2 \projRitz \sol 
				+
				\ell \nabla \solh \cdot \nabla \pt \errh  + \nabla \errh \cdot \nabla \pt \projRitz \sol  
				-
		\delta_h}{\Deltah \errh} \ds. 
			\end{multlined}
			\end{aligned}
		\end{equation}
		We use Young's inequality to bound the right-hand side:
	\begin{equation} \label{ineq_Deltah_errh}
	\begin{aligned}
		&\begin{multlined}[t]
	\int_0^t	\ip{(1+\kappa \pt \solh)\pt^2 \errh 
				+\kappa \pt \errh \pt^2 \projRitz \sol 
				+
				\ell \nabla \solh \cdot \nabla \pt \errh  + \nabla \errh \cdot \nabla \pt \projRitz \sol  
				-
				\delta_h}{\Deltah \errh} \ds
		\end{multlined} \\
		&\leq \,\begin{multlined}[t] \frac{1}{2 c^2}\norm{	(1+\kappa \pt \solh)\pt^2 \errh 
			+\kappa  \pt \errh \pt^2 \projRitz \sol 
			+
			\ell \nabla \solh \cdot \nabla \pt \errh  + \nabla \errh \cdot \nabla \pt \projRitz \sol  
			-\delta_h}_{L^2_t(\Ltwo)}^2
		\\
		+
		\frac{c^2}{2} \norm{\Deltah \errh}^2_{L^2_t(\Ltwo)}
				\end{multlined}
	\end{aligned}
\end{equation}
 We can conclude similarly to before by using \eqref{est_grad_1}--\eqref{est_grad_3} that
 \begin{equation}
 	\begin{aligned}
 		&\norm{	(1+\kappa \pt \solh)\pt^2 \errh 
 			+\kappa  \pt \errh \pt^2 \projRitz \sol 
 			+
 			\ell \nabla \solh \cdot \nabla \pt \errh  + \nabla \errh \cdot \nabla \pt \projRitz \sol  
 			-\delta_h}_{L^2_t(\Ltwo)} \\
 		\lesssim&\, \norm{\pt^2 \errh}_{L^2_t(\Ltwo)}+ \norm{\nabla \pt \errh}_{L^2_t(\Ltwo)}+ \norm{\nabla \errh}_{L^2_t(\Ltwo)}+\norm{\delta_{h}}_{L^2_t(\Ltwo)}
 	\end{aligned}
 \end{equation}
 for $t \in [0, \finalth]$. Using absorption via the $c^2$ term in \eqref{eq_Deltah_errh}, we arrive at 
		\begin{equation}  \label{eq:error2_test_Deltah_errh}
			\begin{aligned}
				&
				\frac{c^2}{2}
				\int_0^t \norm{ \Deltah \errh(s)}_{\Ltwo}^2 \ds+\beta 
				\norm{ \Deltah \errh(t)}_{\Ltwo}^2 
				\\
				&\lesssim \begin{multlined}[t]
				\int_0^t \Bigl(
					\norm{\pt^2 \errh(s)}_{\Ltwo}^2
					+
					\norm{\delta_h(s)}_{\Ltwo}^2 
					+
					\norm{\nabla \pt \errh(s)}_{\Ltwo}^2
					+
					\norm{\nabla \errh(s)}_{\Ltwo}^2 \Bigr)
					\ds.
				\end{multlined}
			\end{aligned}
		\end{equation}
		Then adding estimates \eqref{est_eh_Step1} and \eqref{eq:error2_test_Deltah_errh} and choosing $\alpha>0$ small enough (independently of $h$, $\finalth$, and $\beta$) so that the corresponding term can be absorbed by the left-hand side
		leads to \eqref{final_uniformbound_errh}.
	\end{proof}

	To show \eqref{goal_est_finalth}, it remains to estimate the defect terms on the right-hand side of \eqref{final_uniformbound_errh}.	
		
	\begin{lemma} \label{Lemma:Bound_defect} 
		Let the assumptions of Theorem~\ref{thm:kuznetsov_space_main} hold. On $[0,\finalth]$, the defect satisfies the following bounds:
		\begin{equation} \label{estimates_deltah}
			\begin{aligned}
				\norm{\delta_h}_{L^2_t(\Ltwo)}  \leq&\, C(u) 
				\bigl(h^k+ \norm{\pt \errh}_{L^2_t(\Ltwo)} 
				+ 
				 \norm{\nabla \errh}_{L^2_t(\Ltwo)}
				  \bigr), 
				  \\
				\norm{\pt \delta_h}_{L^2_t(\Ltwo)} \leq&\,
				 \begin{multlined}[t]C(u)
				 	  \bigl( h^k+ \norm{\pt \errh}_{L^\infty_t(\Ltwo)}
				 	    + \norm{\pt^2 \errh}_{L^2_t(\Ltwo)} 
				 	  \\ +  \norm{\nabla \errh}_{L^\infty_t(\Ltwo)}
					+  
					\norm{\nabla \pt \errh}_{L^2_t(\Ltwo)}  \vphantom{h^k}
					\bigr),
				\end{multlined}
			\end{aligned}
		\end{equation}
		where $C(u)= C(1+ \|\sol\|_{H^{3}(W^{1,\infty}(\Om))} )\|u\|_{H^3(H^{k+1}(\Om))}$ does not depend on $h$ or $\beta$. 
	\end{lemma}
	
	\begin{proof}
		We can rewrite the equation for the defect in \eqref{def_defect} as follows:
		\begin{equation}
		\begin{aligned}
			 \ip{\delta_h}{\phih} 
			= 
			&\, \ip{( 1 + \kappa \pt \projRitz \sol) (\pt^2\projRitz \sol-\pt^2\sol)
				 +
				   \kappa( \pt  \projRitz \sol-  \pt \sol)\pt^2  \sol }{\phih}		
			\\
			&\,	-
				\kappa\ip{ \pt \errh \, \pt^2 \projRitz \sol}{\phih}	
		  +  \ip{\ell \nabla \projRitz \sol \cdot \nabla \pt \projRitz \sol  -\ell \nabla  \sol \cdot \nabla \pt  \sol}{\phih}
				\\		
				&\,  - \ell \ip{\nabla \errh \cdot \nabla \pt \projRitz \sol}{\phih}			+\ip{f-f_h}{\phih}. 
		\end{aligned}
	\end{equation}
	Using estimate \eqref{eq:projRitz_approx} for the Ritz projection several times and the assumption on $f-f_h$, we arrive at the first bound in \eqref{estimates_deltah}. \\
		%
		\indent	The time-differentiated defect solves
		\begin{equation}
			\begin{aligned}
			&\, \ip{\pt \delta_h}{\phih} 
				\\
				 &= 
			 	\ip{	( 1 + \kappa \pt \projRitz \sol) (\pt^3\projRitz \sol-\pt^3\sol) + \kappa \pt^2 \projRitz \sol (\pt^2\projRitz \sol-\pt^2\sol) 
			 		\\
					&+  \kappa( \pt^2  \projRitz \sol-  \pt^2 \sol)\pt^2  \sol +  \kappa( \pt \projRitz \sol-  \pt \sol)\pt^3  \sol -	\kappa \pt^2 \errh \, \pt^2 \projRitz \sol-\kappa \pt \errh \, \pt^3 \projRitz \sol 
					 \\
					&+	\ell \nabla \pt \projRitz \sol \cdot \nabla \pt \projRitz \sol  -\ell \nabla \pt \sol \cdot \nabla \pt  \sol	+	\ell \nabla \projRitz \sol \cdot \nabla \pt^2 \projRitz \sol  -\ell \nabla  \sol \cdot \nabla \pt^2  \sol
					\\
					&- \ell \nabla \pt \errh \cdot \nabla \pt \projRitz \sol 	- \ell \nabla \errh \cdot \nabla \pt^2 \projRitz \sol
					+\pt(f-f_h)
				}{\phih} 
			\end{aligned}
		\end{equation}
	for all $\phih \in V_h$. We can similarly bound $\pt \delta_h$ in $L^2(0,\finalth; \Ltwo)$  using the stability and approximation properties 
		\eqref{eq:projRitz_approx}
		of the Ritz projection to arrive at the second estimate in \eqref{estimates_deltah}. 
	\end{proof}
~\\[2mm]
	We now have all the ingredients to prove our first main result on the robust finite element bounds stated in Theorem~\ref{thm:kuznetsov_space_main}. \\
	
	\begin{proof}[Proof of Theorem~\ref{thm:kuznetsov_space_main}]
		Using the bounds derived in Lemma~\ref{Lemma:Bound_defect} on the defect in \eqref{final_uniformbound_errh}, together with
\[
\norm{\pt \errh}_{L^\infty_t(\Ltwo)} \leq \sqrt{T} \norm{\pt^2 \errh}_{L^2_t(\Ltwo)}, \quad  \norm{\nabla \errh}_{L^\infty_t(\Ltwo)} \leq \sqrt{T} \norm{\nabla \pt \errh}_{L^2_t(\Ltwo)}
\]
(since $\errh(0)=\pt \errh(0)=0$) and Lemma~\ref{lemma:est_errh_zero}, by employing Gr\"onwall's inequality, we arrive at
\begin{equation} \label{eq:error_bound}
	\norm{\pt^2 \errh(t) }^2_{\Ltwo}
	+  
	\norm{\nabla \pt \errh(t)}^2_{\Ltwo}  
	+
	\int_0^t \norm{ \Deltah \errh(s)}_{\Ltwo}^2 \ds
	\leq C 
	h^{2 k } 
\end{equation}
	for all $t \in [0, \finalth]$.	Therefore, estimate \eqref{goal_est_finalth} holds. In turn, we
		 conclude that for $h\leq h_0$ and $k\geq 2$, the following estimate holds:
		\begin{align}
		h^{-1-d/6} \max \{ \norm{\pt^2 \errh(\finalth)}_{\Ltwo}   , \norm{\nabla \pt \errh(\finalth)}_{\Ltwo}    \} < C_0.
		\end{align} 
		Along the previous lines of reasoning, we also have 	
		$|\kappa|\norm{\pt \solh(\finalth)}_{\Linf} < 1$. Altogether, we have shown that \[(\solh(\finalth), \pt \solh(\finalth), \pt^2 \solh(\finalth)) \in U_h\] which means that the solution can be extended beyond $\finalth$. Hence, we can conclude that $\finalth = T$.
		Since the constant in \eqref{eq:error_bound} does not depend on $\finalth$, the bound on $\errh$ is valid on $[0,T]$.
\\
	\indent	Then writing the overall error as  \[\sol - \solh = (\sol - \projRitz \sol ) + \errh,\]  and using the approximation property of the Ritz projection  \eqref{eq:projRitz_approx} yields the bound
	 in \eqref{FE_final_est}.
	\end{proof}

	\subsection{The inviscid limit of the finite element solutions}
	\label{sec:limit_space_discrete}
	
	By using the established $\beta$-uniform finite element error bound in Theorem~\ref{thm:kuznetsov_space_main}, we can also determine the asymptotic properties of the semi-discrete solution as $\beta \rightarrow 0$. The difference  $\solhbar=\solhzero-\solhbeta$ of the semi-discrete solutions $\solhzero$ of the undamped
	\begin{equation} 
	\ip{	( 1 +  \kappa \pt \solhzero ) \pt^2 \solhzero - c^2 \Delta_h \solhzero + \ell \nabla \solhzero \cdot \nabla \pt \solhzero 
	}{\phih}
	=
	 	\ip{f_h }{\phih}, 
	\end{equation}
	 and $\solhbeta$  of the damped problem
	\begin{equation} 
		\ip{	( 1 +  \kappa \pt \solhbeta ) \pt^2 \solhbeta - c^2 \Delta_h \solhbeta - \beta \Delta_h \pt \solhbeta + \ell \nabla \solhbeta \cdot \nabla \pt \solhbeta
			}{\phih}
	 	=
 		\ip{f_h }{\phih}
	\end{equation}
		solves
	\begin{equation} \label{diff_eq}
		\begin{aligned} 
	\ip{	( 1 +  \kappa \pt \solhzero ) \pt^2 \solhbar+ \kappa \pt \solhbar \pt^2 \solhbeta - c^2 \Delta_h \solhbar 
		+ 
		\ell \nabla \solhbar \cdot \nabla \pt \solhbeta 
		+
		&\ell \nabla \solhzero \cdot \nabla \pt \solhbar
		}{\phih}
		\\
		&=
-\beta  \ip{ \Delta_h \pt \solhbeta 	}{\phih}
\end{aligned}
	\end{equation}
for all $\phih \in V_h$	on $[0,T]$. We next prove the statement of Theorem~\ref{thm:FE_beta_limit} on the inviscid limit of the finite element solutions. \\

		\begin{proof}[Proof of Theorem~\ref{thm:FE_beta_limit}]
The proof follows by testing the difference equation \eqref{diff_eq} with $\pt \solhbar$. We can rely on the identity
	\begin{equation}
		\begin{aligned}
	&\int_0^t \ip{	( 1 +  \kappa \pt \solhzero ) \pt^2 \solhbar }{\pt \solhbar} \ds \\
	&= \frac12 \ip{	( 1 +  \kappa \pt \solhzero ) \pt \solhbar(t) }{\pt \solhbar(t)} -\frac12 \int_0^t \ip{	\kappa \pt^2 \solhzero  \pt \solhbar }{\pt \solhbar} \ds
	\end{aligned}
	\end{equation}
	and the estimate
		\begin{equation}
	\int_0^t \ip{\kappa \pt \solhzero \pt^2 \solhbeta}{\pt \solhbar} \ds \lesssim \norm{\pt^2 \solhzero}_{L^\infty(\Linf)} \norm{\pt \solhbar}^2_{L^2(\Ltwo)}.
	\end{equation}
	We note that thanks to the previous analysis and the assumptions on the exact solution the following uniform bound holds :
		\[
 \norm{\pt^2 \solhzero}_{L^\infty(\Linf)} \lesssim \norm{ \pt^2 \solzero}_{L^\infty(\Linf)}+h^{-d/2}\norm{\pt^2 \errhzero}_{L^\infty(\Ltwo)} \leq C.
	\]
To estimate the $\ell$ terms, we rely on a rewriting with the help of the exact solution and integration by parts in space:
	\begin{equation}
		\begin{aligned}
			&\int_0^t\ip{	\ell \nabla \solhbar \cdot \nabla \pt \solhbeta +\ell \nabla \solhzero \cdot \nabla \pt \solhbar}{\pt \solhbar} \ds
			 \\
			&=
			\begin{multlined}[t] \int_0^t
				 \Bigl\{\ip{	\ell \nabla \solhbar \cdot \nabla \pt \solhbeta}{\pt \solhbar} +\frac12  \ip{\ell \Delta  \solzero   \pt \solhbar}{\pt \solhbar} 
				  \\
				  -\ip{\ell \nabla (\solzero-\solhzero) \cdot \nabla \pt \solhbar}{\pt \solhbar} \Bigr\}
				  \ds. 
			  \end{multlined}
		\end{aligned}
	\end{equation}
We can further transform the last term on the right by decomposing it via the Ritz projection: 
\begin{align}
&\int_0^t \ip{\ell \nabla (\solzero-\solhzero) \cdot \nabla \pt \solhbar}{\pt \solhbar} \ds\\
&=
\int_0^t \ip{\ell \nabla (\solzero- \projRitz \solzero ) \cdot \nabla \pt \solhbar}{\pt \solhbar}
+
\ip{\ell \nabla \errhzero \cdot \nabla \pt \solhbar}{\pt \solhbar} \ds
\\
&\leq%
C h
\norm{\solzero}_{L^\infty(W^{2,\infty}(\Om))} h^{-1}
\| \pt \solhbar \|^2_{L^2(\Ltwo)}
+
\int_0^t  |	\ip{\ell \nabla \errhzero \cdot \nabla \pt \solhbar}{\pt \solhbar}| \ds. 
\end{align}
Furthermore, in the last term, we have for any $\alpha_1>0$
		\begin{align}
&	\int_0^t	|	\ip{\nabla \errhzero \cdot \nabla \pt \solhbar}{\pt \solhbar}| \ds\\
	&\lesssim
	\int_0^t
	h^{-1-d/6}
	\norm{\Delta_h \errhzero}_{\Ltwo}
	\norm{\pt \solhbar}_{\Ltwo}^2 \ds
	\\
	&\lesssim 
	\max\limits_{s\in[0,t]} 
	\norm{\pt \solhbar(s)}_{\Ltwo}
	h^{-1-d/6}
	\bigl( \int_0^t		
	\norm{\Delta_h \errhzero}_{\Ltwo}^2 \ds 
	\bigr)^{1/2}
	\bigl( \int_0^t		
	\norm{\pt \solhbar(s)}_{\Ltwo}^2 \ds 
	\bigr)^{1/2}
	\\
	&\leq
	\alpha_1  
	\max\limits_{s\in[0,t]} 
	\norm{\pt \solhbar(s)}_{\Ltwo}^2
	+
	C_\alpha  h^{k-1-d/6} \norm{\pt \solhbar(s)}_{L^2(\Ltwo)}^2,
	\end{align}
where we have used the uniform bound on $\|\Delta_h \errhzero\|_{L^2(\Ltwo)}$ by Theorem~\ref{thm:kuznetsov_space_main}. Thanks also to the (assumed) uniform bound on $ \|  \solzero\|_{L^\infty(W^{2,\infty}(\Om))}$, we arrive at an estimate of the form
	\begin{equation}\label{first_est_betalim}
		\begin{aligned}
			&\| \pt \solhbar (t)\|^2_{\Ltwo}+ \|\nabla \solhbar (t) \|^2_{\Ltwo} \\
			\lesssim&\, \begin{multlined}[t] \beta \, 
				\Bigl| \int_0^t \ip{ \nabla \pt \solhbeta}{\nabla \pt \solhbar } \ds 
				\Bigr|
				+
					\alpha_1  
			\max\limits_{s\in[0,t]} 
			\norm{\pt \solhbar(s)}_{\Ltwo}^2
			+
			\| \pt \solhbar \|^2_{L^2_t(\Ltwo)}
			+
			\| \nabla \solhbar \|^2_{L^2_t(\Ltwo)}. 
			\end{multlined}
		\end{aligned}
	\end{equation}
	Observe that we cannot absorb $\nabla \pt \solhbar$ by the left-hand side. In the $\beta$ term above we thus integrate by parts in time:
	\begin{equation}
	\begin{aligned}
	\beta  \Bigl| \int_0^t \ip{ \nabla \pt \solhbeta}{\nabla \pt \solhbar } \ds
	\Bigr| 
	= 	\beta  \Bigl| \ip{ \nabla \pt \solhbeta(t)}{\nabla \solhbar(t) } -\int_0^t \ip{ \nabla \pt^2 \solhbeta}{\nabla  \solhbar } \ds \Bigr|.
	\end{aligned}
	\end{equation}
	We can then rely on the uniform bounds
	\begin{equation} 	\label{eq:nabla_pt2_solhbeta}
		\begin{aligned}
			\norm{\nabla \pt \solhbeta}_{L^\infty(\Ltwo)} \lesssim& \norm{\nabla \pt \solbeta}_{L^\infty(\Linf)}+h^{-1}\norm{\pt \errh}_{L^\infty(\Ltwo)} \leq C, \\
	\norm{\nabla \pt^2 \solhbeta}_{L^2(\Ltwo)} \lesssim& \norm{\nabla \pt^2 \solbeta}_{L^2(\Linf)}+h^{-1}\norm{\pt^2 \errh}_{L^2(\Ltwo)} \leq C,
	\end{aligned}
	\end{equation}
and Young's inequality to obtain
		\begin{equation}\label{second_est_betalim}
	\begin{aligned}
	&\| \pt \solhbar (t)\|^2_{\Ltwo}+ \|\nabla \solhbar (t) \|^2_{\Ltwo} \\
	\lesssim&\, \begin{multlined}[t] \beta^2 +	\alpha_1  
	\max\limits_{s\in[0,t]} 
	\norm{\pt \solhbar(s)}_{\Ltwo}^2+	\alpha_2  
	\max\limits_{s\in[0,t]} 
	\norm{\nabla \solhbar(s)}_{\Ltwo}^2
	\\
	+
	\| \pt \solhbar \|^2_{L^2_t(\Ltwo)}+\| \nabla \solhbar \|^2_{L^2_t(\Ltwo)}
	\end{multlined}
	\end{aligned}
	\end{equation}
	for any $\alpha_1$, $\alpha_2>0$.
From \eqref{second_est_betalim} by taking the maximum over $t \in (0, \tilde{t})$ for some $\tilde{t} < T$, choosing $\alpha_{1,2}>0$ small enough (independently of $h$ and $\beta $) so that the corresponding terms can be absorbed, and then applying Gr\"onwall's inequality, we arrive at
	\begin{equation}
\begin{aligned}
\| \pt \solhbar \|^2_{L^\infty(\Ltwo)}+ \|\nabla \solhbar  \|^2_{L^\infty(\Ltwo)} \lesssim \beta^2,
\end{aligned}
\end{equation}
as claimed.
\end{proof}

We note that this result matches the convergence order of the exact solutions of the damped Kuznetsov equation as $\beta \rightarrow  0$; see~\cite[Thm.\ 7.1]{kaltenbacher2022parabolic}.

\section{Robust semi-implicit time discretization} \label{sec:Semi_implicit}

We now turn our attention to the analysis of a fully discrete scheme given by \eqref{eq:Euler} with the aim of proving Theorems~\ref{thm:ErrorSemiImplicit} and~\ref{thm:FullyDiscrete_beta_limit}. 
We first collect several useful results when using the discrete derivatives $\ptau$ defined in \eqref{eq:def_ptau}.
We state relations
which mimic the product rule:
\begin{equation} \label{eq:product_rule_tau}
\begin{aligned} 
	\ptau \bigl( a^{n+1} b^{n+1} \bigr) &= 
	\bigl( \ptau a^{n+1} \bigr)  b^{n+1} + a^{n}  \bigl(  \ptau b^{n+1} \bigr)
	\\
	&= 	 \bigl( \ptau a^{n+1} \bigr)  b^{n} + a^{n+1}  \bigl(  \ptau b^{n+1} \bigr),
\end{aligned}
\end{equation}
the integration by parts formula:
\begin{align} \label{eq:sum_by_parts}
	 \tau \sum\limits_{n=1}^{N}   a^{n+1} \, \ptau b^{n+1}
	&=
	a^{N+1} b^{N+1} - a^{1} b^{1}
-
 \tau \sum\limits_{n=1}^{N} \, \ptau  a^{n+1}  b^{n+1} ,
\end{align}
as well as  the fundamental theorem of calculus:
\begin{equation} \label{eq:ftc_Euler_simple}
	\norm{a^N}_{\Ltwo}^2 - \norm{a^0}_{\Ltwo}^2 \leq 2 \tau \sum\limits_{j=1}^N \ip{  a^j }{ \ptau a^j};
\end{equation}
see, e.g., \cite{HocP17}. 
However due to the nonlinear structure in the highest order term, we need the following extension of  \eqref{eq:ftc_Euler_simple}.
\begin{lemma} \label{lem:ftc_Euler} 
	Let $\max\limits_{j=1,\ldots ,N }\norm{\omega_j}_{\Linf} + \norm{\ptau \omega_j}_{\Linf} \leq C_{\omega}$ be uniformly bounded from above and from below with $\min\limits_{j=1,\ldots ,N } \omega_j \geq \alpha >0$. Then, it holds
\begin{equation} \label{eq:ftc_Euler}
	\norm{a^N}_{\Ltwo}^2 
	\lesssim  \norm{a^0}_{\Ltwo}^2 
	+  \tau \sum\limits_{j=1}^N \ip{ \omega_j  a^j }{  \ptau a^j} 
	+
		\norm{a^j}_{\Ltwo}^2  ,
\end{equation}
with constants that only depend on $C_{\omega}$ and $\alpha$.
\end{lemma}

\begin{proof}
	Simply setting $\widetilde{a}_j = \sqrt{\omega_j} a_j$  in \eqref{eq:ftc_Euler_simple} and using \eqref{eq:product_rule_tau} leads to the claim.
\end{proof}
~\\[2mm]
In addition, we need the following error bounds in the defects.
We state the result here in a general form, but postpone the proof to the Appendix~\ref{sec:appendix}.
\begin{lemma} \label{lem:diff_ptau_pt}
	Let $m\geq 0$ and $\sol \in H^{m+1}(0,T;H^k(\Omega))$. Then, for $0 \leq \ell \leq m$, it holds
	\begin{equation}
	\norm{\ptau^\ell \sol(\tn{n})}_{H^k(\Omega)} \leq C \norm{\sol}_{H^{\ell+1}(H^k(\Omega))} 
\end{equation}
with a constant $C$ independent of $\tau$, 
and for $0\leq \ell_1,\ell_2 \leq m$ it holds
	\begin{align}
		&\, \norm{ \ptau^{m-\ell_1} \pt^{\ell_1} \sol(\tn{n}) - \ptau^{m-\ell_2 } \pt^{\ell_2} \sol(\tn{n})
			}_{H^k(\Omega)}^2
		\leq 
		C_m \,
		\tau 
		\int_{\tn{n-m}}^{\tn{n}}
		\norm{
			\pt^{m+1} \sol(s)   
		}_{H^k(\Omega)}^2
		\ds 
	\end{align}
	with a constant $C_m$ that only depends on $m$. If $\sol \in H^{m+2}(0,T;H^k(\Omega))$,
	then we further have
	\begin{align}
	&\, \norm{ \ptau^{m-\ell_1} \pt^{\ell_1} \sol(\tn{n}) - \ptau^{m-\ell_2 } \pt^{\ell_2} \sol(\tn{n})
	}_{H^k(\Omega)}
	\leq 
	 C \tau \norm{\sol  }_{H^{m+2}(H^k(\Omega))} .
\end{align}	
\end{lemma}
~\\[2mm]
These identities will be used throughout the proofs in the following section. 
Analogously to the approach in the finite element analysis,	we
define the fully discrete error by
\begin{equation} \label{eq:def_full_discrete_error}
	\errhn{n} \coloneqq \projRitz \soln{n} - \solhn{n}
\end{equation}
with $\soln{n} = \sol(\tn{n})$, and proceed to investigate it.

\begin{myproposition} \label{prop:intermediate_step}
	Under the assumptions of Theorem~\ref{thm:ErrorSemiImplicit},
	for all $n = 2, \ldots  N+1$,
 the approximation $\solhn{n}$ defined in \eqref{eq:Euler} exists
  and the error defined in \eqref{eq:def_full_discrete_error} satisfies  
	\begin{equation} \label{eq:bounds_intermediate}
		\begin{aligned}
			\norm{ \ptau^2 \errhn{n}}^2_{\Ltwo}
			&+  
			\norm{\nabla   \ptau  \errhn{n} }^2_{\Ltwo}  
			+
			\tau \sum_{j=1}^{n}
			\norm{ \Delta_h  \errhn{j} }^2_{L^2(\Omega)}  
			\leq C  \bigl( \tau + h^k \bigr)^2,
		\end{aligned}
	\end{equation}
as well as
\begin{equation} \label{def:finaln}
	\begin{aligned}
		&\, h^{-1-d/6 - \varepsilon} \norm{\ptau^2 \errhn{n}}_{\Ltwo} \leq C_0,  \\
		&\, h^{-1-d/6 - \varepsilon} \norm{\nabla \ptau \errhn{n}}_{\Ltwo} 
		\leq C_0 ,
		\\
		&\, h^{-1-d/6 - \varepsilon} \bigl( \tau \sum_{j=1}^{n} \norm{ \Delta_h \errhn{j}}_{\Ltwo}^2 \bigr)^{1/2} 
		\leq C_0, 
	\end{aligned}
\end{equation}
with some constants $C$, $C_0 >0$ that are independent of $h$, $\tau$, $n$, and $\beta$.
\end{myproposition}

The rest of Section~\ref{sec:Semi_implicit} is devoted to the proof of Proposition~\ref{prop:intermediate_step} which we conduct via induction over $n$.
In Section~\ref{subsec:induction_base}, we first show that the statement holds in the case $n=2$
as induction basis.
In the following Section~\ref{subsec:induction_step},
we perform the step from $n$ to $n+1$ to conclude that the statement of Proposition~\ref{prop:intermediate_step} holds. Then Theorem~\ref{thm:ErrorSemiImplicit} will follow in a straightforward manner.

\subsection{Induction basis}
\label{subsec:induction_base}

This part is dedicated to the induction base $n=2$. We first study the error induced by the initial values for $\errhn{1}$, and then proceed to bound $\errhn{2}$ in a series of lemmas.
To keep the presentation short, we formulate several of them 
such that they also apply to the induction step, assuming that the bounds in Proposition~\ref{prop:intermediate_step} already hold up to some $n\geq 2$.

	\begin{lemma} \label{lem:errhn1}
	Let the assumptions of Theorem~\ref{thm:ErrorSemiImplicit} hold,
	and let the initial values be defined by \eqref{eq:def_Euler_inits}.
Then, $\errhn{0} = 0$ and
	\begin{align}
	 \norm{\ptau^2 \errhn{1}}_{\Ltwo}
	 +
	 	\tau^{-1} \norm{\ptau \nabla \errhn{1}}_{\Ltwo}
		+
		\tau^{-2} \norm{\Delta_h \errhn{1} }_{\Ltwo}	
		&\lesssim
		\tau + h^k, 
	\end{align}	
	where the 
	 constant is independent of $h$, $\tau$ and $\beta$.
\end{lemma}
	
\begin{proof}
	Recalling that by \eqref{eq:def_ptau} we have
	$\ptau^2 \errhn{1} = \ptau \errhn{1}$,
	the estimate directly follows from the definitions of $\solhn{0}$, $\solhn{1}$, and $w_0$ in \eqref{eq:def_Euler_inits}.
\end{proof}
~\\[2mm]
Along the lines of Lemma~\ref{lem:a-priori_bounds_on_solh} in Section~\ref{sec:FE_analysis},
we next derive some useful bounds on the numerical solution.

	\begin{lemma} \label{Lemma:UniformEst_solhn}
	Let the assumptions of Theorem~\ref{thm:ErrorSemiImplicit} hold.
	 If the assertions of Proposition~\ref{prop:intermediate_step} hold for 
	up to $n$, then the following bound holds for $j=1,\ldots, n$: 
	\begin{align}
		\norm{\solhn{j}}_{W^{1,\infty}(\Om)}
		+
		\norm{\nabla \ptau \solhn{j}}_{\Linf}
		+
		\norm{\ptau^2 \solhn{j}}_{\Linf}
		\lesssim 1,
	\end{align}
	and, in addition, 
	\begin{equation} \label{gamma_bound_Euler}
		1 + \kappa \ptau \solhn{j} \geq 
		\quasilowerbound > 0,  \quad j = 1 ,\ldots n,
	\end{equation}
	where the constant $\gamma$ does not depend on $h$, $\tau$, $n$,
	or $\beta$. 
\end{lemma}

In the next steps, we derive the equation solved by $\errhn{n+1}$.
To this end, we insert the projected exact solution $\projRitz \soln{n}$ into \eqref{eq:Euler} to obtain for $n \geq 1$
	\begin{align} 
	\ip{	(1+\kappa \ptau \solhn{n}) \ptau^2 \projRitz \soln{n+1}  
		- c^2 \Delta_h \projRitz \soln{n+1} 
		-\beta \Delta_h  \ptau \projRitz \soln{n+1} 
		&+ \ell \nabla \solhn{n} \cdot \nabla \ptau \projRitz \soln{n+1} 
	}{\varphi_h}
	\\
		&=
		\ip{ f_h^{n+1} + \delta_h^{n+1}
	}{\varphi_h} ,
	\end{align}
	with defect $ \delta_h^{n+1}$ given below in \eqref{eq:def_defect_euler}.
	This leads us to the error equation:
	\begin{equation} \label{eq:error_Euler}
	\begin{aligned}
	\ip{	(1+\kappa \ptau \solhn{n}) \ptau^2 \errhn{n+1}  
		- c^2 \Delta_h \errhn{n+1} 
		-\beta \Delta_h  \ptau \errhn{n+1} 
		&+ \ell \nabla \solhn{n} \cdot \nabla \ptau\errhn{n+1}
			}{\varphi_h}
		\\
		&=\ip{ \delta_h^{n+1}
			}{\varphi_h}
		\end{aligned}
	\end{equation}
for $n \geq 1$. In the fully discrete case, we cannot use a Picard--Lindel\"of theorem,
and hence we explicitly have to show
the existence of the approximation $\solhn{n+1}$.
By \eqref{eq:def_full_discrete_error}, it is sufficient to show the unique solvability
of
\eqref{eq:error_Euler} or, in other words, existence of a unique $\errhn{n+1}$. 
By multiplying \eqref{eq:error_Euler} with $\tau^2$
and solving for $\errhn{n+1}$, we rewrite the problem as a linear system of the form
\begin{subequations} \label{eq:recursion_for_erhn}
\begin{align+}
	\ip{  \resolvent{n} \errhn{n+1} }{ \varphi_h}  = \ip{ \widetilde{f}^n }{ \varphi_h} ,
\end{align+}
where
%
	\begin{align+}
		\resolvent{n}  &= (1+\kappa \ptau \solhn{n}) \Id
		-
		c^2 \tau^2  \Delta_h
		- 
		\tau  \beta \Delta_h    
		+
		\tau  \ell \nabla \solhn{n} \cdot \nabla ,
		\\
		\widetilde{f}^n &= 
		2 (1+\kappa \ptau \solhn{n}) \errhn{n}  
		+
		(1+\kappa \ptau \solhn{n}) \errhn{n-1}  
		-
		\tau \beta \Delta_h  \ptau \errhn{n}  
		+
		\tau  \ell \nabla \solhn{n} \cdot \nabla \ptau\errhn{n}
		+
		\tau^2 \delta_h^{n+1},
	\end{align+}
\end{subequations}
and $\Id$ is the identity operator. 
This rewriting enables us to prove the following existence result.

\begin{lemma} \label{lem:existence_Euler_approximation}
	Let the assumptions of Theorem~\ref{thm:ErrorSemiImplicit} hold.
	
	(a) There exists a unique solution $\solhn{2}$
	of \eqref{eq:Euler} for $n=2$.
	
	(b) If the assertions of Proposition~\ref{prop:intermediate_step} hold for 
	up to $n$, then there
	exists a unique solution $\solhn{n+1}$
	of \eqref{eq:Euler}.
	%
\end{lemma}

\begin{proof}
	Since we consider a finite dimensional solution space, it is sufficient to show injectivity of $\resolvent{n}$.
	We only present the proof of part (b)
	as part (a) can be proven along the same lines.
	For $\varphi_h \in \Vh$, by using the lower bound in \eqref{gamma_bound_Euler}, we compute
	\begin{align}
		&\, \ip{\resolvent{n} \varphi_h}{\varphi_h }
		\\
		= &\,  \ip{(1+\kappa \ptau \solhn{n}) \varphi_h  }{ \varphi_h }
		+
		\bigl( \tau^2 c^2 + \tau \beta \bigr) \norm{ \nabla \varphi_h }_{\Ltwo}^2
		+
		\tau  \ell
		\ip{ \nabla \solhn{n} \cdot \nabla \varphi_h }{\varphi_h }
		\\
		\geq
		&\, \gamma \norm{ \varphi_h }_{\Ltwo}^2
		+
		\bigl( \tau^2 c^2 + \tau \beta \bigr) \norm{ \nabla \varphi_h}_{\Ltwo}^2
		-
		\tau \ell
		| \ip{ \nabla \solhn{n} \cdot \nabla \varphi_h }{ \varphi_h } | \,,
	\end{align}
	and thus it holds
	\[
	\begin{aligned}
		\gamma \norm{ \varphi_h }_{\Ltwo}^2
		+
		\bigl( \tau^2 c^2 + \tau \beta \bigr) \norm{ \nabla \varphi_h}_{\Ltwo}^2 \leq \ip{\resolvent{n} \varphi_h}{\varphi_h }+	\tau \ell
		| \ip{ \nabla \solhn{n} \cdot \nabla \varphi_h }{ \varphi_h } |.
	\end{aligned}
	\]
	We expand the last term
	and rely on inverse estimates \eqref{eq:inv_estimate} and the discrete embedding
	\eqref{eq:discrete_Sobolev_embedding} to obtain
	\begin{equation} \label{eq:expansion_critical_term}
		\begin{aligned} 
			&\, \tau | \ip{ \nabla \solhn{n} \cdot \nabla \varphi_h }{ \varphi_h }  |
			\\
			=
			&\, \tau  | -\ip{ \Delta \soln{n} \varphi_h }{ \varphi_h }
			+
			\ip{ \nabla \bigl( \projRitz \soln{n} -  \soln{n}) \nabla  \varphi_h }{ \varphi_h }
			-
			\ip{\nabla \errhn{n}  \nabla \varphi_h }{ \varphi_h } |
			%
			\\
			\leq 
			&\, C \tau
			\bigl( 1 + h^{k-1} \bigr)  \norm{ \varphi_h }_{\Ltwo}^2 
			+
			\tau  h^{-1-d/6} \norm{ \Delta_h \errhn{n} }_{\Ltwo}    \norm{ \varphi_h }_{\Ltwo}^2. 
		\end{aligned}
	\end{equation}
	We absorb the first term for $\tau$ sufficiently small,
	and estimate the second term with
	the $C_0$ bound in \eqref{def:finaln} and
	 the {CFL-type condition} \eqref{eq:def_CFL_Euler} 
	\begin{align}
		\tau h^{-1-d/6} \norm{ \Delta_h \errhn{n} }_{\Ltwo} &\leq
		\tau^{1/2} h^{-1-d/6}  \bigl( \tau \sum_{j=1}^n \norm{ \Delta_h \errhn{j}}_{\Ltwo}^2 \bigr)^{1/2} \leq C_0 h^{\varepsilon} \tau^{1/2} .
	\end{align}
	Hence, this term can also be absorbed, such that we obtain for some $\alpha >0$ 
	\begin{equation} \label{eq:Rn_inv_bound}
		\norm{\varphi_h}_{\Ltwo}^2 + \bigl( \tau^2 c^2 + \tau \beta \bigr) \norm{ \nabla \varphi_h }_{\Ltwo}^2 \leq \alpha \ip{\resolvent{n} \varphi_h}{\varphi_h }, 
	\end{equation}
	where $\alpha$ is independent of $h$, $\tau$, $n$, and $\beta$.
\end{proof}

The following lemma provides an estimate of the defect. Again here, we state it in its full generality to be used not only for proving the induction basis $n=2$ but also later for completing the induction step.
	\begin{lemma} \label{Lemma:Bound_defect_Euler}
	Let the assumptions of Theorem~\ref{thm:ErrorSemiImplicit} hold. 
	If the assertions of Proposition~\ref{prop:intermediate_step} hold for 
	up to $n$, then
		\begin{align} 
		\tau \sum\limits_{j=1}^{n}	\norm{\delta_h^{j+1}}_{\Ltwo}^2 
			&\leq\,
			 C(u,f) \bigl(  h^k+ 	\tau \bigr)^2
			 +
			\tau \sum\limits_{j=1}^n 
			\bigl(
			\norm{\ptau \errhn{j}}_{\Ltwo}^2 + \norm{\nabla \errhn{j}}_{\Ltwo}^2
			\bigr) ,
		\end{align}
		and 
		%
		\begin{align} 
		\tau \sum\limits_{j=1}^{n} \norm{\ptau \delta_h^{j+1}}_{\Ltwo}^2 
		&\leq\,
		C(u,f) \bigl(  h^k+ 	\tau \bigr)^2
		\\
		&+
		\tau \sum\limits_{j=1}^{n}
		\bigl(
		\norm{\ptau \errhn{j}}_{\Ltwo}^2 
		+ 
		\norm{\ptau^2 \errhn{j}}_{\Ltwo}^2
		+
		\norm{ \nabla \errhn{j}}_{\Ltwo}^2 
		+ 
		\norm{\ptau \nabla \errhn{j}}_{\Ltwo}^2
		\bigr)
		\end{align}
	with constants independent of $h$, $\tau$, $n$, and $\beta$.
	\end{lemma}

\begin{proof}
It is straightforward to check that the defect in \eqref{eq:error_Euler} is given by
		%
%
		\begin{equation} \label{eq:def_defect_euler}
	\begin{aligned} 
		\delta_h^{n+1} 	
		&= 	(1+\kappa \ptau  \projRitz \soln{n} ) \ptau^2 \projRitz \soln{n+1}  
			-
			(1+\kappa \pt \soln{n+1}) \pt^2 \soln{n+1}  
			-
			\kappa \ptau \errhn{n} \ptau^2 \projRitz \soln{n+1} 
		\\
		&+
		 \beta \bigl( 
			\Delta  \pt \soln{n+1} 
			-
			\Delta \ptau  \soln{n+1} 
			\bigr)
			+
			f(\tn{n+1}) - f_h^{n+1} 
		\\
		&+ 
		\ell \nabla \projRitz \soln{n}  \cdot \nabla \ptau \projRitz \soln{n+1} 
			-
			\ell \nabla \soln{n+1} \cdot \nabla \pt \soln{n+1}
			-
			\ell \nabla \errhn{n} \cdot \nabla \ptau \projRitz \soln{n+1} .	\end{aligned}	
\end{equation}
	Most of the terms were already estimated in Lemma~\ref{Lemma:Bound_defect}.
	Using also the estimates provided in Lemma~\ref{lem:diff_ptau_pt}
%
and 
	the approximation properties of $f_h^n$ assumed in \eqref{assumption_accuracy_fh_full}, we obtain the first bound.
	Using the product rule \eqref{eq:product_rule_tau}
	several times, by the same strategy, we arrive at the second bound for 
	$\ptau \delta_h^{n+1}$.
\end{proof}

With this preparation, we can show that estimates \eqref{eq:bounds_intermediate} and \eqref{def:finaln} hold for $n=2$ and thus complete the induction basis. 

\begin{lemma} \label{lem:errhn2}
	Let the assumptions of Theorem~\ref{thm:ErrorSemiImplicit} hold,
	and let the initial values be defined by \eqref{eq:def_Euler_inits}.
	Then, the error $\errhn{2} $ satisfies the following bounds
 	\begin{equation} \label{est:uhtwo}
		\begin{aligned}	
		\norm{\ptau^2 \errhn{2}}_{\Ltwo}
			+
			\norm{\ptau \nabla \errhn{2}}_{\Ltwo}
			+
			\norm{\Delta_h \errhn{2}}_{\Ltwo}
			&\lesssim
			\tau + h^k  \,,
		\end{aligned}
		\end{equation}
	with a constant independent of $h$, $\tau$, and $\beta$.
	\end{lemma}
	
	\begin{proof}
	We employ the estimate \eqref{eq:Rn_inv_bound}  in Lemma~\ref{lem:existence_Euler_approximation} and 
	use $n=1$ in \eqref{eq:recursion_for_erhn} to obtain 
		\begin{align}
	\norm{\errhn{2}}_{\Ltwo} + \tau \norm{\nabla \errhn{2}}_{\Ltwo}  
	&\lesssim \norm{\resolvent{1} \errhn{2}}_{\Ltwo}
	\\
		&\lesssim \begin{multlined}[t]
		\norm{ \errhn{1}  }_{\Ltwo} 
		+
		\norm{ \errhn{0}	}_{\Ltwo} 
		+
		\tau \beta	\norm{		 \Delta_h  \ptau \errhn{1}  }_{\Ltwo} 
		\\
		+
		\tau   	\norm{ \nabla \ptau\errhn{1} }_{\Ltwo} 
		+
		\tau^2 \norm{	 \delta_h^{2}	}_{\Ltwo}  \end{multlined}
		\\
		&\lesssim \tau^2 \bigl( \tau + h^k) ,
	\end{align}
	where we have used $\errhn{0} = 0$, Lemma~\ref{lem:errhn1},
	and for estimating the defect, Lemma~\ref{lem:diff_ptau_pt} with $0\leq m\leq 2$.
%
The first two terms can then  be bounded using
\begin{align}
	\norm{\ptau \nabla \errhn{2}}_{\Ltwo} 
	&\leq \tau^{-1}
	\bigl(
	\norm{ \nabla\errhn{2}}_{\Ltwo}
	+
	\norm{ \nabla \errhn{1}}_{\Ltwo}
	\bigr), 
	\\
	\norm{\ptau^2 \errhn{2}}_{\Ltwo} 
	&\leq \tau^{-2}
	\bigl(
	\norm{ \errhn{2}}_{\Ltwo}
	+
	2 \norm{ \errhn{1}}_{\Ltwo}
	+
	\norm{ \errhn{0}}_{\Ltwo}
	\bigr).
\end{align}
Using the inverse estimate \eqref{eq:inv_estimate} and the CFL-type condition \eqref{eq:def_CFL_Euler},
we additionally have
\begin{equation}
\norm{\Delta_h \errhn{2}}_{\Ltwo}
 \lesssim 
 h^{-2} \norm{\errhn{2}}_{\Ltwo}
 \lesssim
  \tau^{-2} \norm{ \errhn{2}}_{\Ltwo} ,
\end{equation}
and conclude the desired bound.
\end{proof}

We thus conclude that estimate \eqref{eq:bounds_intermediate} holds for $n=2$. Estimates \eqref{def:finaln} for $n=2$ then directly follow by exploiting the CFL-type condition in \eqref{eq:def_CFL_Euler}. Altogether, we conclude that the statement of Proposition~\ref{prop:intermediate_step} holds for $n=2$.
\subsection{Completing the induction step}
\label{subsec:induction_step}

%
We next perform the induction step needed to prove estimates \eqref{eq:bounds_intermediate} and \eqref{def:finaln}. 
Note that by Lemma~\ref{lem:existence_Euler_approximation},
we have already shown the existence of $\solhn{n+1}$.
To prove \eqref{eq:bounds_intermediate}, we proceed similarly to
Propositions~\ref{prop:fe_FirstEstimate}
and \ref{prop:kuznetsov_space_main_v2}
in two testing steps. First, we test the equation for $\errhn{j+1}$ with $-\Delta_h \errhn{j+1}$
in
Proposition~\ref{prop:Deltah_errh_Euler},
and then test the discretely differentiated version with 
$\ptau^2 \errhn{j+1}$
in Proposition~\ref{prop:pt2_errh_Euler}.

	\begin{myproposition} \label{prop:Deltah_errh_Euler}
	Let the assumptions of Theorem~\ref{thm:ErrorSemiImplicit} hold.
	If the assertions of Proposition~\ref{prop:intermediate_step} hold 
	up to $n$, then
	%
	\begin{equation} \label{est:_errh}
	\begin{aligned}
	&\tau \sum\limits_{j=1}^n
	\norm{\Delta_h \errhn{j+1} }^2_{\Ltwo}
	+
	\beta \norm{\Delta_h \errhn{n+1} }^2_{\Ltwo} 
	\\
	&\lesssim
	\beta \norm{\Delta_h \errhn{1} }^2_{\Ltwo}
	+
	\tau  \sum\limits_{j=1}^n \bigl(
	\norm{\ptau^2 \errhn{j+1}  }^2_{\Ltwo}
	+
	\norm{\nabla \ptau\errhn{j+1}  }^2_{\Ltwo}
	+
	\norm{\delta_h^{j+1} }^2_{\Ltwo}\bigr),
\end{aligned}
\end{equation}
with constants independent of $h$, $\tau$, $n$ and $\beta$.
	\end{myproposition}

	\begin{proof}
	As announced, we test the error equation for $\errhn{j+1}$ with $\phih = - \Deltah \errhn{j+1}$ to obtain
		%
		\begin{align}
			&c^2 \norm{\Delta_h \errhn{j+1} }^2_{\Ltwo}
			+
			\beta \ip{\Delta_h  \ptau \errhn{j+1} }{\Delta_h \errhn{j+1} }
			\\
			&=
			\ip{ (1+\kappa \ptau \solhn{j}) \ptau^2 \errhn{j+1}  
			}{\Delta_h \errhn{j+1} }
		+ \ip{  \ell \nabla \solhn{j} \cdot \nabla \ptau\errhn{j+1} }
		{\Delta_h \errhn{j+1} }
			\\
			&\quad + \ip{\delta_h^{j+1}}{\Delta_h \errhn{j+1} }.
		\end{align}
	Note that since the assertions of Proposition~\ref{prop:intermediate_step} hold up to $n$, we can rely on the uniform bounds stated in Lemma~\ref{Lemma:UniformEst_solhn}. Thus, summing from $1$ to $n$, using Lemma~\ref{lem:ftc_Euler}
	as well as Young's inequality and the uniform bounds in Lemma~\ref{Lemma:UniformEst_solhn},
	leads to estimate \eqref{est:_errh}.
	%
%
\end{proof}

We next need a discretely differentiated version of 
	the error equation \eqref{eq:error_Euler},
	analogously to \eqref{eq:Kuznetsov_space_discr_no_kappa_pt}. We use
	the discrete product rule \eqref{eq:product_rule_tau} to obtain 
	\begin{align} \label{eq:error_Euler_ptau}
	&\, 	\ip{	(1+\kappa \ptau \solhn{n}) \ptau^3 \errhn{n+1}  
	+
	\kappa \ptau^2 \solhn{n}  \ptau^2 \errhn{n}  
	 - c^2 \Delta_h \ptau \errhn{n+1} 
	-\beta \Delta_h  \ptau^2 \errhn{n+1} 
	\\
	& \qquad \qquad
	+ \ell \nabla \solhn{n} \cdot \nabla \ptau^2\errhn{n+1}
	+ \ell \nabla \ptau \solhn{n} \cdot \nabla \ptau \errhn{n}
}{\phih}
	= 
	\ip{ \ptau\delta_h^{n+1} 
}{\phih},
\end{align}	
for $n\geq 2$. Further, we introduce the notation
\begin{equation}
	\normlinfLtwo{a_h^j}{1}{n} \coloneqq  \max\limits_{j=1,\ldots,n} \norm{a_h^j}_{\Ltwo} ,
\end{equation}
which allows us to formulate the next proposition.
	
	\begin{myproposition}  \label{prop:pt2_errh_Euler} 
		Let the assumptions of Theorem~\ref{thm:ErrorSemiImplicit} hold.
		If the assertions of Proposition~\ref{prop:intermediate_step} hold for up to $n$, then
		for 
		 any $\alpha >0$ 
		it holds 
		\begin{equation} \label{eq:est_pt2_errh_Euler}
		\begin{aligned}
			&\, \norm{\ptau^2 \errhn{n+1}}_{\Ltwo}^2
			+
			\norm{\ptau \nabla \errhn{n+1}}_{\Ltwo}^2
			+
			\tau \beta \sum\limits_{j=2}^n
			 \norm{\nabla \ptau^2 \errhn{j+1} }^2_{\Ltwo} 
			\\
			&\lesssim \begin{multlined}[t]
			 \alpha
		 	\, \normlinfLtwo{\ptau^2 \errhn{j}}{2}{n+1}^2
+
			\norm{\ptau^2 \errhn{2}}_{\Ltwo}^2
+
\norm{\ptau \nabla \errhn{2}}_{\Ltwo}^2
			\\
			 +
			\tau 	\sum\limits_{j=2}^{n} \bigl(
			\norm{\ptau^2 \errhn{j+1}  }^2_{\Ltwo}
			+
			\norm{\nabla \ptau\errhn{j+1}  }^2_{\Ltwo}
			+
			\norm{\ptau \delta_h^{j+1} }^2_{\Ltwo} \bigr), \end{multlined}
			\end{aligned}
		\end{equation}
	with constants independent of $h$, $\tau$, $n$ and $\beta$.
	\end{myproposition}

	\begin{proof}
		
		We test the discretely differentiated error for $\errhn{j+1}$ with $\varphi_h = \ptau^2 \errhn{j+1}$ to obtain
		\begin{align} 
		&\ \quad   \ip{(1+\kappa \ptau \solhn{j}) \ptau^3 \errhn{j+1}  }{\ptau^2 \errhn{j+1}}
		+
		c^2 \ip{\nabla \ptau \errhn{j+1} }{\nabla \ptau^2 \errhn{j+1} }    
		+	
		\beta \norm{\nabla \ptau^2 \errhn{j+1} }_{\Ltwo}^2
		\\
		&\leq \begin{multlined}[t]
		\kappa 
		 | \ip{ \ptau^2 \solhn{j}  \ptau^2 \errhn{j}}{\ptau^2 \errhn{j+1}}  |
		+ \ell 
		| \ip{\nabla \solhn{j} \cdot \nabla \ptau^2\errhn{j+1}}{\ptau^2 \errhn{j+1}} |
		\\
		 +\ell | \ip{   \nabla \ptau \solhn{j} \cdot \nabla \ptau \errhn{j}}{\ptau^2 \errhn{j+1}} |
		  + | \ip{  \ptau\delta_h^{j+1}}{\ptau^2 \errhn{j+1}} | \end{multlined}
		\\
		&\lesssim
		\norm{ \ptau^2 \errhn{j}}_{\Ltwo}^2
				+
		\norm{ \nabla \ptau \errhn{j+1}  }_{\Ltwo}^2
		+
		\norm{ \ptau\delta_h^{j+1} }_{\Ltwo}^2
		+
		|  \ell \ip{   \nabla \solhn{j} \cdot \nabla \ptau^2 \errhn{j+1}}{\ptau^2 \errhn{j+1}} |,
	\end{align}		
	where we have also used the uniform bounds on $\solhn{n}$ stated in Lemma~\ref{Lemma:UniformEst_solhn} in the last line.
	We sum these inequalities from $j = 2,\ldots,n$ and use Lemma~\ref{lem:ftc_Euler} to conclude that
	\begin{equation} \label{est:semi_implicit_diff}
		\begin{aligned}
		&\, \norm{\ptau^2 \errhn{n+1}}_{\Ltwo}^2
		+
		\norm{\ptau \nabla \errhn{n+1}}_{\Ltwo}^2
		+
		\tau \sum\limits_{j=2}^n
		\beta \norm{\nabla \ptau^2 \errhn{n+1} }^2_{\Ltwo} 
		\\
		\lesssim
		&\, \begin{multlined}[t] 
		\tau 	\sum\limits_{j=2}^{n} \bigl(
		\norm{\ptau^2 \errhn{j+1}  }^2_{\Ltwo}
		+
		\norm{\nabla \ptau\errhn{j+1}  }^2_{\Ltwo}
		+
		\norm{\ptau \delta_h^{j+1} }^2_{\Ltwo} \bigr)
	+\norm{\ptau^2 \errhn{2}}_{\Ltwo}^2\\
	+
	\norm{\ptau \nabla \errhn{2}}_{\Ltwo}^2	 +
		\tau 	\sum\limits_{j=2}^{n}
		|  \ip{   \nabla \solhn{j} \cdot \nabla \ptau^2 \errhn{j+1}}{\ptau^2 \errhn{j+1}} | .\end{multlined}
	\end{aligned}
	\end{equation}
It remains to bound the last term. To this end, we use the expansion from \eqref{eq:expansion_critical_term} to obtain
	\begin{align}
	&\, \tau 	\sum\limits_{j=2}^{n}
|  \ip{   \nabla \solhn{j} \cdot \nabla \ptau^2 \errhn{j+1}}{\ptau^2 \errhn{j+1}} | 
\\
\lesssim
&\, \tau 	\sum\limits_{j=2}^{n}  	\bigl(	\norm{\ptau^2 \errhn{j+1}  }^2_{\Ltwo}
+
\norm{\nabla \ptau\errhn{j+1}  }^2_{\Ltwo} \bigr)
+
h^{-1-d/6} \tau 	\sum\limits_{j=2}^{n} 
\norm{\Delta_h \errhn{j}}_{\Ltwo}
\norm{\ptau^2 \errhn{j+1}  }^2_{\Ltwo}.
\end{align}
Since the assertions of Proposition~\ref{prop:intermediate_step} hold up to $n$, by the $C_0$ bounds in \eqref{def:finaln}
and Young's inequality, we have
	\begin{align}
	&\, h^{-1-d/6} \tau 	\sum\limits_{j=2}^{n} 
	\norm{\Delta_h \errhn{j}}_{\Ltwo}
	\norm{\ptau^2 \errhn{j+1}  }^2_{\Ltwo}
	\\
	\lesssim
	&\, 
	\normlinfLtwo{\ptau^2 \errhn{j}}{2}{n+1}
	h^{-1-d/6} 
	\bigl(
	\tau 	\sum\limits_{j=2}^{n} 
	\norm{\Delta_h \errhn{j}}_{\Ltwo}^2
		\bigr)^{1/2} 
	\bigl(
	\tau 	\sum\limits_{j=2}^{n} 
	\norm{\ptau^2 \errhn{j+1}  }^2_{\Ltwo} \bigr)^{1/2} 
		\\
	\lesssim
	&\, 
	\alpha
	\normlinfLtwo{\ptau^2 \errhn{j}}{2}{n+1}^2
	+
	\tau 	\sum\limits_{j=2}^{n} 
	\norm{\ptau^2 \errhn{j+1}  }^2_{\Ltwo},
\end{align}	
where $\alpha>0$ can be chosen arbitrarily.	Employing this bound in \eqref{est:semi_implicit_diff} leads to \eqref{eq:est_pt2_errh_Euler}.
\end{proof}

We now combine all previous results in this section to arrive at the statement of Proposition~\ref{prop:intermediate_step}. 
%

\begin{proof}[Proof of Proposition~\ref{prop:intermediate_step}] 
Recall that the statement of Proposition~\ref{prop:intermediate_step} holds for $n=2$ by the results of Section~\ref{subsec:induction_base}.   We complete the induction step by showing the existence and proving
the estimates  \eqref{eq:bounds_intermediate} and \eqref{def:finaln}. Since the assertions in Proposition~\ref{prop:intermediate_step}
are assumed to hold up to $n$,
by Lemma~\ref{lem:existence_Euler_approximation}
we have existence of the solution $\solhn{n+1}$ of \eqref{eq:Euler}.
In addition,
	by
	Proposition~\ref{prop:Deltah_errh_Euler}
	and
	Proposition~\ref{prop:pt2_errh_Euler}, we have 
	\begin{align}
	&\, \begin{multlined}[t] \norm{\ptau^2 \errhn{n+1}}_{\Ltwo}^2
	+
	\norm{\ptau \nabla \errhn{n+1}}_{\Ltwo}^2
		+
	\tau \sum\limits_{j=1}^{n}
	\norm{\Delta_h \errhn{j+1} }^2_{\Ltwo}
 +
	\beta \tau \sum\limits_{j=2}^n
	\norm{\nabla \ptau^2 \errhn{j+1} }^2_{\Ltwo} \\
	+
	\beta \norm{\Delta_h \errhn{n+1} }_{\Ltwo}^2 \end{multlined}
	\\
	&\lesssim
\alpha
\normlinfLtwo{\ptau^2 \errhn{j}}{2}{n+1}^2
+
\tau \sum\limits_{j=1}^n 	\bigl( 
\norm{\ptau^2 \errhn{j+1}  }^2_{\Ltwo}
+
\norm{\nabla \ptau\errhn{j+1}  }^2_{\Ltwo}	\bigr)
\\
&\, 
+	\norm{\ptau^2 \errhn{2}}_{\Ltwo}^2
	+
	\norm{\ptau \nabla \errhn{2}}_{\Ltwo}^2
	+
	\beta \norm{\Delta_h \errhn{1} }_{\Ltwo}^2	
	+
	\tau  \sum\limits_{j=1}^n
	\bigl( 
		\norm{\delta_h^{j+1} }^2_{\Ltwo}
	+
		\norm{\ptau \delta_h^{j+1} }^2_{\Ltwo}
	\bigr)
\end{align}
with the hidden constant independent of $h$, $\tau$, $n$, and $\beta$. From here using
Lemmas~\ref{lem:errhn1},~\ref{Lemma:Bound_defect_Euler},
and~\ref{lem:errhn2}, together with
	\begin{equation} \label{eq:lower_order_bounds}
		\norm{ \errhn{j} }_{\Ltwo} \lesssim 	\norm{ \nabla \errhn{j} }_{\Ltwo},
		\qquad
		\norm{ \nabla \errhn{j} }_{\Ltwo}^2 \leq T  
		\tau \sum\limits_{k=1}^{j}
		\norm{ \nabla \ptau \errhn{j} }^2_{\Ltwo},
	\end{equation}
	due to $\errhn{0} = 0$,
we infer 
for $m=n+1$
	\begin{align}
&\, \begin{multlined}[t] \norm{\ptau^2 \errhn{m}}_{\Ltwo}^2
+
\norm{\ptau \nabla \errhn{m}}_{\Ltwo}^2
+
\tau \sum\limits_{j=1}^{m-1}
\norm{\Delta_h \errhn{j+1} }^2_{\Ltwo}
+
\beta \tau \sum\limits_{j=2}^{m-1}
\norm{\nabla \ptau^2 \errhn{j+1} }^2_{\Ltwo} \\
+
\beta \norm{\Delta_h \errhn{m} }_{\Ltwo}^2 \end{multlined}
\\
&\lesssim
\alpha
\normlinfLtwo{\ptau^2 \errhn{j}}{2}{m}^2
+
\bigl( \tau + h^k \bigr)^2
+
\tau \sum\limits_{j=1}^{m-1} \bigl(
\norm{\ptau^2 \errhn{j+1}  }^2_{\Ltwo}
+
\norm{\nabla \ptau\errhn{j+1}  }^2_{\Ltwo} \bigr)
\end{align}
for any $\alpha>0$.
It is straightforward to prove that
 analogous estimates hold for $m\leq n$.
Therefore, taking the maximum
of this inequality over $m=2,\ldots,n+1$ and choosing $\alpha$ sufficiently small,
together with a Gr\"onwall argument yields the error estimate
stated in \eqref{eq:bounds_intermediate}
with a constant independent of $n$. \\
\indent We then use the bound in \eqref{eq:bounds_intermediate}, which is uniform in $n$,
and the \revisedd{2}{CFL-type condition} in
\eqref{eq:def_CFL_Euler} to obtain estimates in
\eqref{def:finaln}. This step closes the induction argument.
%
%
\end{proof}

The statement of Theorem~\ref{thm:ErrorSemiImplicit} now follows immediately. 

\begin{proof}[Proof of Theorem~\ref{thm:ErrorSemiImplicit}] 
	Using the embedding  in \eqref{eq:discrete_Sobolev_embedding}
	and the best approximation properties of the Ritz projection in \eqref{eq:projRitz_approx},
	we obtain the claimed estimate.
\end{proof}

\subsection{The inviscid limit of the fully discrete solution}
	\label{sec:limit_full_discrete}

We next study the limiting behavior of the fully discrete problem as $\beta \rightarrow 0$ and prove Theorem~\ref{thm:FullyDiscrete_beta_limit}.
Similarly to Section~\ref{sec:limit_space_discrete}, we emphasize the $\beta$ dependence of the fully discrete solution by using the notation $\solhnbeta{n}$ when
 $\beta \in (0, \bar{\beta}]$ and $\solhnzero{n}$ in the inviscid case $\beta=0$.\\
%
\indent We define the quantity  
\begin{equation}
	\solhnbar{n} = \solhnzero{n}  - \solhnbeta{n} ,
\end{equation}
and estimate it to arrive at Theorem~\ref{thm:FullyDiscrete_beta_limit}.\\

\begin{proof}[Proof of Theorem~\ref{thm:FullyDiscrete_beta_limit}]
By subtracting the equation for $\solhnbeta{j+1}$ from the equation for $\solhnzero{j+1}$, we conclude that $\solhnbar{j+1}$ satisfies
\begin{equation} \label{eq:Euler_bar}
\begin{aligned} 
	&\, \ip{ (1+\kappa \ptau \solhnzero{j}) \ptau^2 \solhnbar{j+1} 
	 - 
	 c^2 \Delta_h \solhnbar{j+1}  
	 	}{\varphi_h}
	 \\
		= 
&\,   \ip{  - \kappa \ptau \solhnbar{j} \ptau^2 \solhnbeta{j+1} 
-
\ell \nabla \solhnzero{j} \cdot \nabla \ptau \solhnbar{j+1} 
-
	 \ell \nabla \solhnbar{j} \cdot \nabla \ptau \solhnbeta{j+1} 
	-
	\beta \Delta_h  \ptau \solhnbeta {j+1}
		}{\varphi_h}
\end{aligned}
\end{equation}
for  $j=1,\ldots, n$. We test this problem with $\varphi_h = \ptau \solhnbar{j+1} $, sum from $j=1,\ldots, n$, and use
Lemma~\ref{lem:ftc_Euler} to obtain
\begin{equation} \label{est_fully_discrete_betalim}
\begin{aligned} 
&	\norm{\ptau \solhnbar{n+1} }^2_{\Ltwo}
	+
	\norm{\nabla \solhnbar{n+1} }^2_{\Ltwo} 
	\\
	%
&\lesssim \begin{multlined}[t]
	\tau \sum\limits_{j=1}^{n} \bigl(
	\norm{\ptau \solhnbar{j+1} }^2_{\Ltwo}
	+
	\norm{\nabla \solhnbar{j+1} }^2_{\Ltwo} \bigr)
 + | \ell  \ip{   \nabla \solhnzero{n} \cdot \nabla \ptau \solhnbar{j+1}    }{\ptau \solhnbar{j+1}}  |\\
+ 
| \beta \tau \sum\limits_{j=1}^{n} \ip{     \nabla  \ptau \solhnbeta {j+1}  }{ \nabla \ptau \solhnbar{j+1}}  |, \end{multlined}
\end{aligned}
\end{equation}
where we have also used the uniform bounds on $\solhnbeta{n+1}$ guaranteed by Lemma~\ref{Lemma:UniformEst_solhn}.
We proceed to estimate the right-hand side terms. 
Using the expansion in \eqref{eq:expansion_critical_term}, we estimate
\begin{equation} 
	\begin{aligned}
	&\tau \sum\limits_{j=1}^{n}
	| \ell  \ip{   \nabla \solhnzero{n} \cdot \nabla \ptau \solhnbar{j+1}    }{\ptau \solhnbar{j+1}}  |
	\\
	&
	\lesssim
	\tau \sum\limits_{j=1}^{n} 
	\norm{\ptau  \solhnbar{j+1}}_{\Ltwo}^2 
	+
	\normlinfLtwo{\ptau  \solhnbar{j}}{2}{n+1}
	h^{-1-d/6}
	\tau \sum\limits_{j=1}^{n} 
	\norm{\Delta_h \errhnzero{j}}_{\Ltwo}
	\norm{ \ptau  \solhnbar{j+1}}_{\Ltwo} 
		\\
	&
	\lesssim
	\tau \sum\limits_{j=1}^{n}
	\norm{\ptau  \solhnbar{j+1}}_{\Ltwo}^2 
	+
\alpha_1	\normlinfLtwo{\ptau  \solhnbar{j}}{2}{n+1}
\end{aligned}
\end{equation}
for any $\alpha_1 >0 $. It remains to bound the term involving $\beta$ in \eqref{est_fully_discrete_betalim} to set up a Gr\"onwall argument. 
To this end, we employ the summation by parts formula \eqref{eq:sum_by_parts} to obtain
\begin{equation}
		\begin{aligned}
&\, | \beta \tau \sum\limits_{j=1}^{n}   \ip{     \nabla  \ptau \solhnbeta {j+1}  }{ \nabla \ptau \solhnbar{j+1}} |
\\
=
&\, | \beta \bigl( \ip{     \nabla  \ptau \solhnbeta {n+1}  }{ \nabla  \solhnbar{n+1}}  
-
\ip{     \nabla  \ptau \solhnbeta {1}  }{ \nabla  \solhnbar{1}} 
\bigr)
 -
\beta \tau \sum\limits_{j=1}^{n}  \ip{  \ptau^2   \nabla   \solhnbeta {j+1}  }{ \nabla  \solhnbar{j+1}} |
\\
\lesssim
&\,\begin{multlined}[t] \beta^2
\bigl(
\norm{   \nabla  \ptau \solhnbeta {n+1} }^2_{\Ltwo}
+\norm{   \nabla  \ptau \solhnbeta {1} }^2_{\Ltwo}
+
\tau \sum\limits_{j=1}^{n}  \norm{   \ptau^2   \nabla   \solhnbeta {j+1}  }^2_{\Ltwo}
\bigr)
+
\norm{ \nabla  \solhnbar{1} }^2_{\Ltwo}\\
+
\alpha_2
\norm{ \nabla  \solhnbar{n+1} }^2_{\Ltwo}
+
\tau \sum\limits_{j=1}^{n}  \norm{ \nabla  \solhnbar{j+1} }^2_{\Ltwo} \end{multlined}
\end{aligned}
\end{equation}
for any $\alpha_2>0$, where we have also relied on the $C_0$ bounds given in \eqref{def:finaln}. Similarly to the reasoning in Section~\ref{sec:FE_analysis},  by Lemma~\ref{Lemma:UniformEst_solhn} we have uniform bounds for the first two terms multiplied with $\beta^2$ on the right-hand side, and we can
proceed similarly to \eqref{eq:nabla_pt2_solhbeta} to bound the third term.
Further, by \eqref{eq:def_Euler_inits}, it holds
\begin{align}
\norm{ \nabla  \solhnbar{1} }_{\Ltwo} = \frac{\beta \tau^2}{2} \norm{ \projRitz ( 1 +  \kappa \soltinit )^{-1}  \Delta \soltinit  }_{\Ltwo} 
\lesssim  C (\norm{\sol}_{{\calU}}) \cdot \tau^2 \beta
\end{align}
and we can conclude by reducing $\alpha_2$
that
\begin{equation} \label{eq:recursion_bar_n}
	\begin{aligned}
	&\norm{\ptau \solhnbar{n+1} }^2_{\Ltwo}
	+
	\norm{\nabla \solhnbar{n+1} }^2_{\Ltwo} \\
	&\leq 
	C
	\beta^2
	+
	C \alpha_1 
	\normlinfLtwo{ \ptau  \solhnbar{j}}{2}{n+1}^2
	+ C
	\tau \sum\limits_{j=1}^{n} \bigl(
	\norm{\ptau \solhnbar{j}}^2_{\Ltwo} + \norm{\nabla \solhnbar{j}}^2_{\Ltwo} \bigr). 
	\end{aligned}
	\end{equation}
We now take the maximum of this inequality over $n=2,\ldots, N+1$, and for small enough $\alpha_1$ apply a Gr\"onwall argument to 
obtain 
\begin{equation}
\max_{n=1,\ldots, N+1}	\| \ptau \solhnbeta{n} - \ptau \solhnzero{n} \|_{\Ltwo}
+
\max_{n=1,\ldots, N+1} \|\nabla (\solhnbeta{n} - \solhnzero{n})\|_{\Ltwo} \leq C \beta,
\end{equation}
as claimed.
\end{proof}

We see that, under the assumptions of Theorem~\ref{thm:ErrorSemiImplicit}, also the fully discrete problem preserves the asymptotic behavior of the exact and semi-discrete solutions as $\beta \rightarrow 0$.

\section{Non-robust estimates for linear finite elements}
\label{sec:non_robust}

In this final section, we extend the results presented in Section~\ref{sec:Main_results}
to the case of linear finite elements, i.e., $k=1$. 
We can qualitatively prove the same error bounds with constants that do not depend on the damping parameter $\beta>0$,
as long as we couple the discretization parameters with the damping parameter correctly.

\subsection{Semi-discretization}
\label{subsec:non_robust_semi}

We first consider the error bound for the semi discretization in space,
and state a variant of Theorem~\ref{thm:kuznetsov_space_main} that takes $\beta >0 $ into account. We first state our theorem, and devote the rest of this section to its proof. Since several arguments are unchanged compared to Section~\ref{sec:FE_analysis}, we only present the key estimates here. 

\begin{mytheorem}[Non-robust finite element estimates] \label{thm:kuznetsov_space_main_beta_dep} 
	Let the assumptions of Theorem~\ref{thm:kuznetsov_space_main}
	hold, but replace the assumptions on $k$ and $\beta$ with
	$k \geq 1$ and $\beta > 0$,
	satisfying the relation
	\begin{equation} \label{eq:beta_CFL_space}
		h^{k - d/6 - 2\varepsilon} \leq C_1 \sqrt{\beta},
	\end{equation}
	for some $C_1$, $\varepsilon>0$ which are independent of $h$ and $\beta$.
	Then there exists $h_0>0$ and a constant $C>0$,  independent of $h$ and $\beta$, such that for all $h\leq h_0$, the following error bound holds:
	\begin{equation} \label{FE_final_est_nonrobust}
		\begin{aligned}
			\begin{multlined}[t]	\norm{\pt^2 \sol(t) - \pt^2 \solh(t) }^2_{\Ltwo}
				+  
				\norm{\nabla \pt \sol(t) - \nabla \pt \solh(t)}^2_{\Ltwo}  
			\\
				+
			\int_0^t \norm{\nabla \sol(s) - \nabla \solh(s)}^2_{L^6(\Omega)} \ds
				\leq C 
				h^{2 k } \end{multlined}
		\end{aligned}
	\end{equation}
	for all $t \in [0,T]$.
\end{mytheorem}

The key idea of the proof remains the same as before, and hence
analogously to Section~\ref{sec:FE_analysis}, we work on the time interval $[0,\finalthbeta]$
with
\begin{equation} \label{def:finaltimeh_beta}
	\begin{aligned}
		\finalthbeta \coloneqq \sup  \Big \{  t \in (0,T] \mid
		\	&\text{a unique solution }	\solh \in H^{3}(0,t; V_h) \text{ of \eqref{eq:Kuznetsov_space_discr_full_eq}  exists, and} \\
		%
		&\, \beta^{-1/2} h^{-d/6-\varepsilon} \norm{\pt^2 \errh(s)}_{\Ltwo} \leq C_0,  
		\\
		&\, \beta^{-1/2} h^{-d/6-\varepsilon} \norm{\nabla \pt \errh(s)}_{\Ltwo}
		\leq C_0 ,
		\\
		&\, h^{-d/6-\varepsilon} \norm{\Delta_h \errh(s)}_{\Ltwo}
		\leq C_0 \ \text{ for all } s \in [0,t]  \Big  \} ,
	\end{aligned}
\end{equation}
for some fixed $C_0>0$ and $\varepsilon$ as in \eqref{eq:beta_CFL_space}.
We then conduct the error analysis on this interval, with the aim of later extending $\finalthbeta$ to $T$, analogously to before. 
By arguing as in Lemma~\ref{lemma:est_errh_zero}, we can prove that $\finalthbeta>0$, as well as obtain the correct estimates for $\errh(0)$, $\pt \errh(0)$, and $\pt^2 \errh(0)$. We omit those details here.

\begin{lemma} \label{lem:a-priori_bounds_on_solh_beta} 
	Let the assumptions of Theorem~\ref{thm:kuznetsov_space_main_beta_dep}  hold.
	Then, we have
	%
	\begin{equation} \label{gamma_bound_beta}
		1 + \kappa \pt \solh \geq 
		\quasilowerbound > 0,  \quad (x,t) \in \Omega \times [0, \finalth],
	\end{equation}
	where  $\gamma$ does not depend on $h$, $\beta$, or $\finalthbeta$.
\end{lemma}

\begin{proof}
	Using the stability properties of the Ritz projection stated in \eqref{eq:projRitz_approx},
	we obtain
	\begin{align}
		\norm{\pt \solh(t)}_{\Linf} 
		&\lesssim
		\norm{\pt \sol(t)}_{\Linf} 
		+
		\norm{(\Id -\projRitz) \pt \sol(t)}_{\Linf}  + h^{-d/6} \norm{\nabla \pt \errh(t)}_{\Ltwo} 
		\\
		&\leq 
		\norm{\pt \sol(t)}_{\Linf} + C h^k +  C \beta^{1/2} h^{\varepsilon} 
	\end{align}
	for all $t \in [0, \finalth]$. Hence we have the uniform lower bound in \eqref{gamma_bound_beta} that guarantees non-degeneracy as well as uniform boundedness of $\norm{\pt \solh}_{L^\infty(\Linf)} $.
\end{proof}
~\\[2mm]
We are now ready to derive the relevant estimates and prove the error bound in \eqref{FE_final_est_nonrobust}. Before, we briefly comment on the changes compared to Section~\ref{sec:FE_analysis}.

By the bounds in Propositions~\ref{prop:fe_FirstEstimate}
and~\ref{prop:kuznetsov_space_main_v2}, one obtains
\begin{equation} \label{beta_bound}
\beta  \int_0^t   \norm{\nabla \pt^2 \errh(s) }^2_{\Ltwo} \ds
+
\beta \norm{\Delta_h \errh(t) }^2_{\Ltwo} 
\leq C h^{2 k}.
\end{equation}
However, in order to stay uniform in $\beta$, these bounds were not exploited in the analysis of Section~\ref{sec:FE_analysis}. If we consider now fixed $\beta>0$, this enables us to employ \eqref{beta_bound}
while paying with inverse powers of $\beta$ via Young's inequality. The coupling condition in
\eqref{def:finaltimeh_beta} allows us to close the proof even for $k =1$.
More details can be found in the following proof.\\

\begin{proof}[Proof of Theorem~\ref{thm:kuznetsov_space_main_beta_dep}]
Below whenever the temporal argument is skipped, we assume that the given (in)equality holds for all $t \in [0, \finalthbeta]$. We proceed in two steps. \vspace*{2mm}
	
(a) Starting from the estimate in \eqref{ineq:error_test_pt2_errh}
and	using the uniform lower bound in \eqref{gamma_bound_beta}, we obtain
	\begin{equation}
		\begin{aligned}
			&	\pt \norm{(1+\kappa \pt \solh)^{1/2}\pt^2 \errh }^2_{\Ltwo}
			+ c^2  \pt  \norm{\nabla \pt \errh}^2_{\Ltwo}
			+ \beta    \norm{\nabla \pt^2 \errh}^2_{\Ltwo}
			\\
			&\lesssim \begin{multlined}[t]
			\|\pt^2 \solh \pt^2 \errh \|^2_{\Ltwo}
			+
			\ell |	\ip{  \nabla \pt \solh \cdot \nabla \pt \errh}{\pt^2 \errh}|
			+
			\ell | \ip{\nabla \solh \cdot \nabla \pt^2 \errh  }{\pt^2 \errh}|
			\\
 +
			\norm{\pt^2 \errh}^2_{\Ltwo}
			+    \norm{\nabla \pt \errh}^2_{\Ltwo}
			+
			\norm{\pt \delta_h }^2_{\Ltwo} \end{multlined}
		\end{aligned}
	\end{equation}
and have to treat the terms involving $\solh$ separately.
	Now exploiting $\beta>0$, we estimate
	\begin{align}
		\norm{\pt^2 \solh \pt^2 \errh }^2_{\Ltwo}
		&\lesssim
		\norm{\pt^2 \errh \pt^2 \errh }^2_{\Ltwo}
		+
		\norm{\pt^2 \projRitz \sol \, \pt^2 \errh }^2_{\Ltwo}
		\\
		&\lesssim
		\norm{ \pt^2 \errh }^2_{\Lthree}
		\norm{\pt^2 \errh }^2_{\Lsix}
		+
		\norm{\pt^2 \errh }^2_{\Ltwo}
		\\
		&\lesssim
		\bigl(\beta^{-1} h^{-d/3} \norm{ \pt^2 \errh }^2_{\Ltwo} \bigr)
		\, \beta \norm{\pt^2 \nabla \errh }^2_{\Ltwo}
		+
		\norm{\pt^2 \errh }^2_{\Ltwo}
		\\
		&\leq
		\frac{\beta}{4} \norm{\pt^2 \nabla \errh }^2_{\Ltwo}
		+
		C \norm{\pt^2 \errh }^2_{\Ltwo},
	\end{align}
	where we have used the $C_0$ bounds in \eqref{def:finaltimeh_beta} in the last step.
	Next, we estimate
	\begin{align}
		& | \ip{ \nabla \pt \solh \cdot \nabla \pt \errh}{ \pt^2 \errh } |
		\\
		%
		&\lesssim
		\norm{ \nabla \pt \errh}_{\Ltwo}^2
		\norm{\pt^2 \errh}_{\Linf} 
		+
		\norm{\nabla \pt \projRitz}_{\Linf}
		\norm{ \nabla \pt \errh}_{\Ltwo}
		\norm{\pt^2 \errh}_{\Ltwo}
		\\
		&\lesssim
		\beta^{-1/2}   \norm{ \nabla \pt \errh}_{\Ltwo}^2
		\beta^{1/2} h^{-d/6} \norm{\nabla \pt^2 \errh}_{\Ltwo} 
		+
		\norm{ \nabla \pt \errh}_{\Ltwo}
		\norm{\pt^2 \errh}_{\Ltwo}.
	\end{align}
	We further estimate the first term in the last line above by using the $C_0$ bounds in \eqref{def:finaltimeh_beta} for $h \leq h_0$:
	\begin{align}
		&\, \beta^{-1/2} h^{-d/6}  \norm{ \nabla \pt \errh}_{\Ltwo}^2
		\beta^{1/2} \norm{\nabla \pt^2 \errh}_{\Ltwo} 
		\\
		= &\,  \bigl(\beta^{-1/2} h^{-d/6}  \norm{\nabla \pt \errh}_{\Ltwo} \bigr)
		\beta^{1/2} \norm{\nabla \pt^2 \errh}_{\Ltwo} 
		\norm{ \nabla \pt \errh}_{\Ltwo}
		\\
		\leq
		&\,  \frac{\beta}{4} \norm{\nabla \pt^2 \errh}_{\Ltwo}^2
		+
		C 
		\norm{ \nabla \pt \errh}_{\Ltwo}^2 .
	\end{align}
	%
	%
	Next, proceeding as in
	\eqref{eq:nonlinear_term_v1} results in
	\begin{align}
		| \ip{\nabla \solh \cdot \nabla \pt^2 \errh  }{\pt^2 \errh}|
		&\lesssim
		\norm{\pt^2 \errh}_{\Ltwo}^2
		+
		\ip{\nabla \errh \cdot \nabla \pt^2 \errh   }{\pt^2 \errh} ,
	\end{align}
	and further with the discrete embedding \eqref{eq:discrete_Sobolev_embedding} we have
	\begin{align}
		|\ip{\nabla \errh \cdot \nabla \pt^2 \errh   }{\pt^2 \errh}|
		&\leq 
		\norm{\nabla \errh}_{\Linf}
		\norm{\nabla \pt^2 \errh}_{\Ltwo}
		\norm{\pt^2 \errh}_{\Ltwo}
		\\
		&
		\lesssim
		\norm{\nabla \errh}_{\Lsix} 
		\beta^{1/2}\norm{\nabla \pt^2 \errh}_{\Ltwo}
		\bigl(\beta^{-1/2} h^{-d/6} \norm{\pt^2 \errh}_{\Ltwo} \bigr)
		\\
&\leq
C C_0^2 h^{2 \varepsilon} \norm{\Deltah \errh}_{\Ltwo}^2 
+
\frac{\beta}{4} \norm{\nabla \pt^2 \errh}_{\Ltwo}^2,
	\end{align}
	%
	where we have used
	the $C_0$ bounds in
	 \eqref{def:finaltimeh_beta}.
	 For $h \leq h_0$,
	 we have thus again derived \eqref{est_eh_Step1} with constants independent of $h$ and $\beta$. \vspace*{2mm}

	(b) Testing the error equation \eqref{eq_errh} with $\phih = -  \Delta  \errh$ yields with Young's inequality
	\begin{align} 
		&c^2 \norm{ \Deltah \errh}_{\Ltwo}^2
		+
		\beta 
		\pt \norm{ \Deltah \errh}_{\Ltwo}^2\\
		%
		%
		&\leq
		\frac{c^2}{4}
		\norm{ \Deltah \errh}_{\Ltwo}^2
		\\
		&\quad +
		C\left(\norm{ \pt \errh }_{\Ltwo}^2
		+
		\norm{ \pt^2 \errh }_{\Ltwo}^2
		+
		\norm{ \delta_h}_{\Ltwo}^2
		+
		\ip{		\ell \nabla \solh \cdot \nabla \pt \errh }{\Deltah \errh}\right);
	\end{align}
cf.\ \eqref{ineq_Deltah_errh}. The last term is here estimated via 
	\begin{align}
		&|\ip{\ell \nabla \solh \cdot \nabla \pt \errh }{\Deltah \errh} |
		\\
		&\leq
		|\ip{\ell \nabla \errh \cdot \nabla \pt \errh }{\Deltah \errh} |
		+
		|\ip{\ell \nabla \projRitz \sol \cdot \nabla \pt \errh }{\Deltah \errh} |
		\\
		&\lesssim
		\bigl( h^{-d/6}
		\norm{ \Deltah \errh}_{\Ltwo} \bigr)  
		\norm{ \nabla \pt \errh }_{\Ltwo} 	\norm{ \Deltah \errh}_{\Ltwo}
		+
		\norm{ \nabla \pt \errh }_{\Ltwo}^2	 +	
		\frac{c^2}{4} \norm{ \Deltah \errh}_{\Ltwo}^2
		\\
		&\lesssim
		\norm{ \nabla \pt \errh }_{\Ltwo}^2	 +
				\frac{c^2}{2}	\norm{ \Deltah \errh}_{\Ltwo}^2,
	\end{align}
	where we have relied in the last step on the last $C_0$ bound in \eqref{def:finaltimeh_beta}.
	We
absorb the $\norm{ \Deltah \errh}_{\Ltwo}^2$ terms and 
	 conclude as before by Gr\"onwall's inequality that
	\begin{equation} 
\begin{aligned}
\begin{multlined}[t]	\norm{\pt^2 \errh(t) }^2_{\Ltwo}
		+  
		\norm{\nabla \pt \errh(t)}^2_{\Ltwo}  
		+
		\beta
		\norm{ \Deltah \errh(t)}_{\Ltwo}^2 + \beta  \int_0^t   \norm{\nabla \pt^2 \errh(s) }^2_{\Ltwo} \ds
\\
		+
		\int_0^t \norm{ \Deltah \errh(s)}_{\Ltwo}^2 \ds
		\leq C 
		h^{2k}  \end{multlined}
		\end{aligned}
	\end{equation}
on $[0, \finalthbeta]$. Thanks to this uniform bound, we can reason as in Section~\ref{sec:FE_analysis} to close again the arguments with \eqref{eq:beta_CFL_space} and obtain $\finalthbeta = T$. 
\end{proof}

\subsection{Non-robust estimates for a full discretization}

Our last main result for the full discretization is a variant of 
Theorem~\ref{thm:ErrorSemiImplicit} in the case of fixed $\beta>0$.
The strategy of the proof is similar to the one on Section~\ref{sec:Semi_implicit}.
In order to compensate the inverse powers of $\beta$, we have to assume the following coupling: 
\begin{equation} \label{eq:CFL_type_beta}
%
	\tau
	\leq C_1 \beta^{1/2} h^{d/6 + 2\varepsilon} ,
%
	\qquad
	h^{ k - d/6 - 2 \varepsilon} \leq C_1 \beta^{1/2} \, ,
\end{equation}
for constants $C_1$, $\varepsilon>0$ which are independent of $h$, $\tau$, and $\beta$.

\begin{mytheorem}[Non-robust fully discrete error bounds] \label{thm:ErrorSemiImplicit_beta}	
		Let the assumptions of Theorem~\ref{thm:ErrorSemiImplicit}
	hold, but replace the conditions on $k$ and $\beta$ with
	$k \geq 1$ and $0 < \beta \leq \bar{\beta}$.
	Under the coupling conditions
	\eqref{eq:CFL_type_beta},
	for $h \leq h_0$ and $\tau \leq \tau_0$,
	it holds
	\begin{equation} \label{eq:thm_ErrorSemiImplicit_estimate_nonrobust}
		\begin{aligned}
			\norm{\pt^2 \sol(\tn{n}) - \ptau^2 \solhn{n}}^2_{\Ltwo}
			&+  
			\norm{\nabla \pt \sol(\tn{n}) - \nabla \ptau \solhn{n} }^2_{\Ltwo}  
			\\
			&+
			\tau \sum_{j=1}^n
			\norm{\nabla \sol(\tn{n}) - \nabla \solhn{n} }^2_{L^6(\Omega)}  
			\leq C  \bigl( \tau + h^k \bigr)^2,
		\end{aligned}
	\end{equation}
	where the constant $C>0$ is independent of $h$, $\tau$, and $\beta$.
\end{mytheorem}

In order to prove the result, we set up an induction argument as before, and show that for $n = 2, \ldots, N+1$ the solution $\solhn{n}$ exists, and similarly to
\eqref{eq:bounds_intermediate}, it holds
	\begin{equation}  \label{eq:bounds_intermediate_beta}
	\begin{aligned}
	\begin{multlined}[t]	\norm{ \ptau^2 \errhn{n}}^2_{\Ltwo}
		+  
		\norm{\nabla   \ptau  \errhn{n} }^2_{\Ltwo}  
		+
		\beta \norm{\Delta_h \errhn{n+1} }^2_{\Ltwo} 
		+
		\tau \sum_{j=1}^{n}
		\norm{ \Delta_h  \errhn{j} }^2_{L^6(\Omega)}  
		\leq C  \bigl( \tau + h^k \bigr)^2, \end{multlined}
	\end{aligned}
\end{equation}
as well as, analogously to \eqref{def:finaln},
\begin{equation}  \label{def:finaln_beta}
	\begin{aligned}
		&\, \beta^{-1/2} h^{-d/6 - \varepsilon}
		 \norm{\ptau^2 \errhn{n}}_{\Ltwo} \leq C_0,  
\\[1mm]
&\,\beta^{-1/2} h^{-d/6 - \varepsilon} \norm{\nabla \ptau \errhn{n}}_{\Ltwo}  
\leq C_0, 
\\[1mm]
&\, h^{-d/6 - \varepsilon} \norm{\Delta_h \errhn{n}}_{\Ltwo} 
\leq C_0, 
\\[1mm]
&\, \beta^{-1/2} h^{-d/6 - \varepsilon}\bigl( \tau \sum_{j=1}^n \norm{ \Delta_h \errhn{j}}_{\Ltwo}^2 \bigr)^{1/2} 
\leq C_0  ,
\end{aligned}
\end{equation}
with some constants $C$, $C_0 >0$ that are independent of $h$, $\tau$, $n$, and $\beta$ and $\varepsilon$ chosen as in \eqref{eq:CFL_type_beta}.

\begin{lemma}
	Under the assumptions of Theorem~\ref{thm:ErrorSemiImplicit_beta},
	the assertions of
	Lemma~\ref{lem:errhn1},
	Lemma~\ref{lem:existence_Euler_approximation},
	Lemma~\ref{Lemma:Bound_defect_Euler},
	and
	Lemma~\ref{lem:errhn2} hold true, and in particular 
	\eqref{eq:bounds_intermediate_beta} and
	\eqref{def:finaln_beta}
	hold for $n=2$.
\end{lemma}

\begin{proof}
	The bounds in Lemmas~\ref{lem:errhn1} and \ref{lem:errhn2} 
	directly follow from the conditions in \eqref{eq:CFL_type_beta}.
	For the existence statement in Lemma~\ref{lem:existence_Euler_approximation},
	we estimate the term in \eqref{eq:expansion_critical_term} now via
	\begin{equation} 
		\begin{aligned} 
			 \tau | \ip{ \nabla \solhn{n} \cdot \nabla \varphi_h }{ \varphi_h }  |
			\lesssim 
			&\,	\tau h^{-d/6} \norm{ \Delta_h \errhn{n} }_{\Ltwo}
			\norm{ \nabla \varphi_h }_{\Ltwo}
			\norm{ \varphi_h }_{\Ltwo}
			+
			\tau \norm{ \varphi_h }_{\Ltwo}^2 
			\\
			\leq
			&\,\alpha \bigl(
			\norm{ \varphi_h }_{\Ltwo}^2 + \tau^2   \norm{ \nabla \varphi_h }_{\Ltwo}^2
			\bigr)
		\end{aligned}
	\end{equation}
	for any $\alpha >0$, where we have used \eqref{def:finaln_beta} in the last step.
\end{proof}

Further, we have the crucial result in the leading nonlinear term
which prevents degeneracy of the problem also for $k=1$.

\begin{lemma} \label{lem:a-priori_bounds_on_solhn_beta} 
	Let the assumptions of Theorem~\ref{thm:kuznetsov_space_main_beta_dep}  hold.
	If the estimates 
	\eqref{eq:bounds_intermediate_beta} 
	and
	\eqref{def:finaln_beta}
	hold up to $n\geq 2$, then
	 we have 
	\begin{equation} \label{gamma_bound_Euler_beta}
		1 + \kappa \ptau \solhn{j} \geq 
		\quasilowerbound > 0,  \quad j = 2,\ldots,n,
	\end{equation}
	where  $\gamma$ does not depend on $h$, $\tau$ $\beta$, or $n$.
\end{lemma}

\begin{proof}
Along the lines of Lemma~\ref{lem:a-priori_bounds_on_solh_beta}, we have 
	\begin{align}
		\norm{\ptau \solhn{j}}_{\Linf} 
		&\lesssim
		\norm{\ptau \soln{j}}_{\Linf} 
		+
		\norm{(\Id -\projRitz) \ptau \soln{j}}_{\Linf}
		+ h^{-d/6} \norm{\pt \errh}_{\Lsix} 
		\\
		&\leq 
		\norm{\pt \sol}_{L^{\infty}(\Linf)} + C h^k +  C \beta^{1/2} h^{\varepsilon} 
	\end{align}
	and hence the lower bound in \eqref{gamma_bound_Euler_beta} follows as well as boundedness of $\norm{\ptau \solh}_{\Linf} $.	
\end{proof}

With this result, we can prove the principal result on a non-robust fully discrete bound.
As already explained in Section~\ref{subsec:non_robust_semi},
the appearance of the inverse powers of $\beta$ comes in by exploiting the following bounds from
Proposition~\ref{prop:Deltah_errh_Euler}
and \ref{prop:pt2_errh_Euler}:
\begin{equation}
	\beta \norm{\Delta_h \errhn{n+1} }^2_{\Ltwo} 
	+
	\tau \beta \sum\limits_{j=2}^n
	 \norm{\nabla \ptau^2 \errhn{j+1} }^2_{\Ltwo}
	\leq C \bigl( \tau +h^k \bigr)^2 ,
\end{equation}
and applying the relations in \eqref{eq:CFL_type_beta}.

\begin{proof} [Proof of Theorem~\ref{thm:ErrorSemiImplicit_beta}] 
	We conduct the proof in two testing steps. \midskip
	(a) 
	We proceed as in Proposition~\ref{prop:pt2_errh_Euler} 
	and
	test the differentiated error equation \eqref{eq:error_Euler_ptau} with $\varphi_h = \ptau^2 \errhn{j+1}$ to obtain
	\begin{align} 
		& \ip{(1+\kappa \ptau \solhn{j}) \ptau^3 \errhn{j+1}  }{\ptau^2 \errhn{j+1}}
		+
		c^2 \ip{\nabla \ptau \errhn{j+1} }{\nabla \ptau^2 \errhn{j+1} }    
		+	
		\beta \norm{\nabla \ptau^2 \errhn{j+1} }_{\Ltwo}^2
		\\
		\leq
		\kappa 
		& | \ip{ \ptau^2 \solhn{j}  \ptau^2 \errhn{j}}{\ptau^2 \errhn{j+1}}  |
		+ \ell 
		| \ip{\nabla \solhn{j} \cdot \nabla \ptau^2\errhn{j+1}}{\ptau^2 \errhn{j+1}} |
		\\
		&\qquad +\ell | \ip{   \nabla \ptau \solhn{j} \cdot \nabla \ptau \errhn{j}}{\ptau^2 \errhn{j+1}} |
		+ \norm{ \ptau^2 \errhn{j+1}}_{\Ltwo}^2
		+
		\norm{ \ptau\delta_h^{j+1} }_{\Ltwo}^2 ,
	\end{align}		
	and estimate the three terms separately.
	We use the $C_0$ bounds in
	\eqref{def:finaln_beta} to conclude
	\begin{align}
		&\, 		\kappa \ip{ \ptau^2 \solhn{j}  \ptau^2 \errhn{j}}{\ptau^2 \errhn{j+1}} 
		\\
		\lesssim
		&\, \bigl( \beta^{-1/2}  h^{-d/6}
		\norm{ \ptau^2 \errhn{j}}_{\Ltwo} \bigr)
		\norm{ \ptau^2 \errhn{j}}_{\Ltwo}
		\beta^{1/2} \norm{ \nabla \ptau^2 \errhn{j+1}}_{\Ltwo}
		+
		\norm{ \ptau^2 \errhn{j}}^2_{\Ltwo}
		+
		\norm{ \ptau^2 \errhn{j+1}}^2_{\Ltwo}
		\\
		\leq 
		&\, \frac{\beta}{4}\norm{ \nabla \ptau^2 \errhn{j+1}}^2_{\Ltwo}
		+
		C \norm{ \ptau^2 \errhn{j}}^2_{\Ltwo}
		+
		C \norm{ \ptau^2 \errhn{j+1}}^2_{\Ltwo} ,
	\end{align}
	and absorb the $\beta$ term by the left-hand side $\beta$ term.
	We sum from $j=2,\ldots,n$ and obtain
	by the expansion in  \eqref{eq:expansion_critical_term}
	%
	\begin{align}
		&\, \tau \sum\limits_{j=2}^n| \ell 
		\ip{\nabla \solhn{j} \cdot \nabla \ptau^2\errhn{j+1}}{\ptau^2 \errhn{j+1}}  |
		\\
		\lesssim
		&\, \begin{multlined}[t] h^{-d/6} \beta^{-1/2 }
		\bigl(  \tau  \sum\limits_{j=2}^n  \norm{\Delta_h \errhn{j}}_{\Ltwo}^2
		\bigr)^{1/2}
		\bigl( 
		\tau \beta  \sum\limits_{j=2}^n 
		\norm{\nabla \ptau^2 \errhn{j+1}}_{\Ltwo}^2
		\bigr)^{1/2}
		\normlinfLtwo{\ptau^2 \errhn{j}}{2}{n+1}
		\\
		+
		\tau  \sum\limits_{j=2}^n 
		\norm{ \ptau^2 \errhn{j}}_{\Ltwo}^2 \end{multlined}
		\\
		\leq
		&\,
		\alpha_1
		\normlinfLtwo{\ptau^2 \errhn{j}}{2}{n+1}
		+
		 \frac{\beta}{4}
		\tau  \sum\limits_{j=2}^n 
		\norm{ \nabla \ptau^2 \errhn{j+1}}_{\Ltwo}^2
		+
		C \tau  \sum\limits_{j=2}^n \norm{\ptau^2 \errhn{j+1}}_{\Ltwo}^2
	\end{align}
	for any $\alpha_1>0$ by the bounds in \eqref{def:finaln_beta}
	for $h\leq h_0$ and $\tau \leq \tau_0$.
	Finally, we estimate
	\begin{align}
		| \ell \ip{   \nabla \ptau \solhn{j} \cdot \nabla \ptau \errhn{j}}{\ptau^2 \errhn{j+1}} |
		&\leq
		| \ell \ip{   \nabla \ptau \errhn{j} \cdot \nabla \ptau \errhn{j}}{\ptau^2 \errhn{j+1}} |
		\\
		&\quad +
		\norm{\nabla \ptau \errhn{j+1}}_{\Ltwo}^2
		+
		\norm{\ptau^2 \errhn{j+1}}_{\Ltwo}^2
	\end{align}
	and with this, by the $C_0$ bounds in \eqref{def:finaln_beta},
	\begin{align}
		| \ell \ip{   \nabla \ptau \errhn{j} \cdot \nabla \ptau \errhn{j}}{\ptau^2 \errhn{j+1}} |
		&\lesssim
		\norm{\nabla \ptau \errhn{j}}_{\Ltwo}
		\norm{\nabla \ptau \errhn{j}}_{\Ltwo}
		h^{-d/6} \norm{\nabla \ptau^2 \errhn{j+1}}_{\Ltwo}
		\\
		&\leq  
		\frac{\beta}{4}  \norm{\nabla \ptau^2 \errhn{j+1}}_{\Ltwo}^2
		+
		C \norm{\nabla \ptau \errhn{j}}_{\Ltwo}^2 .
	\end{align}
	Lemma~\ref{lem:ftc_Euler} then yields \eqref{eq:est_pt2_errh_Euler}. \\
	
 	(b)	We then test
	\eqref{eq:error_Euler} with $\phih = - \Deltah \errhn{j+1}$ to obtain
	\begin{align}
		&\, c^2 \norm{\Delta_h \errhn{j+1} }^2_{\Ltwo}
		+
		\beta \ip{\Delta_h  \ptau \errhn{j+1} }{\Delta_h \errhn{j+1} }
		\\
		&=
		\ip{ (1+\kappa \ptau \solhn{j}) \ptau^2 \errhn{j+1}  
		}{\Delta_h \errhn{j+1} }
		+ \ip{  \ell \nabla \solhn{j} \cdot \nabla \ptau\errhn{j+1} }
		{\Delta_h \errhn{j+1} }
		\\
		&\quad + \ip{\delta_h^{j+1}}{\Delta_h \errhn{j+1} }
		\\
		&\leq \begin{multlined}[t]
		\frac{c^2}{4} 
			\norm{\Delta_h \errhn{j+1} }^2_{\Ltwo}
		+
	C \bigl(	\norm{\nabla \ptau \errhn{j+1} }^2_{\Ltwo}
		+
		\norm{ \ptau^2 \errhn{j+1} }^2_{\Ltwo}
		+
		\norm{\delta_h^{j+1} }^2_{\Ltwo}
	  \\
	  +
		| \ip{  \ell \nabla \errhn{j} \cdot \nabla \ptau\errhn{j+1} }{\Delta_h \errhn{j+1} } | \bigr). \end{multlined}
	\end{align}
	For the last term, we use the $C_0$ bound in \eqref{def:finaln_beta} on $\norm{\Delta_h \errhn{j} }_{\Ltwo}$ to conclude that
	\begin{align}
		| \ip{  \ell \nabla \errhn{j} \cdot \nabla \ptau\errhn{j+1} }{\Delta_h \errhn{j+1} } |
		&\lesssim
		h^{-d/6}	\norm{\Delta_h \errhn{j} }_{\Ltwo}
		\norm{\nabla \ptau \errhn{j} }_{\Ltwo}
		\norm{\Delta_h \errhn{j+1} }_{\Ltwo}
		\\
		&\leq 
		C \norm{\nabla \ptau \errhn{j} }^2_{\Ltwo}
		+
		\frac{c^2}{4} 
		\norm{\Delta_h \errhn{j+1} }^2_{\Ltwo} .
	\end{align}
	We can absorb the $\norm{ \Deltah  \errhn{j+1}}_{\Ltwo}^2$ terms by the left-hand side
	and reason as in the proof of Proposition~\ref{prop:intermediate_step}
	to arrive at \eqref{eq:bounds_intermediate_beta}.
The coupling in \eqref{eq:CFL_type_beta} then implies
\eqref{def:finaln_beta}, and the claim follows by induction.
\end{proof}

\begin{appendices}
%
%

\section{Estimates for discrete derivatives}
\label{sec:appendix}

In order to keep the presentation self-contained, we include the proof of Lemma~\ref{lem:diff_ptau_pt} here in the appendix. \\

\begin{proof}[Proof of Lemma~\ref{lem:diff_ptau_pt}]
	For the sake of readability, we just consider a generic norm a
	and assume without loss of generality 
	$\ell_1 \leq \ell_2$ to obtain
	\begin{align}
		\norm{\ptau^{m - \ell_1} \pt^{\ell_1} \soln{n} - \ptau^{m - \ell_2} \pt^{\ell_2} \soln{n}}
		=
		\norm{
			\ptau^{m - \ell_2} \bigl(  \ptau^{\ell_2 - \ell_1} \pt^{\ell_1}\soln{n} -  \pt^{\ell_2} \soln{n} \bigr)}.
	\end{align}
	Applying the fundamental theorem of calculus $\ell$-times gives
	\begin{equation}
		\ptau^\ell \soln{n} = \frac{1}{\tau^\ell} \int_0^\tau \ldots \int_0^\tau 
		\pt^{\ell} \sol(\tn{n-\ell} + \sigma_1 + \ldots + \sigma_\ell ) \dint{\sigma_1} \ldots  \dint{\sigma_\ell} \,,
	\end{equation}
	and similarly
	\begin{align}
		\ptau \soln{k} - \pt \soln{k} &= 
		- \tau 
		\int_0^{1} \int_{s}^1
		\pt^2 \sol(\tn{k-1}+ \tau \eta )
		\dint{\eta} \ds.
	\end{align}
	With the estimate
	\begin{align}
		\bigl| \int_0^1 \ldots \int_0^1 
		\sol(\sigma_1 + \ldots + \sigma_\ell )
		\dint{\sigma_{1}} \ldots 	\dint{\sigma_{\ell}}
		\bigr|
		&\leq
		\int_{0}^{\ell}  | \sol(\sigma_1   ) | \dint{\sigma_1},
	\end{align}
	we may write
	\begin{align}
		&\, \ptau^{m-\ell_2}\bigl(	\ptau^{\ell_2 - \ell_1} \pt^{\ell_1} \soln{n} -  \pt^{\ell_2} \soln{n}
		\bigr)
		=
		\sum_{j=\ell_1}^{\ell_2 - 1 }  (\ptau - \pt) 
		\bigl( \ptau^{m - 1 - j} \pt^{j}  \soln{n} \bigr)
		\\
		&= 
		- \tau \sum_{j=\ell_1}^{\ell_2 - 1 } 
		\int_0^1 \ldots \int_0^1 
		\int_0^{1} \int_s^{1}
		\pt^{m+1} \sol(\tn{n-(m-j)} 
		+
		\tau ( \eta
		+ \sigma_1 + \ldots + \sigma_{m-1-j} ) )   
		\dint{\eta} \ds
		\dint{\sigma_1} \ldots  \dint{\sigma_{m-1-j}}
	\end{align}
	and hence
	\begin{align}
		&\, \norm{ \ptau^{m-\ell_2}\bigl(	\ptau^{\ell_2 - \ell_1} \pt^{\ell_1} \soln{n} -  \pt^{\ell_2} \soln{n}
			\bigr) }
		\\
		&\leq
		\tau \sum_{j=\ell_1}^{\ell_2 - 1 } 
		\int_0^1 \ldots \int_0^1 
		\norm{
			\pt^{m+1} \sol(\tn{n-(m-j)} 
			+
			\tau ( \eta
			+ \sigma_1 + \ldots + \sigma_{m-1-j} ) )   
		}
		\dint{\eta} 
		\dint{\sigma_1} \ldots  \dint{\sigma_{m-1-j}}
		\\
		&\leq 
		\tau \sum_{j=\ell_1}^{\ell_2 - 1 } 
		\int_0^{m-j}
		\norm{
			\pt^{m+1} \sol(\tn{n-(m-j)} 	+	\tau  \eta	)   
		}
		\dint{\eta} 
		\\
		&\leq 
		\tau m
		\int_0^{m}
		\norm{
			\pt^{m+1} \sol(\tn{n-m} +	\tau  \eta	)   
		}
		\dint{\eta} 
	\end{align}
	and with this
	\begin{align}
		&\, \norm{ \ptau^{m-\ell_2}\bigl(	\ptau^{\ell_2 - \ell_1} \pt^{\ell_1} \soln{n} -  \pt^{\ell_2} \soln{n}
			\bigr) }^2
		\leq 
		\tau m^3
		\int_{\tn{n-m}}^{\tn{n}}
		\norm{
			\pt^{m+1} \sol( s	)   
		}^2
		\ds,
	\end{align}
	which concludes the proof.
\end{proof}

\end{appendices}

\bmhead{Acknowledgements}

The first author is funded by the Deutsche Forschungsgemeinschaft (DFG, German Research Foundation) -- Project-ID 258734477 -- SFB 1173.



\bibliography{references}


\end{document}